\documentclass[11pt]{amsart}
\usepackage{extarrows}
\usepackage{latexsym}

\usepackage{multirow}
\usepackage{makecell, colortbl}
\usepackage{booktabs, makecell, tabularx}
\usepackage{lscape}
\newcommand*\circled[1]{\tikz[baseline=(char.base)]{
            \node[shape=circle,dashed,draw,inner sep=2pt] (char) {#1};}}
\usepackage{wrapfig}
\usepackage{multicol}
\usepackage{mathdots}

\usepackage{amssymb,amsmath,amscd}
\usepackage[pdftex]{graphicx}
\usepackage{enumerate}
\usepackage[dvipsnames]{xcolor}
\usepackage{tikz}
\usetikzlibrary{patterns}
\usepackage[lmargin=1in, rmargin=1in, tmargin=0.95in, bmargin=0.95in]{geometry}
\usepackage{listings}
\usepackage{courier}
\usepackage{tikz-cd, wrapfig}
\usetikzlibrary{decorations.pathreplacing}
\usepackage{color}
\usepackage{MnSymbol}
\usepackage{hyperref}
\usepackage{mathrsfs, colonequals}

\newcommand{\pp}{\mathbb{P}}

\newcommand{\zz}{\mathbb{Z}}

\newcommand{\na}[3]{\boldsymbol{a}_{#2,#3}^{#1}} 
\newcommand{\naa}[2]{\boldsymbol{a}_{#2}^{#1}} 
\newcommand{\nb}[3]{\boldsymbol{b}_{#2,#3}^{#1}} 
\newcommand{\nbb}[2]{\boldsymbol{b}_{#2}^{#1}} 
\newcommand{\x}[3]{x_{#1,#2}(#3)} 
\newcommand{\s}{s} 

\newcommand{\C}{\mathcal{C}}
\newcommand{\E}{\mathcal{E}}
\newcommand{\F}{\mathcal{F}}

\renewcommand{\H}{\mathcal{H}}

\renewcommand{\L}{\mathcal{L}}
\newcommand{\M}{\mathcal{M}}
\renewcommand{\O}{\mathcal{O}}
\renewcommand{\P}{\mathcal{P}}
\newcommand{\Q}{\mathcal{Q}}
\newcommand{\T}{\mathcal{T}}

\newcommand{\N}{\mathcal{N}}

\newcommand{\X}{\mathcal{X}}

\newcommand{\W}{\mathcal{W}}

\newcommand{\Res}{\text{Res}}

\newcommand{\pr}{\operatorname{pr}}
\newcommand{\Id}{\operatorname{Id}}

\newcommand{\ch}{\operatorname{ch}}
\newcommand{\Spec}{\operatorname{Spec}}

\newcommand{\Pic}{\operatorname{Pic}}

\renewcommand{\Im}{\operatorname{Im}}
\newcommand{\coker}{\operatorname{coker}}

\newcommand{\codim}{\operatorname{codim}}

\newcommand{\rk}{\operatorname{rk}}
\newcommand{\ev}{\mathrm{ev}}

\renewcommand{\bar}{\overline}
\newcommand{\Gr}{\operatorname{Gr}}
\newcommand{\Fr}{\operatorname{Fr}}

\newcommand{\fixme}[1]{}
\newcommand{\isabel}[1]{}
\newcommand{\hannah}[1]{}
\newcommand{\eric}[1]{}

\newcommand{\defi}[1]{\textsf{#1}} 
\usetikzlibrary{decorations.pathreplacing}

\newtheorem{thm}{Theorem}[section]

\newtheorem{lem}[thm]{Lemma}

\newtheorem{prop}[thm]{Proposition}

\newtheorem{cor}[thm]{Corollary}

\theoremstyle{definition}
\newtheorem{defin}[thm]{Definition}

\newtheorem{example}[thm]{Example}

\newtheorem*{question*}{Question}
\newtheorem{question}[thm]{Question}

\theoremstyle{remark}
\newtheorem{rem}{Remark}

\title{Global Brill--Noether Theory over the Hurwitz Space}
\author{Eric Larson, Hannah Larson, and Isabel Vogt}
\thanks{During the preparation of this article E.L.\ and I.V.\ were supported by
NSF MSPRF grants DMS-1802908 and DMS-1902743 respectively and H.L.\ was supported by the Hertz Foundation and NSF GRFP under grant DGE-1656518.}

\DeclareMathVersion{normal2}

\newlength{\myindent}
\begin{document}
\maketitle

\begin{abstract}
Let $C$ be a curve of genus $g$. A fundamental problem in the theory of
algebraic curves is to understand maps $C \to \pp^r$ of specified degree $d$.
When $C$ is general, the moduli space of such maps is well-understood by the
main theorems of Brill--Noether theory. 
Despite much study over the past three decades, a similarly complete
picture has proved elusive for curves of fixed gonality.
Here we complete such a picture, by proving analogs of all of the main theorems of Brill--Noether theory
in this setting.
As a corollary, we prove a conjecture of Eisenbud and Schreyer regarding versal deformation spaces of vector bundles on $\pp^1$.
\end{abstract}

\setcounter{tocdepth}{1}
\tableofcontents
\setcounter{tocdepth}{2}

\section{Introduction}

The notion of a (complex) algebraic curve without reference to an embedding
in projective space was developed in the 19th century. Ever since, a fundamental problem in algebraic geometry ---
whose study goes back at least to Riemann in 1851 \cite{riemann} -- has been:

\begin{question*}
Given an algebraic curve $C$, what is the geometry of the space of maps $C \to \pp^r$ of given degree $d$?
\end{question*}

The data of such a map is equivalent to a
line bundle $\mathcal{L}$ on $C$ of degree $d$, equipped with an $(r + 1)$-dimensional
basepoint-free space of sections $V \subseteq H^0(C, \mathcal{L})$.
A central object of study is therefore the \defi{Brill--Noether locus $W^r_d(C)$}
defined by
\[W^r_d(C) \colonequals \{\text{line bundles $\mathcal{L}$ on $C$ with $h^0(C, \mathcal{L}) \geq r + 1$}\} \subseteq \Pic^d(C).\]

When $C$ is a \emph{general} curve of genus $g$,
the fundamental results of Brill--Noether theory from the 1970s and 1980s give a good description of the geometry of $W^r_d(C)$.
Namely, $W^r_d(C)$ is\ldots
\begin{enumerate}
\item \label{expdim} Of the expected dimension $\rho = g - (r + 1)(g + r - d)$
and nonempty if and only if $\rho \geq 0$.
(Griffiths and Harris in 1980 \cite{bn})
\item \label{smooth} Normal and Cohen--Macaulay, and is smooth away from $W^{r + 1}_d(C)$. (Geiseker in 1982 \cite{gp})
\item \label{class} Of class
\[[W^r_d(C)] = \prod_{\alpha = 0}^r \frac{\alpha !}{(g - d + r + \alpha)!} \cdot \theta^{(r + 1)(g - d + r)}.\]
(Independently by Kempf in 1971 \cite{kempf}, and by Kleiman and Laksov in 1972 \cite{kl})
\item \label{irred} Irreducible if $\rho > 0$. (Fulton and Lazarsfeld in 1981 \cite{fl})
\item \label{monodromy} Reducible if $\rho = 0$, by \eqref{class},
except if $(d, r) = (0, 0)$ or $(d, r) = (2g - 2, g - 1)$.
Nonetheless, when $\rho \geq 0$, the \emph{universal} $\W^r_d$ has a \emph{unique} irreducible component dominating
the moduli space of curves. (Eisenbud and Harris in 1987 \cite{im})
\end{enumerate}

However, in nature, curves $C$ are often encountered already equipped with a map $C \to \pp^{r_0}$.
It is thus natural to ask how the presence of a \emph{given} map $C \to \pp^{r_0}$ --- which may force $C$ to not be general ---
affects the moduli spaces of \emph{other} maps $C \to \pp^r$.
The simplest case of this problem is when $r_0 = 1$, i.e.\ when $C$ is general among curves of fixed gonality $k$.
Unsurprisingly, therefore, the following question has received much attention from the 1990s to the present day:

\begin{question} \label{question-hurwitz}
Given a general degree $k$ genus $g$ cover $f \colon C \to \pp^1$, what is the geometry of $W^r_d(C)$?
\end{question}

For $k = 2$ and $3$, classical results answer this question (in fact for \emph{every} curve of genus $g$):
the case of hyperelliptic curves is a famous result of Clifford \cite{Clifford} from 1878;
the case of trigonal curves was answered by Maroni \cite{maroni} in 1946 (for further interpretations see also \cite{MS, trig}).
Partial progress has been made  when $k=4$ (by Coppens--Martens \cite{CM00} in 2000), and when $k=5$ (by Park \cite{park} in 2002).
Upper bounds on the dimension of $W^r_d(C)$ were given for odd $k$ by Martens \cite{M96} in 1996 and for arbitrary $k$ by Ballico--Keem \cite{BK} in 1996.
Moreover, for arbitrary $k$, the dimension of all components of $W^1_d(C)$ were determined by Coppens--Keem--Martens (in 1994 \cite{CKM}). Later, Coppens--Martens showed that $W^r_d(C)$ has a component of the expected dimension $\rho$ when $d - g < r \leq k-2$ (in 1999 \cite{CM99}), and that  $W^r_d(C)$ has components of the ``wrong dimension" $\rho(g, \alpha - 1, d) - (r - \alpha + 1)k$ for $\alpha$ dividing $r$ or $r + 1$  (in 2002 \cite{CM02}).
In 2016, Pflueger \cite{Pf} proved that a maximum over formulas of this type provide an upper bound
\begin{equation} \label{pfl}
\dim W^r_d(C) \leq \rho_k(g, r, d) \colonequals \max_{\ell \in \{0, \ldots, r'\}} \rho(g, r - \ell, d) - \ell k,
\end{equation}
where $r' \colonequals \min\{r, g-d+r -1\}$. The value where the above maximum is attained need not satisfy the divisibility conditions of Coppens--Martens. Nevertheless, in 2017 Jensen--Ranganathan \cite{JR} proved that equality holds in \eqref{pfl}, determining the dimension of the largest component.

In 2019, H.\ Larson \cite{refinedBN} and Cook-Powell--Jensen \cite{CPJ} independently proposed that these multiple components of varying dimensions are explained by splitting loci.
Indeed, if $f \colon C \to \pp^1$ is a $k$-gonal curve,
the condition $h^0(C, \mathcal{L}) \geq r + 1$ is equivalent to $h^0(\pp^1, f_* \mathcal{L}) \geq r + 1$.
The Brill--Noether locus $W^r_d(C)$ therefore splits into a union of \defi{Brill--Noether splitting loci}
$W^{\vec{e}}(C)$
corresponding to the possible splitting types $\vec{e}$
of the pushforward.
Namely if $\vec{e} = (e_1, \ldots, e_k)$ is a splitting type,
then write $\O(\vec{e}) \colonequals \O_{\pp^1}(e_1) \oplus \cdots \oplus \O_{\pp^1}(e_k)$, and define:
\[W^{\vec{e}} (C) = \{\text{line bundles $\mathcal{L}$ on $C$ with $f_* \mathcal{L} \simeq \O(\vec{e})$ or a specialization thereof}\} \subseteq \Pic^d(C).\]
In this language, we have
\[W^r_d(C) = \bigcup_{h^0(\O(\vec{e})) \geq r + 1} W^{\vec{e}}(C).\]

It thus natural to instead ask whether the Brill--Noether splitting loci $W^{\vec{e}}(C)$
satisfy analogs of \eqref{expdim}--\eqref{monodromy}.
An analog of \eqref{expdim} is known,
and there has been progress towards analogs of \eqref{smooth} and \eqref{class}. Namely,
it is known that $W^{\vec{e}}(C)$ is\ldots
\begin{enumerate}
\item[(1$'$)] Of the expected dimension $\rho' \colonequals g - u(\vec{e})$,
where $u(\vec{e}) \colonequals h^1(\operatorname{End}(\O(\vec{e}))) = \sum_{e_i < e_j} e_j - e_i - 1$,
and nonempty if and only if $\rho' \geq 0$. \\
(Independently by H.\ Larson in 2019 \cite{refinedBN}, and by Cook-Powell--Jensen in 2019/2020 \cite{CPJ, CPJ2}.)
\item[(2$_\circ'$)] Smooth away from the union of $W^{\vec{e}'}(C)$, for $\vec{e}'$ a specialization of $\vec{e}$. (H.\ Larson \cite{refinedBN})
\item[(3$_\circ'$)] Of class
\[[W^{\vec{e}}(C)] = \frac{N(\vec{e})}{u(\vec{e})!} \cdot \theta^{u(\vec{e})},\]
for some unknown integer $N(\vec{e})$ depending on $\vec{e}$ but not on $g$.
(H.\ Larson \cite{refinedBN})
\end{enumerate}

With the exception of (3$_\circ'$) above --- which follows from a structure theorem on the Chow ring of
the moduli stack of vector bundles on $\pp^1$ obtained in \cite{P1} ---
the principal tool in the study of $W^{\vec{e}}(C)$ thus far has been degeneration: Given
an element of $W^{\vec{e}}(C)$ on a general smooth $k$-gonal curve,
one can study the limiting behavior as the curve $C$ is specialized.
The central difficulty with this approach has been the lack of a ``regeneration theorem'':
Given the sort of object that looks like it might be a limit, we had
no way of showing that it was indeed a limit.
Thus, ``local'' information about the loci $W^{\vec{e}}(C)$
(e.g.\ smoothness, dimension, etc.)\ was accessible via degeneration, but ``global'' information
(irreducibility, class, etc.)\ remained inaccessible.

The central innovation of the present paper is to establish such a regeneration theorem,
thus enabling a degenerative study of \emph{global} information about $W^{\vec{e}}(C)$.
As a consequence, we obtain the following results, which provide a good answer to
Question~\ref{question-hurwitz}:

\begin{thm} \label{thm:main}
Suppose that the characteristic of the ground field is zero, or greater than $k$.
Let $f \colon C \to \pp^1$ be a general degree~$k$ cover of genus~$g$,
and let $\vec{e}$ be any splitting type.
\begin{enumerate}
\item[(2$'$)] \label{ncm} $W^{\vec{e}}(C)$ is normal and Cohen--Macaulay, and is smooth away from the union of the splitting loci
$W^{\vec{e}'}(C) \subset W^{\vec{e}}(C)$ having codimension $2$ or more. 
\item[(3$'$)] \label{NeTag} The integers $N(\vec{e})$ can be described in terms of
a well-studied problem in the theory of Coxeter groups (see Theorem~\ref{thm:enum} below
for a more precise statement).
\item[(4$'$)] $W^{\vec{e}}(C)$ is irreducible when $\rho' > 0$.
\item[(5$'$)] When $\rho' \geq 0$, the universal $\W^{\vec{e}}$ has a \emph{unique} component dominating
the Hurwitz space $\H_{k,g}$ of degree $k$ genus $g$ covers of $\pp^1$.
\end{enumerate}
See Remark \ref{rem:char} for more details on the characteristic assumptions.
\end{thm}

It turns out that part (2$'$) of Theorem~\ref{thm:main} implies a conjecture of Eisenbud and Schreyer regarding the equations of splitting loci on versal deformation spaces. Suppose $\vec{e}' \leq \vec{e}$; let $\F$ on $\pp^1 \times \mathrm{Def}(\O(\vec{e}'))$ be the versal deformation of $\O(\vec{e}')$. 
The subscheme $\Sigma_{\vec{e}} \subseteq \mathrm{Def}(\O(\vec{e}'))$, defined by the Fitting support for $\rk R^1 \pi_* \F(m) \geq h^1(\pp^1, \O(\vec{e})(m))$, is clearly supported on the splitting locus for splitting type $\vec{e}$ or worse. Eisenbud and Schreyer conjecture that $\Sigma_{\vec{e}}$ is reduced (Conjecture 5.1 \cite{ES}).

\begin{cor} \label{cor:ESconj}
The splitting locus $\Sigma_{\vec{e}}$ is normal and Cohen--Macaulay (and hence reduced).
\end{cor}
\begin{proof}
Let $f\colon C \rightarrow \pp^1$ be a general cover of genus $g \geq u(\vec{e}')$ and let $L \in W^{\vec{e}'}(C)$. By \cite{refinedBN}, the induced map from $\Pic^d(C)$ near $L$ to $\mathrm{Def}(f_*L) = \mathrm{Def}(\O(\vec{e}'))$ is smooth. Thus, the fact that $W^{\vec{e}}(C)$ (whose scheme structure shall be defined by the appropriate Fitting supports) is normal and Cohen--Macaulay implies $\Sigma_{\vec{e}}$ is normal and Cohen--Macaulay.
\end{proof}

To further explain (3$'$), let $W$ be a Coxeter group with generating set $S$,
and let $w \in W$ be an element. Define
\[R(w) \colonequals \text{number of reduced words for $(W, S)$ equal to $w$}.\]

Determination of the integers $R(w)$ is a well-studied problem in combinatorics,
starting with Stanley's computation of $R(w)$ for Coxeter groups of type $A$
(i.e.\ the symmetric groups), and his proposal for a systematic study of $R(w)$
for other Coxeter groups, in 1984 \cite{stanley}.
This problem has since been solved completely for other finite Coxeter groups ---
including of type B by Haiman in 1992 \cite{hai}, and of type D by
Billey and Haiman in 1995 \cite{bh} --- and partial progress has been made
for some infinite Coxeter groups by Eriksson, Fan, and Stembridge in a series of papers from the late 1990s
\cite{eriksson, fan, fanstem, stem1, stem2, stem3}.

Of particular relevance to us are the Coxeter systems of type $\tilde{A}$, known as \emph{affine
symmetric groups}. Explicitly, these are groups generated by elements
$s_j$ with $j \in \zz/k\zz$, subject to relations
\[s_j^2 = 1, \quad s_j s_{j'} = s_{j'} s_j \ \text{if $j - j' \neq \pm 1$}, \quad \text{and} \quad (s_j s_{j + 1})^3 = 1.\]
Alternatively, elements of the affine symmetric group can be realized as permutations $f\colon \zz \to \zz$
such that
\[f(x + k) = f(x) + k \quad \text{and} \quad \sum_{x = 1}^k f(x) = \sum_{x = 1}^k x = \frac{k(k + 1)}{2};\]
here $s_j$ corresponds to the simple transposition defined by
\[f(x) = \begin{cases}
x + 1 & \text{if $x \equiv j \mod k$;} \\
x - 1 & \text{if $x \equiv j + 1 \mod k$;} \\
x & \text{otherwise.}
\end{cases}\]
For the affine symmetric group, Eriksson \cite{eriksson}
gave recursive formulas for $R(w)$,
and showed that for fixed $k$ the generating function for $R(w)$ is rational.

We relate the components of the Brill--Noether splitting locus on the central fiber to reduced words in the affine symmetric group.
As a consequence of our regeneration theorem,
the count of points (when $\rho' = 0$) on the general fiber is equal to the count on the central fiber. Therefore we obtain:

\begin{thm} \label{thm:enum}
Given a splitting type $\vec{e}$, define $w(\vec{e})$ to be the affine symmetric group element
that sends (for $1 \leq \ell \leq k$):
\[\ell \mapsto \chi(\O(\vec{e})(-e_{k + 1 - \ell})) - \#\{\ell' : e_{\ell'} \geq e_{k + 1 -\ell}\} + \# \{\ell' : \ell' \geq k + 1 - \ell \ \text{and} \ e_{\ell'} = e_{k+1-\ell}\}.\]
Then
\[N(\vec{e}) = R(w(\vec{e})).\]
\end{thm}

In particular, the integers $N(\vec{e})$ grow rapidly, and may be easily computed
in any desired case using Eriksson's recursions mentioned above.
For example, $N(2, 7, 18, 18, 28, 28)$ is the integer
\[25867977167969459670048709047628541850991022718608668059259099938720 \approx 2.6 \cdot 10^{67}.\]
One can also check that the description of $N(\vec{e})$ in Theorem~\ref{thm:enum} agrees with the conjectural
value of $N(\vec{e})$ proposed by Cook-Powel--Jensen, and hence proves Conjecture~1.6 of \cite{CPJ2}.

\subsection{Overview of Techniques}
The degeneration we will use is to a chain of elliptic curves, as described in Section~\ref{our_degen}.
In Section~\ref{sec:lim-line}, we identify the sorts of objects that look like they might be a limit
of line bundles in $W^{\vec{e}}(C)$; we call these \defi{$\vec{e}$-positive limit line bundles}.

\vspace{1.5pt}

\noindent
\begin{minipage}{.4\textwidth}
\hspace{\myindent}This locus of $\vec{e}$-positive limit line bundles has an intricate combinatorial structure: In Section \ref{sec:epos} we show that its components are in bijection with certain fillings of a certain Young diagram $\Gamma(\vec{e})$.  In Section \ref{sec:combin}, we relate these fillings to the reduced word problem for the affine symmetric group.  As a preview, for example, the splitting type $\vec{e} = (-2, 0, 0, 2)$ corresponds to the Young diagram to the right.  When $g=u(\vec{e}) =7$, there are six $\vec{e}$-positive
limit line bundles on the central fiber,
corresponding to six fillings, one of which is shown to the right.
\end{minipage}
\begin{minipage}{.55\textwidth}
\hspace{5pt}
\begin{tikzpicture}[scale=.5]

\draw[thick] (0,0) -- (5, 0);
\draw[thick] (2,-1) -- (5,-1);
\draw[thick] (1,-2) -- (2,-2);
\draw[thick] (0,-5) -- (1,-5);
\draw[thick] (0,0) -- (0,-5);
\draw[thick] (1,-2) -- (1,-5);
\draw[thick] (2,-1) -- (2,-2);
\draw[thick] (5,0) -- (5,-1);

{\small
\draw[<->] (6.75, -2.5+.75) -- (10.25, -2.5+.75);
\node at (11.75, -2.5+.75) {$w(\vec{e}) \colonequals $};
\node at (14.73, -1-.25+.75) {$1 \mapsto -4$};
\node at (14.5, -1.75-.25+.75) {$2 \mapsto 2$};
\node at (14.5, -2.5-.25+.75) {$3 \mapsto 3$};
\node at (14.5, -3.25-.25+.75) {$4 \mapsto 9$};
\draw [decorate,decoration={brace,amplitude=3pt, mirror}] (13.5, -1+.75) --(13.5, -4+.75);
}

\node at (8.5, -2.2+.75) {{\tiny corresponds}};
\node at (8.5, -2.7+.75) {{\tiny to}};

{\tiny \color{gray!70!black}
\draw [<->] (-.4, 0) -- (-.4, -5);
\node[rotate=90] at (-.85, -2.5) {$h^1(\O_{\pp^1}(\vec{e})(-2))$};

\draw [<->] (5.4, 0) -- (5.4, -1);
\node at (7, -.5) {$h^1(\O_{\pp^1}(\vec{e}))$};

\draw [<->] (0, .4) -- (5, .4);
\node at (2.5, .8) {$h^0(\O_{\pp^1}(\vec{e}))$};

\draw [<->] (0, -1.75) -- (2, -1.75);
\node at (3.1, -2.5) {$h^0(\O_{\pp^1}(\vec{e})(-1))$};

}
\end{tikzpicture}

\vspace{10pt}

\begin{tikzpicture}[scale=.5]
\hspace{5pt}
{\small
\node[rotate=90] at (-.85, -2.5) {{\color{white}$h^1(\O_{\pp^1}(\vec{e})(-2))$}};

\draw (0,0) -- (5, 0);
\draw (0,-1) -- (5,-1);
\draw (0,-2) -- (2,-2);
\draw (0,-3) -- (1,-3);
\draw (0,-4) -- (1,-4);
\draw (0,-5) -- (1,-5);
\draw (0,0) -- (0,-5);
\draw (1,0) -- (1,-5);
\draw (2,0) -- (2,-2);
\draw (3,0) -- (3,-1);
\draw (4,0) -- (4,-1);
\draw (5,0) -- (5,-1);
\draw (0.5,-0.5) node{1};
\draw (1.5,-0.5) node{3};
\draw (2.5,-0.5) node{4};
\draw (3.5,-0.5) node{6};
\draw (4.5,-0.5) node{7};
\draw (0.5,-1.5) node{2};
\draw (1.5,-1.5) node{7};
\draw (0.5,-2.5) node{4};
\draw (0.5,-3.5) node{5};
\draw (0.5,-4.5) node{7};

\node at (14, -2.5+.75) {$w(\vec{e}) = s_4s_3s_1s_2s_1s_3s_4$};
}

\draw[<->] (6.75, -2.5+.75) -- (10.25, -2.5+.75);
\node at (8.5, -2.2+.75) {{\tiny corresponds}};
\node at (8.5, -2.7+.75) {{\tiny to}};

\end{tikzpicture}
\end{minipage}

\vspace{2pt}

We then prove our regeneration theorem,
which is the heart of the paper since it provides the bridge between the combinatorics
of the central fiber and the geometry of the general fiber.
Because the components of $W^r_d$ have the ``wrong'' dimension,
naively applying the techniques used by Eisenbud and Harris to prove their regeneration theorem in \cite{lls}
necessarily produces too many equations.
Our key insight is that the combinatorial structure coming from the affine symmetric group
forces the limit linear series associated to a \emph{general}
$\vec{e}$-positive limit line bundle to ``break up'' into minimally-interacting
pieces that can be regenerated almost independently.
This allows us to avoid overcounting equations, and prove a regeneration theorem in Section~\ref{regeneration-st}.
However, this ``breaking up'' happens a priori only set-theoretically.
We then upgrade this to a scheme-theoretic regeneration theorem in Section~\ref{sec:red_CM} by showing that the
locus of $\vec{e}$-positive limit line bundles on the central fiber
is reduced.

Having established the regeneration theorem, we then deduce the fundamental global
geometric properties of Brill--Noether splitting loci in Sections~\ref{sec:conn}--\ref{sec:mon}.

\begin{rem}[A note on our ground field]\label{rem:char}
Since the conclusion of Theorem~\ref{thm:main} is geometric, we suppose for the remainder of the paper that our ground
field $K$ is algebraically closed.

The assumption that the characteristic of $K$ is zero or greater than $k$ is used only to guarantee the irreducibility
of $\H_{k,g}$ (as proved by Fulton in \cite{fulton}),
and hence to be able to state Theorem~\ref{thm:main} in terms of
a ``general" degree $k$ cover.
However, in any characteristic, the conclusions of Theorem \ref{thm:main} parts (2$'$)--(4$'$)
hold for \emph{some} component of $\H_{k,g}$.  In particular, Corollary~\ref{cor:ESconj} requires no hypotheses on the characteristic.
Moreover, in any characteristic not \emph{dividing} $k$, the conclusion of Theorem \ref{thm:main} part (5$'$)
holds for some component of $\H_{k,g}$.

The paper is organized so that characteristic assumptions are made as late as possible.  All of Sections \ref{our_degen} -- \ref{sec:norm} make no assumptions on the characteristic of the ground field.  Section \ref{sec:mon} assumes that the ground field has characteristic not dividing $k$.
\end{rem}

\begin{rem}[A note on Hurwitz spaces]
Our arguments show the a priori stronger statement
that there exists a smooth degree $k$ cover $f \colon C \to \pp^1$
\emph{with two points of total ramification} satisfying (2$'$)--(4$'$).
Moreover, in (5$'$), the Hurwitz space can be replaced with a component of the
stack $\H_{k, g, 2}$ parameterizing
degree $k$ genus $g$ covers of $\pp^1$ with two marked points of total ramification
(see Definition~\ref{def-h2kg}).
\end{rem}

\subsection*{Acknowledgements}

We would especially like to thank Kaelin Cook-Powell and Dave Jensen for suggesting the importance of $k$-staircase tableaux in this problem during a visit of the second author to the University of Kentucky in 2019. 
We would also like to thank Dan Abramovich, Renzo Cavalieri, Izzet Coskun, David Eisenbud, Joe Harris, Aaron Landesman, Yoav Len, Andrew Obus, Geoffrey Smith, and Ravi Vakil for helpful conversations and comments on an earlier version of this manuscript.

\section{Our Degeneration}\label{our_degen}

We will prove Theorem~\ref{thm:main} via degeneration to a chain
$X = E^1 \cup_{p^1} E^2 \cup_{p^2} \cdots \cup_{p^{g - 1}} E^g$
of $g$ elliptic curves:
\begin{center}
\begin{tikzpicture}
\draw (0, 0) .. controls (1, -1) and (2, -1) .. (3, 0);
\draw (2, 0) .. controls (3, -1) and (4, -1) .. (5, 0);
\draw (4, 0) .. controls (5, -1) and (6, -1) .. (7, 0);
\filldraw (7.8, -0.5) circle[radius=0.02];
\filldraw (7.5, -0.5) circle[radius=0.02];
\filldraw (7.2, -0.5) circle[radius=0.02];
\draw (8, 0) .. controls (9, -1) and (10, -1) .. (11, 0);
\draw (10, 0) .. controls (11, -1) and (12, -1) .. (13, 0);
\filldraw (0.5, -0.42) circle[radius=0.03];
\filldraw (2.5, -0.42) circle[radius=0.03];
\filldraw (4.5, -0.42) circle[radius=0.03];
\filldraw (10.5, -0.42) circle[radius=0.03];
\filldraw (12.5, -0.42) circle[radius=0.03];
\draw (0.5, -0.75) node{$p^0$};
\draw (2.5, -0.75) node{$p^1$};
\draw (4.5, -0.75) node{$p^2$};
\draw (10.5, -0.8) node{$p^{g - 1}$};
\draw (12.5, -0.8) node{$p^g$};
\draw (1.5, -0.95) node{$E^1$};
\draw (3.5, -0.95) node{$E^2$};
\draw (5.5, -0.95) node{$E^3$};
\draw (9.5, -0.95) node{$E^{g - 1}$};
\draw (11.5, -0.95) node{$E^g$};
\end{tikzpicture}
\end{center}
Let $f^i \colon E^i \to \pp^1$ be degree $k$ maps.
Pasting these maps together, we get a map $f \colon X \to P$, where $P$ denotes a chain of $g$
rational curves, attached at points $q^i = f(p^i)$:
\begin{center}
\begin{tikzpicture}
\draw[->] (7.5, 1.5) -- (7.5, 0.25);
\draw (0, 0) .. controls (1, -1) and (2, -1) .. (3, 0);
\draw (2, 0) .. controls (3, -1) and (4, -1) .. (5, 0);
\draw (4, 0) .. controls (5, -1) and (6, -1) .. (7, 0);
\filldraw (7.8, -0.5) circle[radius=0.02];
\filldraw (7.5, -0.5) circle[radius=0.02];
\filldraw (7.2, -0.5) circle[radius=0.02];
\draw (8, 0) .. controls (9, -1) and (10, -1) .. (11, 0);
\draw (10, 0) .. controls (11, -1) and (12, -1) .. (13, 0);
\filldraw (0.5, -0.42) circle[radius=0.03];
\filldraw (2.5, -0.42) circle[radius=0.03];
\filldraw (4.5, -0.42) circle[radius=0.03];
\filldraw (10.5, -0.42) circle[radius=0.03];
\filldraw (12.5, -0.42) circle[radius=0.03];
\draw (0.5, -0.75) node{$q^0$};
\draw (2.5, -0.75) node{$q^1$};
\draw (4.5, -0.75) node{$q^2$};
\draw (10.5, -0.8) node{$q^{g - 1}$};
\draw (12.5, -0.8) node{$q^g$};
\draw (1.5, -0.95) node{$\pp^1$};
\draw (3.5, -0.95) node{$\pp^1$};
\draw (5.5, -0.95) node{$\pp^1$};
\draw (9.5, -0.95) node{$\pp^1$};
\draw (11.5, -0.95) node{$\pp^1$};
\end{tikzpicture}
\end{center}

\emph{If all the $f^i$ are totally ramified at $p^{i-1}$ and $p^i$}, then the theory of admissible covers
implies that $f$ is a limit of smooth $k$-gonal curves.
(The theory of admissible covers was developed by Harris and Mumford
in characteristic zero \cite{admis}; see also Section~5 of~\cite{liu}
for a characteristic-independent proof of this fact.)
In other words,
there is a map $\mathfrak{f} \colon \mathcal{X} \to \mathcal{P}$
between families of curves of genus $g$ and $0$ respectively
over the base $B = \Spec K[[t]]$,
such that the general fiber of $\mathfrak{f}$ is a smooth $k$-gonal curve
and the special fiber of $\mathfrak{f}$ is $f$.
Moreover, we may suppose that the total space $\mathcal{X}$ is smooth,
that $\mathcal{P} \to B$ is the base-change of a family
$\mathcal{P}_0 \to B_0$ with smooth total space via a map $\beta \colon B \to B_0$,
and that $\mathfrak{f}$ is totally ramified along sections $\mathfrak{p}^0$ and $\mathfrak{p}^g$
of $\C \to B$
whose special fibers are $p^0$ and $p^g$ respectively.

A map $f^i \colon E^i \to \pp^1$, of degree $k$ totally ramified at $p^{i - 1}$ and $p^i$, exists if and only if $p^i - p^{i - 1} \in \Pic E^i$ is $k$-torsion.
To keep things as generic as possible,
we therefore suppose for the remainder of the paper that $p^i - p^{i - 1}$
has order \emph{exactly} $k$ in $\Pic E^i$.

\begin{rem}[A note on ``general'' degree $k$ covers] \label{rem:general}
By a \defi{general} degree $k$ cover, we mean one in a component of $\H_{k, g}$ containing the above deformation of $X$.
When the characteristic of the ground field is zero or greater than $k$,
then $\H_{k,g}$ is irreducible \cite{fulton}, so such a component is the entire Hurwitz space.
\end{rem}

\section{Limits of Line Bundles} \label{sec:lim-line}

In this section, let $\mathfrak{f} \colon \C \to \P \to B$ be a family of degree $k$ genus $g$ covers,
over a smooth irreducible base $B$, which is smooth over the generic point $B^*$,
and has smooth total space $\C$.
(Prior to Section~\ref{sec:mon}, the only case of interest will be when $B$ is the spectrum of a DVR.)
We suppose that \emph{all} fibers (including over non-closed points) of $\C \to B$
are \defi{chain curves}, i.e.\ of the form $C^1 \cup_{p^1} \cup \cdots \cup_{p^{n-1}} C^n$,
with all $C^i$ smooth.  (The integer $n$ will depend on which fiber we consider.)
Equivalently, all \emph{geometric} fibers of $\C \to B$ are chain curves,
and these chain curves can be oriented (i.e.\ the two ends can be distinguished)
in a way which is consistent over $B$. This second condition holds, in particular,
if $\C \to B$ has a section whose value at any geometric point
$C^1 \cup_{p^1} \cup \cdots \cup_{p^{n-1}} C^n$ is supported in $C^1 \smallsetminus \{p^1\}$
(which allows us to consistently pick which end of the chain is ``left'' and ``right'').

Similarly, we suppose that all fibers of $\P \to B$ are chain curves with all components $P^i \simeq \pp^1$,
and that the map $\mathfrak{f} \colon \C \to \P$ respects this structure.
Finally, we suppose that for each fiber the maps $f^i \colon C^i \to P^i$ are totally ramified at the nodes
$p^{i - 1}$ and $p^i$ (note that this condition is vacuous if $C$ is smooth).

Note that such covers include our degeneration $\X \to \P \to B$ from the previous section as the special case
where $B$ is the spectrum of a DVR and all $C^i$ have genus~$1$.  Similarly, this includes $\P_0 \xrightarrow{\sim} \P_0 \to B_0$
as the special case
where all $C^i$ have genus~$0$.

In this section, we address the following two fundamental questions:
\begin{enumerate}
\item  \label{q-limit}
Suppose $\mathcal{L}^*$ is a line bundle of degree $d$
on the generic fiber $\mathcal{C}^* = \mathcal{C} \times_B B^*$.
What data do we obtain on a special fiber over $b \in B$?
\item 
If $\mathfrak{f}_* \mathcal{L}^*$ has splitting type $\vec{e}$,
what conditions must this data on a special fiber satisfy?
\end{enumerate}
These questions are local on $B$. Shrinking $B$ if necessary, we may suppose that every component
of the singular locus $\Delta$ of $\mathfrak{f}$ meets the fiber over $b \in B$.
In other words, writing
\[C = \C \times_B b = C^1 \cup_{p^1} \cup \cdots \cup_{p^{n-1}} C^n,\]
every component of $\Delta$ contains some $p^i$.

We now turn to Question~\eqref{q-limit} above. Since $\mathcal{C}$ is smooth, we may extend $\mathcal{L}^*$ to a line bundle $\mathcal{L}$
on $\mathcal{C}$. However, this extension is only unique up to twisting by divisors on $\C$ that do not meet the generic fiber,
i.e.\ which do not dominate $B$. We now describe a basis for such divisors.

\medskip

Since $\C \to B$ is a family of chain curves, each component of $\Delta$ contains at most one $p^i$.
Because $\C \to B$ is a family of nodal curves, $\mathfrak{f} \colon \Delta \to B$ is unramified.
Moreover, 
because the versal deformation space of a node is $\Spec K[[x, y, t]] / (xy - t) \to \Spec K[[t]]$,
and the total space $\C$ is smooth,
the image under $\mathfrak{f}$ of any component of $\Delta$ is a smooth divisor in $B$.
Consequently, $\Delta$ is smooth of codimension~$2$ in $\C$.
Thus, each $p^i$ is contained in a unique component of $\Delta^i$.

Putting this together, there are exactly $n - 1$ components of $\Delta$,
one containing each node of $C$.
Label these components $\Delta^1, \Delta^2, \ldots, \Delta^{n - 1}$, so that $\Delta^i$ contains $p^i$.

Consider any component $S$ of $\mathfrak{f}(\Delta)$, and let $\{i_1, i_2, \ldots, i_{m(S)}\}$
denote the set of $i$ such that $\mathfrak{f}(\Delta^i) = S$.  (As we range through all components of $\mathfrak{f}(\Delta)$, these sets form a partition of $\{1, 2, \dots, n\}$.)
Then, because $\C \to B$ is a family of chain curves, $\mathfrak{f}^{-1}(S) = S_1 \cup S_2 \cup \cdots \cup S_{m(S) + 1}$
has exactly $m(S) + 1$ components, meeting pairwise along the $\Delta^{i_j}$:

\begin{center}
\begin{tikzpicture}
\node[above right] at (-0.1, 9.95) {$C$};

\draw (-0.1, 10.1) .. controls (0.2, 9.75) and (0.2, 9.25) .. (-0.1, 8.9);
\draw (-0.1, 9.1) .. controls (0.2, 8.75) and (0.2, 8.25) .. (-0.1, 7.9);
\draw (-0.1, 7.1) .. controls (0.2, 6.75) and (0.2, 6.25) .. (-0.1, 5.9);
\draw (-0.1, 6.1) .. controls (0.2, 5.75) and (0.2, 5.25) .. (-0.1, 4.9);
\draw (-0.1, 5.1) .. controls (0.2, 4.75) and (0.2, 4.25) .. (-0.1, 3.9);
\draw (-0.1, 4.1) .. controls (0.2, 3.75) and (0.2, 3.25) .. (-0.1, 2.9);
\draw (-0.1, 10.1) -- (-5, 11);
\draw (-0.03, 9) -- (-6, 9);
\draw (-0.03, 6) -- (-4, 7);
\draw (-0.03, 4) -- (-4, 3.75);
\draw (-0.1, 2.9) -- (-5, 3);

\node[left] at (-6, 9) {$\Delta^{i_{m(S)}}$};

\node[left] at (-4, 7) {$\Delta^{i_2}$};

\node[left] at (-4, 3.75) {$\Delta^{i_1}$};

\draw (-5 + .5, 11-0.09)  to[out=190, in=100, looseness=.5] (-6 + .5, 9);
\draw (-6 + .5, 9) to[out=350, in=135, looseness=1] (-4.1, 8.25);
\draw (-3.7, 7.75) to[out=315, in=70, looseness=.5] (-4+.5, 7-0.125);
\draw (-4+.5, 7-0.125) to[out=285, in=130, looseness=1] (-4+.5, 3.75+ 0.031);
\draw (-4+.5, 3.75+ 0.031) to[out=200, in=90, looseness=.5] (-5+.5, 3-0.01);

\node at (-3.87, 8) {$\ddots$};
\node at (-.1, 7.5) {$\vdots$};

\node at (-1.8, 3.3) {$S_1$};
\node at (-2.2, 5.25) {$S_2$};
\node at (-1.5, 6.75) {$S_3$};
\node at (-2, 8.65) {$S_{m(S)}$};
\node at (-3.25, 9.75) {$S_{m(S)+1}$};

\draw[decorate,decoration={brace,amplitude=10pt, mirror}] (-5.25, 7) -- (-5.25, 3);

\node[left] at (-5.5, 5.1) {$\Sigma^{i_2}$};

\draw (-6.5, 1) -- (-5, 2) -- (1, 2) -- (-.5, 1) --(-6.5, 1);

\draw[fill=black] (-.1, 1.75) circle (1.5pt);
\node[below] at (-.1, 1.75) {$b$};

\draw (-.1, 1.75) -- (-4.5, 1.6);

\node[below left] at (-4.5, 1.6) {$S$};

\node[right] at (1, 2) {$B$};

\end{tikzpicture}
\end{center}
As shown in the above diagram, these components are indexed so that:
\[S_j \cap C = \begin{cases}
C^1 \cup \cdots \cup C^{i_1} & \text{if $j = 1$;} \\
C^{i_{m(S)} + 1} \cup \cdots \cup C^n & \text{if $j = m(S) + 1$;} \\
C^{i_{j-1} + 1} \cup \cdots \cup C^{i_j} & \text{otherwise.}
\end{cases} \qquad \text{and} \qquad S_j \cap S_{j'} = \begin{cases}
\Delta^{i_j} & \text{if $j' = j + 1$;} \\
\emptyset & \text{if $j' > j + 1$.}
\end{cases}
\]
For $1 \leq j \leq m(S)$, we define
\[\Sigma^{i_j} = S_1 + S_2 + \cdots + S_j \quad \text{which satisfies} \quad \Sigma^{i_j} \cap C = C^1 + C^2 + \cdots + C^{i_j}.\]
By construction, every divisor on $\C$ supported on $\mathfrak{f}^{-1}(S)$
is a unique linear combination of the $\Sigma^{i_j}$ and $\mathfrak{f}^{-1}(S)$.
Repeating this construction for every component $S$ of $\mathfrak{f}(\Delta)$, we will have
defined divisors $\Sigma^i$ for all $1 \leq i \leq n - 1$.

\begin{example}
When $B$ is the spectrum of a DVR, and $b$ is the special fiber, then we have
$\Sigma^i = C^1 + C^2 + \cdots + C^i$.
\end{example}

Now suppose that $D$ is any irreducible divisor such that $\mathfrak{f}(D)$ is a divisor on $B$
not contained in $\mathfrak{f}(\Delta)$.
Then the generic fiber of $\C$ over $\mathfrak{f}(D)$ is irreducible, so $D$ is a multiple of
$\mathfrak{f}^{-1}(\mathfrak{f}(D))$.
Putting this together, we learn that any divisor on $\C$ that does not dominate $B$
can be written uniquely as a linear combination of the $\Sigma^i$ and the pullback of a divisor on $B$.

Note that twisting by the pullback of a divisor on $B$
does not change $\mathcal{L}|_C$, and that twisting by the $\Sigma^i$
changes the $\mathcal{L}|_{C^j}$ as follows:
\begin{equation}
\mathcal{L}(\Sigma^i)|_{C^j} \simeq \begin{cases}
\mathcal{L}|_{C^j}(p^i) & \text{if $j = i$;} \\
\mathcal{L}|_{C^j}(-p^i) & \text{if $j = i + 1$;} \\
\mathcal{L}|_{C^j} & \text{otherwise.}
\end{cases}
\end{equation}
In particular, for any \defi{degree distribution} $\vec{d} = (d^1, d^2, \ldots, d^n)$
with $d = \sum d^i$, there is an extension $\mathcal{L}_{\vec{d}}$ of $\mathcal{L}^*$ to $\mathcal{C}$
so that $\mathcal{L}_{\vec{d}}|_C$ has degree $\vec{d}$
(i.e.\ has degree $d^i$ on $C^i$), which is unique up to twisting by the pullback of a divisor on $B$.
Moreover, any one extension $\mathcal{L}_{\vec{d}}$ determines all other extensions
(up to pullbacks of divisors on $B$)
via the above relation.

Restricting to the fiber $C$ over $b$, we conclude that for each such degree distribution $\vec{d}$,
there is a unique limit $L_{\vec{d}} \colonequals \mathcal{L}_{\vec{d}}|_C$
of degree $\vec{d}$. Moreover, any one limit $L_{\vec{d}}$ determines all other limits
via repeatedly applying the relation:
\begin{equation} \label{twisting-game}
L_{(d^1, d^2, \ldots, d^i + 1, d^{i + 1} - 1, \ldots, d^g)}|_{C^j} \simeq \begin{cases}
L_{(d^1, d^2, \ldots, d^g)}|_{C^j}(p^i) & \text{if $j = i$;} \\
L_{(d^1, d^2, \ldots, d^g)}|_{C^j}(-p^i) & \text{if $j = i + 1$;} \\
L_{(d^1, d^2, \ldots, d^g)}|_{C^j} & \text{otherwise.}
\end{cases}
\end{equation}
The following definition thus encapsulates the data we obtain on any fiber:

\begin{defin} \label{limit-linebundle} Let
\[\Pic^d C \colonequals \frac{\bigsqcup_{\vec{d} : \sum d^i = d} \Pic^{\vec{d}} C}{\sim},\]
where $\sim$ denotes the equivalence relation generated by \eqref{twisting-game}.
We call elements $L$ of $\Pic^d C$ \defi{limit line bundles} of degree $d$,
and write $L_{\vec{d}}$ for the corresponding line bundle on $C$ of degree $\vec{d}$.

If $D = C^i \cup C^{i + 1} \cup \cdots \cup C^j \subset C$ is any connected curve,
we write $L^D$ for the ``restriction of $L$ to $D$ as a limit line bundle of degree $d$''.
More formally, for any degree distribution $(d^i, d^{i + 1}, \ldots, d^j)$ on $D$ with $d^i + d^{i + 1} + \cdots + d^j = d$,
we have
\[(L^D)|_{(d^i, d^{i + 1}, \ldots, d^j)} = L_{(0, \ldots, 0, d^i, d^{i + 1}, \ldots, d^j, 0, \ldots, 0)}|_D.\]
For ease of notation when $C = X$ (respectively $C = P$) is our chain of $g$ elliptic (respectively rational) curves,
we set $L^i = L^{E^i}$ (respectively $L^i = L^{P^i}$).
These are limit line bundles on smooth curves, which are just ordinary line bundles.
\end{defin}

In other words, if we fix a degree distribution $\vec{d}$ with $\sum d^i = d$,
then we have a natural isomorphism $\Pic^d C \simeq \Pic^{\vec{d}} C$;
but $\Pic^d C$ exists without fixing a degree distribution
(although its elements do not then yet correspond naturally to line bundles on $C$).
Note that $\Pic^d C$ is a torsor for $\Pic^\circ C \simeq \prod \Pic^0 C^i$,
and that there are natural tensor product maps $\Pic^{d_1} C \times \Pic^{d_2} C \to \Pic^{d_1 + d_2} C$.

\begin{example} \label{limit-Om}
Consider the family appearing in Section~\ref{our_degen}.
When $\mathcal{L}^* = \O_{\mathcal{C}^*}(m) \colonequals \mathfrak{f}^* \O_{\mathcal{P^*}}(m)$,
we obtain limit line bundles $\O_C(m)$. These can be described
in terms of the geometry of the central fiber alone: For instance,
if we fix the degree distribution $(mk, 0, \ldots, 0)$, we have
\[\O_C(m)_{(mk, 0, \ldots, 0)}|_{C^i} = \begin{cases}
\O_{C^1}(m) \colonequals (f^1)^* \O_{\pp^1}(m) & \text{if $i = 1$;} \\
\O_{C^i} & \text{otherwise.}
\end{cases}\]
By slight abuse of notation, we write $\O_{\mathcal{P}}(m)^i \colonequals \beta^* \O_{\mathcal{P}_0}(m)^i$,
where $\beta \colon \P \to \P_0$ is the base-change of $\beta \colon B \to B_0$ appearing in Section~\ref{our_degen}.
\end{example}

This then provides an answer to the first question posed at the beginning of the section:
To a line bundle $\mathcal{L}^*$ on $\mathcal{C}^*$ on the generic fiber,
we can associate a limit line bundle $L$ of degree $d$ on $C$.

\medskip

We now turn to the second question: Suppose that $\mathfrak{f}_* \mathcal{L}^*$ has splitting type $\vec{e}$. What can we say about
the associated limit line bundle $L$?
First of all,
\[\chi(L) = \chi(\mathcal{L}^*) = \chi(\pp^1, \O_{\pp^1}(\vec{e})),\]
and so 
\begin{equation} \label{deg-constraint}
d = g - 1 + \chi(\pp^1, \O_{\pp^1}(\vec{e})).
\end{equation}
Moreover, since $\mathcal{L}^*$ has splitting type $\vec{e}$, we have
\begin{equation} \label{h0-gen}
h^0(\mathcal{C}^*,\mathcal{L}^*(m)) = h^0(\pp^1, \O_{\pp^1}(\vec{e})(m)) = \sum_{\ell=1}^k \max(0, e_\ell + m + 1) \quad \text{for any $m$}.
\end{equation}
By semicontinuity, the limit line bundle $L$ therefore satisfies
\begin{equation} \label{h0-limit}
h^0(C, L(m)_{\vec{d}}) \geq \sum_{\ell=1}^k \max(0, e_\ell + m + 1) \quad \text{for any degree distribution $\vec{d}$ with $\sum_{i=1}^n d^i = d + mk$}.
\end{equation}
The following definition thus encapsulates the conditions our data on the central fiber
must satisfy:

\begin{defin} \label{def-epos}
We say that a limit line bundle $L \in \Pic^d(C)$ is \emph{$\vec{e}$-positive} if it satisfies \eqref{deg-constraint} and \eqref{h0-limit}.
\end{defin}

This then provides an answer to the second question posed at the beginning of the section:
If $\mathfrak{f}_* \mathcal{L}^*$ has splitting type $\vec{e}$, then the associated limit line bundle
$L$ must be $\vec{e}$-positive.

\medskip

In fact, there is a proper \emph{scheme} $W^{\vec{e}}(\mathcal{C})$ over $B$ whose fibers
over every point parameterize
$\vec{e}$-positive line bundles on the corresponding fiber of $\C \to B$.
This scheme will be an intersection of determinantal loci (over all degree distributions).
To construct this scheme, work locally on the base near $b \in B$ as above, and write
$\pi \colon \Pic^d (\mathcal{C}/B) \times_B \mathcal{C} \to \Pic^d(\mathcal{C}/B)$ for the projection map.
For any degree distribution $\vec{d}$ on $C \colonequals \C \times_B b$
of $d + mk$,
we obtain a universal bundle $\mathcal{L}(m)_{\vec{d}}$. For each $m$
and $\vec{d}$, there is a natural scheme structure on
\begin{align*}
&\{L \in \Pic^{\vec{d}}(\mathcal{C}/B) : h^0(\pi^{-1}(L), L(m)) \geq h^0(\pp^1, \O(\vec{e})(m)) \} \\
 &\quad = \{L \in  \Pic^d (\mathcal{C}/B) :  \rk (R^1\pi_* \L(m)_{\vec{d}})|_{L} \geq h^1(\pp^1, \O(\vec{e})(m))\},
 \end{align*}
defined by the Fitting support for where  $\rk R^1\pi_* \L(m)_{\vec{d}} \geq h^1(\pp^1, \O(\vec{e})(m))$, as we now recall.
The Fitting supports of a coherent sheaf are defined by the appropriately sized determinantal loci of a resolution by vector bundles and are independent of the resolution (see for example Section~20.2 of~\cite{eca}).

An often-used resolution of $R^1\pi_* \L(m)_{\vec{d}}$ is constructed as follows.
Let $D_{\vec{d}} \subset \mathcal{C}$ be a sufficiently relatively ample divisor
(relative to $\vec{d}$), so that $\pi_* [\L(m)_{\vec{d}}(D_{\vec{d}})]$ and $\pi_* [\L(m)_{\vec{d}}(D_{\vec{d}})|_{D_{\vec{d}}}]$
are vector bundles on $ \Pic^d(\mathcal{C}/B)$. 
Pushing forward the exact sequence
\[0 \rightarrow \L(m)_{\vec{d}} \rightarrow \L(m)_{\vec{d}}(D_{\vec{d}}) \rightarrow \L(m)_{\vec{d}}(D_{\vec{d}}) |_{D_{\vec{d}}} \rightarrow 0\]
by $\pi$ we see that 
the restriction map
\[\pi_* [\L(m)_{\vec{d}}(D_{\vec{d}})] \to \pi_* [\L(m)_{\vec{d}}(D_{\vec{d}})|_{D_{\vec{d}}}]\]
provides a resolution of $R^1\pi_* \L(m)_{\vec{d}}$.
Using the scheme structure defined by the appropriate minors, we define
\[W^{\vec{e}}(\C) \colonequals \bigcap_{m, \vec{d}} \left\{L \in  \Pic^{\vec{d}} (\mathcal{C}/B) :  \rk (R^1\pi_* \L(m)_{\vec{d}})|_{L} \geq h^1(\pp^1, \O(\vec{e})(m))\right\}.\]
Since $h^1(\pp^1, \O(\vec{e})(m)) = 0$ for $m$ large, only finitely many terms in the intersection are proper subschemes of $\Pic^{\vec{d}} (\mathcal{C}/B)$.

\section{Classification of $\vec{e}$-Positive Limit Line Bundles}\label{sec:epos}

Returning to notation of Section \ref{our_degen}, in this section we classify $\vec{e}$-positive line bundles on the central fiber $X$. The following description in terms of $k$-staircase tableaux is an observation due to Cook-Powell--Jensen in the tropical setting \cite{CPJ2}. Here, we provide a self-contained proof in the classical setting.

For any $0 \leq i \leq g$, and any degree distribution $\vec{d}$, write
\[X^{\leq i} = E^1 \cup E^2 \cup \cdots \cup E^i \quad \text{and} \quad d^{\leq i} = d^1 + d^2 + \cdots + d^i.\]

\begin{defin} \label{defi-ain}
For a limit line bundle $L$, and $1 \leq i \leq g - 1$, and $n \geq 1$, define
\[a^i_n(L) = \min\{\alpha : \text{we have $h^0(X^{\leq i}, L_{\vec{d}}|_{X^{\leq i}}) \geq n$ for any degree distribution $\vec{d}$ with $d^{\leq i} = \alpha$}\}.\]
We extend this to $i = g$ via
\begin{multline*}
a^g_n(L) = \min\Big\{\alpha : \text{for some $m$ and $\epsilon$ with $d + mk = \alpha + \epsilon$ and $\epsilon \geq 0$, we have} \\
\text{$h^0(X, L(m)_{\vec{d}(m)}) \geq n + \epsilon$ for any degree distribution $\vec{d}(m)$ with $d(m)^{\leq g} = d + mk$} \Big\},
\end{multline*}
and to $i = 0$ via
\[a^0_n(L) = n - 1.\]
\end{defin}

For $1 \leq i \leq g -1$, unwinding the definition of $a^i_n$, there exists a degree distribution $\vec{d}$ with $d^{\leq i} = a^i_n -1$  satisfying
$h^0(X^{ \leq i}, L_{\vec{d}}|_{X^{\leq i}}) \leq n -1.$
Furthermore, since vanishing at a single point imposes at most one condition on global sections, there exists a degree distribution $\vec{d}$
with $d^{\leq i} = a^i_n$ witnessing $h^0(X^{ \leq i}, L_{\vec{d}}|_{X^{ \leq i}}) = n$,
such that not every section of $L_{\vec{d}}|_{X^{ \leq i}}$ vanishes at $p^i$.

\begin{prop} \label{ram-ineq-easy}
We have $a^i_n > a^i_{n - 1}$.
\end{prop}
\begin{proof}
The case $i = 0$ is clear by definition.

When $1 \leq i \leq g - 1$, let $\vec{d}$ be a degree distribution with $d^{\leq i} = a^i_{n - 1}$
witnessing $h^0(X^{ \leq i}, L_{\vec{d}}|_{X^{ \leq i}}) = n - 1$. This implies $a^i_n > a^i_{n - 1}$ as desired.

Finally, when $i = g$, we claim that for any $m$ and $\epsilon$ with $d + mk = a^g_{n - 1} + \epsilon$,
there is some degree distribution $\vec{d}(m)$ with $d(m)^{\leq g} = d + mk$
such that $h^0(X, L(m)_{\vec{d}(m)}) \leq (n - 1) + \epsilon$.
Indeed, if not, then $h^0(X, L(m)_{\vec{d}(m)}) \geq (n - 1) + (\epsilon + 1)$
for every such degree distribution, which would contradict the definition of $a^g_{n - 1}$ because
$d + mk = (a^g_{n - 1} - 1) + (\epsilon + 1)$.
This implies $a^g_n > a^g_{n - 1}$ as desired.
\end{proof}

\begin{prop} \label{ram-ineq}
We have $a^{i}_n \geq a^{i-1}_n$. If equality holds, then $L^i \simeq \O_{E^i}(a^{i-1}_n p^{i - 1} + (d - a^{i-1}_n) p^i)$.
\end{prop}

\begin{rem} Our proof will show that if equality holds when $i = 1$ (respectively $i = g$) then $a^0_n = 0$
(respectively $a^{g - 1}_n \equiv d$ mod $k$). Thus the formula given for $L^i$
is independent of choice of $p^0$ and $p^g$.
\end{rem}

\begin{proof}
We separately consider the following cases:

\smallskip

\paragraph{\textbf{\boldmath The Case $i = 1$:}} 
For any degree distribution $\vec{d}$, the line bundle
$L_{\vec{d}}|_{E^1} \simeq L^1(-(d - d^1)p^1)$
is of degree $d^1$ on a genus $1$ curve and hence by Reimann--Roch
has a $\max(0,d^1)$-dimensional space of global sections unless $d^1=0$ and $L^1(-d p^1) \simeq \O_{E^1}$.
Hence, there is no degree distribution $\vec{d}$ such that $d^1 \leq a^0_n - 1 = n - 2$ and $h^0(E^1, L_{\vec{d}}|_{E^1}) \geq n$.
Furthermore, there is no such degree distribution with $d^1 = a^0_n = n - 1$ and $h^0(L_{\vec{d}}|_{E^1}) \geq n$ unless $a^0_n = 0$ and $L^1 = \O_{E^1}(dp^1) = \O_{E^1}(a^0_n p^0 + (d - a^0_n) p^1)$.

\smallskip

\paragraph{\textbf{\boldmath The Case $2 \leq i \leq g-1$:}}
Let $\vec{d}$ be a degree distribution such that $d^{\leq i-1} = a^{i-1}_n -1$ and 
\[h^0(X^{\leq i-1}, L_{\vec{d}}|_{X^{ \leq i-1}}) < n.\]    
We may further assume that $d^i = 0$. Then
\[h^0(X^{\leq i}, L_{\vec{d}}|_{X^{\leq i}}) \leq h^0(X^{\leq i-1}, L_{\vec{d}}|_{X^{\leq i-1}}) + h^0(E^i, L_{\vec{d}}|_{E^i}(-p^{i - 1})) < n. \]
Therefore $a^{i}_n \geq a^{i-1}_n$.  

Furthermore, there exists a degree distribution $\vec{d}$ with $d^{\leq i-1} = a^{i-1}_n$ and $d^i = 0$ witnessing $h^0(X^{\leq i-1 }, L_{\vec{d}}|_{X^{\leq i-1}}) = n$, and
$h^0(X^{\leq i-1 }, L_{\vec{d}}|_{X^{\leq i-1}}(-p^{i-1})) = n - 1$. Thus
\[h^0(X^{\leq i}, L_{\vec{d}}|_{X^{\leq i}}) = h^0(E^i, L_{\vec{d}}|_{E^i}) + n - 1.\]
If $a^{i}_n = a^{i-1}_n$, then to ensure this degree distribution has enough sections,
$h^0(E^i, L_{\vec{d}}|_{E^i}) > 0$.
Since $L_{\vec{d}}|_{E^i}$ has degree zero, this implies $L_{\vec{d}}|_{E^i} \simeq \O_{E^i}$.
Applying \eqref{twisting-game},
\[L_{\vec{d}}|_{E^i} \simeq L^i(-a^{i-1}_n p^{i-1} - (d- a^{i-1}_n)p^i), \]
so this implies the desired condition.

\paragraph{\textbf{\boldmath The Case $i = g$:}}
Let $\vec{d}$ be a degree distribution such that $d^{\leq g-1} = a^{g-1}_n -1$
and
\[h^0(X^{\leq g-1}, L_{\vec{d}}|_{X^{\leq g-1}}) < n.\]
Let $m$ and $\epsilon \geq 0$ be any integers such that $d + mk = a^{g-1}_n - 1 + \epsilon$.
Define
\[\vec{d}(m) \colonequals (d^1, d^2, \dots, d^{g-1}, d^g + mk).\]
Then $d^g + mk = \epsilon$. Thus
\[h^0(X, L(m)_{\vec{d}(m)}) \leq h^0(X^{\leq g-1}, L(m)_{\vec{d}(m)}|_{X^{\leq g-1}}) + h^0(E^g, L(m)_{\vec{d}(m)}|_{E^g}(-p^{g - 1})) < n + \epsilon.\]
Therefore $a^{g}_n \geq a^{g-1}_n$.

Furthermore, there exists a degree distribution $\vec{d}$
with $d^{\leq g-1} = a^{g-1}_n$, witnessing
\[h^0(X^{\leq g-1}, L_{\vec{d}}|_{X^{\leq g-1}}) = n \quad \text{and} \quad h^0(X^{\leq g-1}, L_{\vec{d}}|_{X^{\leq g-1}}(-p^{g-1})) = n - 1.\]

Let $m$ and $\epsilon \geq 0$ be any integers such that $d + mk = a^{g-1}_n + \epsilon$.
Define $\vec{d}(m)$ as above; as before, $d^g + mk = \epsilon$. We have
\[h^0(X, L(m)_{\vec{d}(m)}) = h^0(E^g, L(m)_{\vec{d}(m)}|_{E^g}) + n - 1.\]
If $a^g_n = a^{g-1}_n$, then for some such choice of $m$ and $\epsilon$, we must have
$h^0(E^g, L(m)_{\vec{d}(m)}|_{E^g}) > \epsilon$.
Since $\deg L(m)_{\vec{d}(m)}|_{E^g} = d^g + mk = \epsilon$, this implies $\epsilon = 0$ and
$L(m)_{\vec{d}(m)}|_{E^g} \simeq \O_{E^g}$. Applying \eqref{twisting-game},
\[L(m)_{\vec{d}(m)}|_{E^g} \simeq L^g(m)(-a^{g - 1}_n p^{g - 1}) \simeq L^g((mk - a^{g - 1}_n) p^{g - 1}) \simeq L^g(-a^{g - 1}_n p^{g - 1} - (d - a^{g - 1}_n) p^g),\]
so this is exactly the desired condition.
\end{proof}

\noindent
We now repackage this information as follows:

\begin{defin} 
For $n \geq 1$, write
\[f_n(i) = i + n - 1 - a^i_n,\]
and define
\begin{align*}
h(n) \colonequals h_{\vec{e}}(n) &= \max\left\{ h^1(\pp^1, \O(\vec{e})(m) ): \text{$m$ satisfies $h^0(\pp^1, \O(\vec{e})(m)) \geq n$}\right\} \\
&=  \max \left\{\sum_{\ell = 1}^k \max(0, -e_\ell - m - 1) : \text{$m$ satisfies $\sum_{\ell = 1}^k \max(0, e_\ell + m + 1) \geq  n$}\right\}.
\end{align*}
\end{defin}
\noindent
Note that $h(n)$ is nonincreasing and is zero for $n$ large.

\begin{prop} \label{f-h}
If $L$ is $\vec{e}$-positive, then $f_n(g) \geq h(n)$.
\end{prop}
\begin{proof}
Suppose that $m$ satisfies $h^0(\pp^1, \O(\vec{e})(m)) \geq n$; let $\epsilon = h^0(\pp^1, \O(\vec{e})(m)) - n \geq 0$.
By \eqref{h0-limit}, we have
\[h^0(X, L(m)_{\vec{d}(m)}) \geq h^0(\pp^1, \O(\vec{e})(m)) = n + \epsilon,\]
for any degree distribution $\vec{d}(m)$ with $d(m)^{\leq g} = d + mk$.
Therefore by Definition~\ref{defi-ain}, we have
\[a^g_n \leq d + mk - \epsilon = d + mk - h^0(\pp^1, \O(\vec{e})(m)) + n.\]
Thus,
\[f_n(g) \geq g + n - 1 - [d + mk - h^0(\pp^1, \O(\vec{e})(m)) + n] = h^1(\pp^1, \O(\vec{e})(m)).\]
Therefore $f_n(g) \geq h(n)$.
\end{proof}

The inequality of Proposition~\ref{f-h} forces equality to hold in Proposition~\ref{ram-ineq}
for many values of $i$ and $n$. To keep track of when equality holds, we use a combinatorial object
that we will term a \defi{$k$-staircase tableau}.

\begin{defin} \label{yd-b}
A \defi{Young diagram} is a finite collection of boxes arranged in left-justified rows, such that the number of boxes in each row is nonincreasing.  We index the boxes by their row and column $(r,c)$, beginning with $(1,1)$, and we define the
\defi{diagonal index} of a box to be $c - r$.

The \defi{boundary} of a Young diagram is the sequence of line segments formed by the right-most edges of the last box in every row and the bottom-most edge of the last box in every column.  For convenience, we extend this to infinity below and to the right of the diagram.
We index the boundary segments by the diagonal index of the box above (if the segment is horizontal),
or to the right (if the segment is vertical).
\begin{center}
\begin{tikzpicture}[scale=.4]
\draw (0,0) -- (8, 0);
\draw (0,-1) -- (8,-1);
\draw (0,-2) -- (5,-2);
\draw (0,-3) -- (4,-3);
\draw (0,-4) -- (3,-4);
\draw (0,-5) -- (2,-5);
\draw (0,-6) -- (1,-6);
\draw (0,-7) -- (1,-7);
\draw (0,-8) -- (1,-8);
\draw (0,0) -- (0,-8);
\draw (1,0) -- (1,-8);
\draw (2,0) -- (2,-5);
\draw (3,0) -- (3,-4);
\draw (4,0) -- (4,-3);
\draw (5,0) -- (5,-2);
\draw (6,0) -- (6,-1);
\draw (7,0) -- (7,-1);
\draw (8,0) -- (8,-1);

\draw[->] (-.5,0) -- (-.5,-7);
\draw (-.5, -3.5)node[left]{$r$};

\draw[->] (0,.5) -- (7, .5);
\draw (3.5, 0.5)node[above]{$c$};

\draw[ ultra thick] (0,-10) -- (0,-8) -- (1,-8) -- (1,-5) -- (2,-5) -- (2,-4) -- (3,-4) -- (3,-3) -- (4, -3) -- (4, -2) -- (5, -2) -- (5, -1) -- (8,-1) -- (8,0) -- (10,0);

\draw[ ultra thick] (0,-10.3)--(0, -10.4);
\draw[ ultra thick] (0,-10.5)--(0, -10.6);
\draw[ ultra thick] (0,-10.7)--(0, -10.8);

\draw[ ultra thick] (10.3,0)--(10.4,0);
\draw[ ultra thick] (10.5,0)--(10.6,0);
\draw[ ultra thick] (10.7,0)--(10.8, 0);

\draw [ultra thick, decorate,decoration={brace,amplitude=10pt, mirror},xshift=-4pt,yshift=0pt]
(3,-6) -- (6,-3);

\draw (7.5, -5.25) node{{\small \textbf{boundary}}};

{\tiny
\draw[color=white, fill=white] (0,-8.5) circle (8pt);
\draw (0,-8.5)node{\textbf{-8}};
\draw[color=white, fill=white] (0.5,-8) circle (8pt);
\draw (0.5,-8)node{\textbf{-7}};
\draw[color=white, fill=white] (1,-7.5) circle (8pt);
\draw (1,-7.5)node{\textbf{-6}};
\draw[color=white, fill=white] (1, -6.5) circle (8pt);
\draw (1, -6.5)node{\textbf{-5}};
\draw[color=white, fill=white] (1, -5.5) circle (8pt);
\draw (1, -5.5)node{\textbf{-4}};
\draw[color=white, fill=white] (1.5, -5) circle (8pt);
\draw (1.5, -5)node{\textbf{-3}};
\draw[color=white, fill=white] (2, -4.5) circle (8pt);
\draw (2, -4.5)node{\textbf{-2}};
\draw[color=white, fill=white] (2.5, -4) circle (8pt);
\draw (2.5, -4)node{\textbf{-1}};
\draw[color=white, fill=white] (3, -3.5) circle (8pt);
\draw (3, -3.5)node{\textbf{0}};
\draw[color=white, fill=white] (3.5, -3) circle (8pt);
\draw (3.5, -3)node{\textbf{1}};
\draw[color=white, fill=white] (4, -2.5) circle (8pt);
\draw (4, -2.5)node{\textbf{2}};
\draw[color=white, fill=white] (4.5, -2) circle (8pt);
\draw (4.5, -2)node{\textbf{3}};
\draw[color=white, fill=white] (5, -1.5) circle (8pt);
\draw (5, -1.5)node{\textbf{4}};
\draw[color=white, fill=white] (5.5, -1) circle (8pt);
\draw (5.5, -1)node{\textbf{5}};
\draw[color=white, fill=white] (6.5, -1) circle (8pt);
\draw (6.5, -1)node{\textbf{6}};
\draw[color=white, fill=white] (7.5, -1) circle (8pt);
\draw (7.5, -1)node{\textbf{7}};
\draw[color=white, fill=white] (8, -.5) circle (8pt);
\draw (8, -.5)node{\textbf{8}};
\draw[color=white, fill=white] (8.5, 0) circle (8pt);
\draw (8.5, 0)node{\textbf{9}};
}

\end{tikzpicture}
\end{center}

\end{defin}

Because $h(n)$ is nonincreasing and zero for $n$ large, the data of the function $h(n)$
(which is defined for positive integers $n$)
is thus the same as the data of a Young diagram,
where we put $h(n)$ boxes in the $n$th column.

\begin{defin}
For a splitting type $\vec{e}$, we write $\Gamma(\vec{e})$ for
the Young diagram determined by $h_{\vec{e}}(n)$ in the above manner.
We call a Young diagram of the form $\Gamma(\vec{e})$ for some $\vec{e}$ a \defi{$k$-staircase}.
\end{defin}

\begin{center}
\begin{tikzpicture}[scale=0.7]
\draw (0,0) -- (7, 0);
\draw (0,-1) -- (7,-1);
\draw (0,-2) -- (4,-2);
\draw (0,-3) -- (2,-3);
\draw (0,-4) -- (2,-4);
\draw (0,0) -- (0,-4);
\draw (1,0) -- (1,-4);
\draw (2,0) -- (2,-4);
\draw (3,0) -- (3,-2);
\draw (4,0) -- (4,-2);
\draw (5,0) -- (5,-1);
\draw (6,0) -- (6,-1);
\draw (7,0) -- (7,-1);
\node[below right] at (-0.1, -3.9) {$h(1) =  h(2) = 4$};
\node[below right] at (1.9, -1.9) {$h(3) =  h(4) = 2$};
\node[below right] at (3.9, -0.9) {$h(5) =  h(6) =  h(7) = 1$};
\node[below right] at (6.9,0.1) {$h(n) = 0$ for $n \geq 8$};
\end{tikzpicture} \\
Example: $\Gamma(\vec{e})$ for $\vec{e} = (-4, -2, 0, 0)$ 
\end{center}

For each $\vec{e}$-positive line bundle $L$, we will use the functions $f_n(i)$ to build a filling $T$ of
$\Gamma(\vec{e})$.
Namely, we have $f_n(0) = 0$ and $f_n(g) \geq h(n)$, and by Proposition~\ref{ram-ineq},
$f_n(i) \leq f_n(i-1) + 1$. Therefore, $f_n$ assumes every value between $0$ and $h(n)$
inclusive.
Our filling $T$ of $\Gamma(\vec{e})$ is obtained by placing
$\min \{i : f_c(i) = r\}$ in the $r$th row of the $c$th column.

\begin{prop} \label{li-value}
If $i$ is in the $r$th row of the $c$th column of $T$, then
\[L^i \simeq \O_{E^i}((c - r + i - 1) p^{i - 1} + (d - (c - r + i - 1)) p^i).\]
In particular, if $i$ appears in multiple boxes of $T$, then it follows that all such boxes have the same value of $c - r$ modulo $k$.
Moreover, this filling is increasing along rows and columns.
\end{prop}
\begin{proof}
Given $r$ and $c$, suppose $i$ is the first time for which $f_c(i) = r$.
 Because this is a new maximum, we must have $f_c(i-1) = r-1$, which implies $a^{i-1}_c = a^{i}_c = c - r + i - 1$. By Proposition \ref{ram-ineq},
\[L^i = \O_{E^i}(a^{i-1}_np^{i-1} + (d - a_n^{i-1})p^i) = \O_{E^i}((c - r + i - 1)p^{i-1} + (d - (c - r + i - 1))p^i), \]
as desired.
In particular, if $i$ appears in multiple boxes of $T$, then since $p^{i - 1} - p^i$ is exactly $k$-torsion in $\Pic^0 E^i$,
all such boxes have the same value of $c - r$ modulo $k$.

We now show that the filling is increasing along rows and columns. Since $f_c(i) \leq f_c(i-1) + 1$, the function $f_c$ must attain the value $r$ before it attains $r+1$. This shows the filling is increasing down column $c$.
Meanwhile, by Proposition~\ref{ram-ineq-easy}, we have $a^i_{c-1} < a^i_{c}$ and so $f_{c-1}(i) \geq f_{c}(i)$. It follows that 
$\min\{i : f_{c-1}(i) = r\} \leq \min \{i : f_{c}(i) = r\}$
(the larger function must attain $r$ at an earlier or same time).
However, if equality holds, the first part of this proposition says that
$c - r \equiv (c - 1) - r$ mod~$k$, which is impossible.
Thus, $\min\{i : f_{c-1}(i) = r\} < \min \{i : f_{c}(i) = r\}$, which shows the filling is increasing along row $r$.
\end{proof}

\begin{defin} \label{Tsquare}
A filling $T$ of a Young diagram is called \defi{$k$-regular} if it is increasing along rows and columns,
and all boxes containing the same symbol $i$ have the same value of $c - r$ modulo $k$.
We write $T[i] \in \zz/k\zz \cup \{*\} = \{1, 2, \ldots, k, *\}$
for this common value of $c - r$ modulo $k$ if $i$ appears in $T$;
if $i$ does not appear in $T$ then we set $T[i] = *$.
We call a $k$-regularly filled $k$-staircase a \defi{$k$-staircase tableau}.
\end{defin}

\noindent
For the remainder of the paper, all fillings of any Young diagram will be assumed to be $k$-regular.

\begin{defin} \label{WT}
Given a tableau $T$, we define a corresponding reduced subscheme of $\Pic^d(X)$ by
\[W^T(X) \colonequals \left\{L \in \Pic^d(X) : L^i \simeq \O_{E^i}((T[i] + i - 1)p^{i-1} + (d - (T[i] + i - 1))p^i) \ \text{if} \ T[i] \neq *\right\}.\]
Similarly, given a diagram $\Gamma$, we define
\[W^\Gamma(X) \colonequals \bigcup_{\text{$T$ filling} \atop \text{of $\Gamma$}} W^T(X).\]
\end{defin}

In this language, Proposition~\ref{li-value} states that
$W^{\vec{e}}(X)_{\text{red}} \subseteq W^{\Gamma(\vec{e})}(X)$.
In fact, we will see later that $W^{\vec{e}}(X) = W^{\Gamma(\vec{e})}(X)$.

\section{Combinatorics}\label{sec:combin}

In the previous section, we classified limit $\vec{e}$-positive line bundles in terms of $k$-staircase tableaux.
Such tableaux are special cases of a more general class of tableaux
known as \defi{$k$-core tableaux}, which are well-studied due to their relationship with the affine symmetric group
(see \cite{lm} and \cite{las}, or for an overview see Section~1.2 of~\cite{ksf}).
To make the paper self-contained, we recall the basic facts about this relationship here (without proof) in the
next two subsections, and
use them to deduce
the structure results for $k$-staircase tableaux that are needed for the proof of the regeneration theorem.  This explicit description of $W^\Gamma(X)$ will also be used directly in the proofs of
all of our main theorems.

\subsection{\boldmath $k$-cores and the affine symmetric group.}
Recall that the \defi{affine symmetric group} $\tilde{S}_k$ is
the group of permutations $f\colon \zz \to \zz$
such that
\[f(x + k) = f(x) + k \quad \text{and} \quad \sum_{x = 1}^k f(x) = \sum_{x = 1}^k x = \frac{k(k + 1)}{2}.\]
Such permutations automatically satisfy
\begin{equation} \label{distinct-A}
f(x) \not\equiv f(y) \pmod{k} \quad \text{for} \quad x \not\equiv y \pmod{k}.
\end{equation}
The affine symmetric group is generated by transpositions $s_j$ (for $j \in \zz/k\zz)$ satisfying
\[s_j(x) = \begin{cases}
x + 1 & \text{if $x \equiv j \mod k$;} \\
x - 1 & \text{if $x \equiv j + 1 \mod k$;} \\
x & \text{otherwise,}
\end{cases}\]
with relations
\[s_j^2 = 1, \quad s_j s_{j'} = s_{j'} s_j \ \text{if $j - j' \neq \pm 1$}, \quad \text{and} \quad (s_j s_{j + 1})^3 = 1.\]
For ease of notation, we include the identity $e = s_*$ as a generator
(so generators are indexed by $\zz/k\zz \cup \{*\}$).

Each line segment making up the boundary of a Young diagram is either vertical or horizontal.  The following key definition generalizes the notion of a $k$-staircase.

\begin{defin}
A sequence $\{\gamma_j\}$ of vertical and horizonal line segments is called \defi{$k$-convex} if $\gamma_j$ is vertical only if $\gamma_{j-k}$ is also vertical.  A Young diagram is called a \defi{$k$-core} if its boundary is $k$-convex.
A ($k$-regular) filling of a $k$-core will be called a \defi{$k$-core tableau}.
\end{defin}

In the literature, $k$-cores are also frequently defined in terms of their
\defi{hook lengths}, which
are the number of boxes to the right or bottom of a given box (including the given box).
Namely, a Young diagram is a $k$-core if and only if no hook lengths are divisible by $k$,
or equivalently if and only if no hook lengths are equal to $k$.

\vspace{10pt}

\begin{center}
\begin{minipage}{.34\textwidth}
\begin{center}
\begin{tikzpicture}[scale=.4]
\draw (0,0) -- (13.5, 0);
\draw (10, 0) -- (10, -.2);
\draw (11, 0) -- (11, -.2);
\draw (12, 0) -- (12, -.2);
\draw (13, 0) -- (13, -.2);
\draw (0,-1) -- (9,-1);
\draw (0,-2) -- (6,-2);
\draw (0,-3) -- (4,-3);
\draw (0,-4) -- (4,-4);
\draw (0,-5) -- (2,-5);
\draw (0,-6) -- (2,-6);
\draw (0,-7) -- (1,-7);
\draw (0,-8) -- (1,-8);
\draw (0,-9) -- (1,-9);
\draw (0,0) -- (0,-10.5);
\draw (0, -10) -- (.2, -10);
\draw (1,0) -- (1,-9);
\draw (2,0) -- (2,-6);
\draw (3,0) -- (3,-4);
\draw (4,0) -- (4,-4);
\draw (5,0) -- (5,-2);
\draw (6,0) -- (6,-2);
\draw (7,0) -- (7,-1);
\draw (8,0) -- (8,-1);
\draw (9,0) -- (9,-1);
{\tiny
\draw[color=white, fill=white] (0, -9.5) circle (8pt);
\draw (0,-9.5)node{3};
\draw[color=white, fill=white] (0.5,-9) circle (8pt);
\draw (0.5,-9)node{\textbf{4}};
\draw[color=white, fill=white] (1,-8.5) circle (8pt);
\draw (1,-8.5)node{1};
\draw[color=white, fill=white] (1,-7.5) circle (8pt);
\draw (1,-7.5)node{2};
\draw[color=white, fill=white] (1,-6.5) circle (8pt);
\draw (1,-6.5)node{{3}};
\draw[color=white, fill=white] (1.5,-6) circle (8pt);
\draw (1.5,-6)node{4};
\draw[color=white, fill=white] (2,-5.5) circle (8pt);
\draw (2,-5.5)node{1};
\draw[color=white, fill=white] (2,-4.5) circle (8pt);
\draw (2,-4.5)node{2};
\draw[color=white, fill=white] (2.5,-4) circle (8pt);
\draw (2.5,-4)node{\textbf{3}};
\draw[color=white, fill=white] (3.5,-4) circle (8pt);
\draw (3.5,-4)node{4};
\draw[color=white, fill=white] (4,-3.5) circle (8pt);
\draw (4,-3.5)node{1};
\draw[color=white, fill=white] (4,-2.5) circle (8pt);
\draw (4,-2.5)node{{2}};
\draw[color=white, fill=white] (4.5,-2) circle (8pt);
\draw (4.5,-2)node{3};
\draw[color=white, fill=white] (5.5,-2) circle (8pt);
\draw (5.5,-2)node{4};
\draw[color=white, fill=white] (6,-1.5) circle (8pt);
\draw (6,-1.5)node{1};
\draw[color=white, fill=white] (6.5,-1) circle (8pt);
\draw (6.5,-1)node{\textbf{2}};
\draw[color=white, fill=white] (7.5,-1) circle (8pt);
\draw (7.5,-1)node{3};
\draw[color=white, fill=white] (8.5,-1) circle (8pt);
\draw (8.5,-1)node{4};
\draw[color=white, fill=white] (9,-.5) circle (8pt);
\draw (9,-.5)node{{1}};
\draw[color=white, fill=white] (9.5,0) circle (8pt);
\draw (9.5,0)node{2};
\draw[color=white, fill=white] (10.5,0) circle (8pt);
\draw (10.5,0)node{3};
\draw[color=white, fill=white] (11.5,0) circle (8pt);
\draw (11.5,0)node{4};
\draw[color=white, fill=white] (12.5,0) circle (8pt);
\draw (12.5,0)node{\textbf{1}};
}
\end{tikzpicture}
\end{center}

\noindent A $4$-staircase 
\noindent is $4$-core.

\end{minipage}
\begin{minipage}{.33\textwidth}
\begin{center}
\begin{tikzpicture}[scale=.4]
\draw (0,0) -- (12.5, 0);
\draw (9, 0) -- (9, -.2);
\draw (10, 0) -- (10, -.2);
\draw (11, 0) -- (11, -.2);
\draw (12, 0) -- (12, -.2);
\draw (0,-1) -- (8,-1);
\draw (0,-2) -- (5,-2);
\draw (0,-3) -- (4,-3);
\draw (0,-4) -- (3,-4);
\draw (0,-5) -- (2,-5);
\draw (0,-6) -- (1,-6);
\draw (0,-7) -- (1,-7);
\draw (0,-8) -- (1,-8);
\draw (0,0) -- (0,-9.5);
\draw (0, -9) -- (.2, -9);
\draw (1,0) -- (1,-8);
\draw (2,0) -- (2,-5);
\draw (3,0) -- (3,-4);
\draw (4,0) -- (4,-3);
\draw (5,0) -- (5,-2);
\draw (6,0) -- (6,-1);
\draw (7,0) -- (7,-1);
\draw (8,0) -- (8,-1);
{\tiny
\draw[color=white, fill=white] (0, -8.5) circle (8pt);
\draw (0,-8.5)node{4};
\draw[color=white, fill=white] (0.5,-8) circle (8pt);
\draw (0.5,-8)node{\textbf{1}};
\draw[color=white, fill=white] (1,-7.5) circle (8pt);
\draw (1,-7.5)node{2};
\draw[color=white, fill=white] (1, -6.5) circle (8pt);
\draw (1, -6.5)node{{3}};
\draw[color=white, fill=white] (1, -5.5) circle (8pt);
\draw (1, -5.5)node{4};
\draw[color=white, fill=white] (1.5, -5) circle (8pt);
\draw (1.5, -5)node{1};
\draw[color=white, fill=white] (2, -4.5) circle (8pt);
\draw (2, -4.5)node{2};
\draw[color=white, fill=white] (2.5, -4) circle (8pt);
\draw (2.5, -4)node{\textbf{3}};
\draw[color=white, fill=white] (3, -3.5) circle (8pt);
\draw (3, -3.5)node{4};
\draw[color=white, fill=white] (3.5, -3) circle (8pt);
\draw (3.5, -3)node{1};
\draw[color=white, fill=white] (4, -2.5) circle (8pt);
\draw (4, -2.5)node{{2}};
\draw[color=white, fill=white] (4.5, -2) circle (8pt);
\draw (4.5, -2)node{3};
\draw[color=white, fill=white] (5, -1.5) circle (8pt);
\draw (5, -1.5)node{4};
\draw[color=white, fill=white] (5.5, -1) circle (8pt);
\draw (5.5, -1)node{1};
\draw[color=white, fill=white] (6.5, -1) circle (8pt);
\draw (6.5, -1)node{\textbf{2}};
\draw[color=white, fill=white] (7.5, -1) circle (8pt);
\draw (7.5, -1)node{3};
\draw[color=white, fill=white] (8, -.5) circle (8pt);
\draw (8, -.5)node{{4}};
\draw[color=white, fill=white] (8.5,0) circle (8pt);
\draw (8.5,0)node{1};
\draw[color=white, fill=white] (9.5,0) circle (8pt);
\draw (9.5,0)node{2};
\draw[color=white, fill=white] (10.5,0) circle (8pt);
\draw (10.5,0)node{3};
\draw[color=white, fill=white] (11.5,0) circle (8pt);
\draw (11.5,0)node{\textbf{4}};
}
\end{tikzpicture}
\end{center}

\noindent Another $4$-core that is \\
\noindent  not a $4$-staircase.

\end{minipage}
\begin{minipage}{.28\textwidth}
\begin{center}
\begin{tikzpicture}[scale=.4]
\draw (0,0) -- (9, 0);
\draw (0,-1) -- (8,-1);
\draw (0,-2) -- (4,-2);
\draw (0,-3) -- (3,-3);
\draw (0,-4) -- (3,-4);
\draw (0,-5) -- (1,-5);
\draw (0,-6) -- (1,-6);
\draw (0,-7) -- (1,-7);
\draw (0,-8) -- (1,-8);
\draw (0,0) -- (0,-9);
\draw (1,0) -- (1,-8);
\draw (2,0) -- (2,-4);
\draw (3,0) -- (3,-4);
\draw (4,0) -- (4,-2);
\draw (5,0) -- (5,-1);
\draw (6,0) -- (6,-1);
\draw (7,0) -- (7,-1);
\draw (8,0) -- (8,-1);
{\tiny
\draw[color=white, fill=white] (0.5,-8) circle (8pt);
\draw (0.5,-8)node{1};
\draw[color=white, fill=white] (1,-7.5) circle (8pt);
\draw (1,-7.5)node{2};
\draw[color=white, fill=white] (1, -6.5) circle (8pt);
\draw (1, -6.5)node{3};
\draw[color=white, fill=white] (1, -5.5) circle (8pt);
\draw (1, -5.5)node{4};
\draw[color=white, fill=white] (1, -4.5) circle (8pt);
\draw (1, -4.5)node{1};
\draw[color=white, fill=white] (1.5, -4) circle (8pt);
\draw (1.5, -4)node{2};
\draw[color=white, fill=white] (2.5, -4) circle (8pt);
\draw (2.5, -4)node{3};
\draw[color=white, fill=white] (3, -3.5) circle (8pt);
\draw (3, -3.5)node{4};
\draw[color=white, fill=white] (3, -2.5) circle (8pt);
\draw (3, -2.5)node{1};
\draw[color=white, fill=white] (3.5, -2) circle (8pt);
\draw (3.5, -2)node{2};
\draw[color=white, fill=white] (4, -1.5) circle (8pt);
\draw (4, -1.5)node{3};
\draw[color=white, fill=white] (4.5, -1) circle (8pt);
\draw (4.5, -1)node{4};
\draw[color=white, fill=white] (5.5, -1) circle (8pt);
\draw (5.5, -1)node{1};
\draw[color=white, fill=white] (6.5, -1) circle (8pt);
\draw (6.5, -1)node{2};
\draw[color=white, fill=white] (7.5, -1) circle (8pt);
\draw (7.5, -1)node{3};
\draw[color=white, fill=white] (8, -.5) circle (8pt);
\draw (8, -.5)node{4};

\draw[very thick, color=purple, <->] (.5, -7.5) -- (.5, -2.5) -- (2.5, -2.5);
}
\end{tikzpicture}
\end{center}

\noindent A diagram that is not \\
\noindent a $4$-core.
\end{minipage}
\end{center} 

\vspace{10pt}

A sequence $\{\gamma_j\}$ is $k$-convex if each residue class of segments is composed
of an infinite sequence of vertical segments followed by an infinite sequence of horizontal segments.
Thus, to specify a $k$-core, it suffices to give a collection $\{t_1, \dots, t_k\}$
of integers (distinct mod $k$),
representing the first horizontal segment in each residue class.
Such data is a priori determined up to
addition of an overall constant (i.e.\ $\{t_j\} \mapsto \{t_j + \delta\}$);
the indexing of boundary segments in Definition~\ref{yd-b}
corresponds to the unique normalization so that
\[\sum_{j=1}^k t_j = \sum_{j =1}^k j = \frac{k(k + 1)}{2}.\]
Therefore, $k$-cores are in bijection with elements of $\tilde{S}_k / S_k$ --- by
sending $\{t_j\}$ to the coset of permutations sending $\{1, 2, \ldots, k\}$ to $\{t_1, t_2, \ldots, t_k\}$.
There is a distinguished coset representative $f$ satisfying $f(1) < f(2) < \cdots < f(k)$.

\begin{defin} \label{Tround}
If $\Gamma$ is a $k$-core and $x \in \zz$, we define $\Gamma(x)$ to be the value
of this distinguished permutation applied to $x$;
if $T$ is a $k$-core tableau of shape $\Gamma$, we define $T(x) = \Gamma(x)$.
\end{defin}

In the definition of a $k$-convex sequence $\{\gamma_j\}$, we could equivalently have considered pairs of adjacent line segments $(\gamma_j, \gamma_{j+1})$.  
Then, the mod $k$ residue class of pairs of boundary segments
\[ \{ \dots, (\gamma_j, \gamma_{j+1}), (\gamma_{j+k}, \gamma_{j+k+1}), \dots \} \]
is composed of a sequence of (vertical, vertical) segments, followed by a (possibly empty) sequence of \emph{either} (vertical, horizontal) \emph{or} (horizontal, vertical) corners, followed by a sequence of (horizontal, horizontal) segments.
In other words,
the mod $k$ residue classes of pair of boundary segments in a $k$-core always progress along \emph{one} of the following trajectories: 

\vspace{5pt}

\begin{center}
\begin{tikzpicture}[scale = .5]

\draw (1,-1) -- (1,0) -- (1,1); 
\filldraw(1,-1) circle (2pt);
\filldraw(1,0) circle (2pt);
\filldraw(1,1) circle (2pt);

\draw[->] (1.5,.5) -- (3.5,1.75);

\draw (4,2) -- (5,2) -- (5,3); 
\filldraw(4,2) circle (2pt);
\filldraw(5,2) circle (2pt);
\filldraw(5,3) circle (2pt);

\draw (4.5, 3.8)node{{\small removable}};

\draw[->] (1.5,-.5) -- (3.5,-1.75);

\draw (4,-3) -- (4,-2) -- (5,-2); 
\filldraw(4,-3) circle (2pt);
\filldraw(4,-2) circle (2pt);
\filldraw(5,-2) circle (2pt);

\draw (4.5, -3.8)node{{\small addable}};

\draw[->] (5.5, 1.75) -- (7.5, .5);

\draw[->] (5.5, -1.75) -- (7.5, -.5);

\draw (8,0) -- (9,0) -- (10,0); 
\filldraw(8,0) circle (2pt);
\filldraw(9,0) circle (2pt);
\filldraw(10,0) circle (2pt);

\end{tikzpicture}
\end{center}
The configuration of a vertical and then horizontal segment is called
an \defi{addable corner}, and the configuration of a horizontal and then vertical segment is called a \defi{removable corner}.

This gives a natural (left) action of the affine symmetric group $\tilde{S}_k$ on the set of $k$-cores.
Namely, $s_j \cdot \Gamma$ is the $k$-core obtained from $\Gamma$
by adding a box in all addable corners whose diagonal index has residue class $j$
(if such addable corners exist),
or removing a box from all removable corners whose diagonal index has residue class $j$
(if such removable corners exist),
or doing nothing (if no such addable or removable corners exist).

\vspace{20pt}
\begin{minipage}{.4\textwidth}
\begin{center}
\begin{tikzpicture}
{\small 
\draw (0,0) -- (4, 0);
\draw (0,-1) -- (4,-1);
\draw (0,-2) -- (2,-2);
\draw (0,-3) -- (1,-3);
\draw (0,-4) -- (1,-4);
\draw (0,0) -- (0,-4);
\draw (1,0) -- (1,-4);
\draw (2,0) -- (2,-2);
\draw (3,0) -- (3,-1);
\draw (4,0) -- (4,-1);

\draw[dashed] (.5,.5) -- ( 3.5,-2.5);
 \node[right, rotate=-45] at ( 3.5,-2.5) {$1$};
\draw[dashed] (1.5,.5) -- ( 4,-2);
 \node[right, rotate=-45] at ( 4,-2) {$2$};
\draw[dashed] (2.5,.5) -- ( 4.5,-1.5);
 \node[right, rotate=-45] at ( 4.5,-1.5) {$3$};
\draw[dashed] (3.5,.5) -- ( 5,-1);
 \node[right, rotate=-45] at ( 5,-1) {$1$};

\draw[dashed] (-.5,.5) -- ( 3,-3);
 \node[right, rotate=-45] at ( 3,-3) {$3$};
\draw[dashed] (-.5,-.5) -- ( 2.5,-3.5);
 \node[right, rotate=-45] at ( 2.5,-3.5) {$2$};
\draw[dashed] (-.5,-1.5) -- ( 2,-4);
 \node[right, rotate=-45] at ( 2,-4) {$1$};
\draw[dashed] (-.5,-2.5) -- ( 1.5,-4.5);
 \node[right, rotate=-45] at ( 1.5,-4.5) {$3$};
\draw[dashed] (-.5,-3.5) -- ( 1,-5);
 \node[right, rotate=-45] at ( 1,-5) {$2$};
}
\end{tikzpicture}

A $k$-core diagram $\Gamma$.
\end{center}
\end{minipage}
\begin{minipage}{.55\textwidth}
{\small 
\begin{center}
\begin{tikzpicture}[scale=.6]
\node at (-1.5, -2) {$s_1 \cdot \Gamma=$};
\draw (0,0) -- (4, 0);
\draw (0,-1) -- (4,-1);
\draw (0,-2) -- (2,-2);
\draw (0,-3) -- (1,-3);
\draw (0,-4) -- (1,-4);
\draw (0,0) -- (0,-4);
\draw (1,0) -- (1,-4);
\draw (2,0) -- (2,-2);
\draw (3,0) -- (3,-1);
\draw (4,0) -- (4,-1);

\draw[thick, color=violet, fill=violet!10] (4,0) -- (5,0) -- (5, -1) -- (4, -1) -- (4,0);
\draw[thick, color=violet, fill=violet!10] (2,-1) -- (3, -1) -- (3, -2) -- (2, -2) -- (2, -1);

\draw[dashed] (.5,.5) -- ( 3.5,-2.5);
 \node[right, rotate=-45] at ( 3.5,-2.5) {$1$};
 \draw[dashed] (3.5,.5) -- ( 5,-1);
 \node[right, rotate=-45] at ( 5,-1) {$1$};
 \draw[dashed] (-.5,-1.5) -- ( 2,-4);
 \node[right, rotate=-45] at ( 2,-4) {$1$};
\end{tikzpicture}
\end{center}

\vspace{5pt}

\begin{center}
\begin{tikzpicture}[scale=.6]
\node at (-1.5, -2) {$s_2 \cdot \Gamma=$};
\draw (0,0) -- (4, 0);
\draw (0,-1) -- (4,-1);
\draw (0,-2) -- (2,-2);
\draw (0,-3) -- (1,-3);
\draw (0,-4) -- (1,-4);
\draw (0,0) -- (0,-4);
\draw (1,0) -- (1,-4);
\draw (2,0) -- (2,-2);
\draw (3,0) -- (3,-1);
\draw (4,0) -- (4,-1);

\draw[thick, color=violet, fill=violet!10] (1, -2) -- (2, -2) -- (2, -3) -- (1, -3) -- (1, -2);
\draw[thick, color=violet, fill=violet!10] (0,-4) -- (1,-4) -- (1, -5) -- (0, -5) -- (0, -4);

\draw[dashed] (1.5,.5) -- ( 4,-2);
 \node[right, rotate=-45] at ( 4,-2) {$2$};
 \draw[dashed] (-.5,-.5) -- ( 2.5,-3.5);
 \node[right, rotate=-45] at ( 2.5,-3.5) {$2$};
 \draw[dashed] (-.5,-3.5) -- ( 1,-5);
 \node[right, rotate=-45] at ( 1,-5) {$2$};
\end{tikzpicture}
\end{center}

\vspace{5pt}

\begin{center}
\begin{tikzpicture}[scale=.6]
\node at (-1.5, -2) {$s_3 \cdot \Gamma =$};
\draw (0,0) -- (3, 0);
\draw (0,-1) -- (3,-1);
\draw (0,-2) -- (1,-2);
\draw (0,-3) -- (1,-3);
\draw (0,0) -- (0,-3);
\draw (1,0) -- (1,-3);
\draw (2,0) -- (2,-1);
\draw (3,0) -- (3,-1);

\draw[color=red, dotted, pattern=north west lines, pattern color=red] (3, 0) -- (4, 0) -- (4, -1) -- (3, -1);
\draw[color=red, dotted, pattern=north west lines, pattern color=red] (1, -1) -- (2, -1) -- (2, -2) -- (1, -2);
\draw[color=red, dotted, pattern=north west lines, pattern color=red] (1, -3) -- (1, -4) -- (0, -4) -- (0, -3);

\draw[dashed] (2.5,.5) -- ( 4.5,-1.5);
 \node[right, rotate=-45] at ( 4.5,-1.5) {$3$};
 \draw[dashed] (-.5,.5) -- ( 3,-3);
 \node[right, rotate=-45] at ( 3,-3) {$3$};
\draw[dashed] (-.5,-2.5) -- ( 1.5,-4.5);
 \node[right, rotate=-45] at ( 1.5,-4.5) {$3$};
\end{tikzpicture}
\end{center}
}
\end{minipage}

One easily checks that this respects the relations for the affine symmetric group, and that under
this action,
\[\{\text{$k$-core diagrams}\} \leftrightarrow \tilde{S}_k / S_k\]
is an $\tilde{S}_k$-equivariant bijection of sets.

\subsection{\boldmath $k$-core tableaux and the word problem\label{tab-word}}

Given a $k$-core $\Gamma$, let $w \in \tilde{S}_k$ be the
representative of the corresponding coset with $w(1) < w(2) < \cdots < w(k)$.
If $w = s_{j_g} s_{j_{g-1}} \cdots s_{j_1}$ is a word for $w$ in $\tilde{S}_k$,
then we obtain a filling of $\Gamma$: Indeed, we build $\Gamma$ from the empty $k$-core
by consecutively applying the $s_{j_i}$; this determines
a $k$-core tableau of shape $\Gamma$ where
any box added in the $i$th step contains the symbol $i$.

\begin{center}
\begin{minipage}{.48\textwidth}
\begin{center}
\begin{tikzpicture}[scale=.75]
\draw (0,0) -- (4, 0);
\draw (0,-1) -- (4,-1);
\draw (0,-2) -- (2,-2);
\draw (0,-3) -- (1,-3);
\draw (0,-4) -- (1,-4);
\draw (0,0) -- (0,-4);
\draw (1,0) -- (1,-4);
\draw (2,0) -- (2,-2);
\draw (3,0) -- (3,-1);
\draw (4,0) -- (4,-1);

\draw (0.5,-0.5) node{1};
\draw (1.5,-0.5) node{3};
\draw (2.5,-0.5) node{4};
\draw (3.5,-0.5) node{5};
\draw (0.5,-1.5) node{2};
\draw (1.5,-1.5) node{5};
\draw (0.5,-2.5) node{3};
\draw (0.5,-3.5) node{5};
{\tiny
\draw[dotted] (.5,.5) -- ( 3.5,-2.5);
 \node[right, rotate=-45] at ( 3.5,-2.5) {$1$};
\draw[dotted] (1.5,.5) -- ( 4,-2);
 \node[right, rotate=-45] at ( 4,-2) {$2$};
\draw[dotted] (2.5,.5) -- ( 4.5,-1.5);
 \node[right, rotate=-45] at ( 4.5,-1.5) {$3$};
\draw[dotted] (3.5,.5) -- ( 5,-1);
 \node[right, rotate=-45] at ( 5,-1) {$1$};

\draw[dotted] (-.5,.5) -- ( 3,-3);
 \node[right, rotate=-45] at ( 3,-3) {$3$};
\draw[dotted] (-.5,-.5) -- ( 2.5,-3.5);
 \node[right, rotate=-45] at ( 2.5,-3.5) {$2$};
\draw[dotted] (-.5,-1.5) -- ( 2,-4);
 \node[right, rotate=-45] at ( 2,-4) {$1$};
\draw[dotted] (-.5,-2.5) -- ( 1.5,-4.5);
 \node[right, rotate=-45] at ( 1.5,-4.5) {$3$};
\draw[dotted] (-.5,-3.5) -- ( 1,-5);
 \node[right, rotate=-45] at ( 1,-5) {$2$};
}

\node at (2, -6.5) {$w = s_3s_2s_1s_2s_3$};
\end{tikzpicture}
\end{center}
\end{minipage}
\begin{minipage}{.48\textwidth}
\begin{tikzpicture}[scale=.75]
\draw (0,0) -- (4, 0);
\draw (0,-1) -- (4,-1);
\draw (0,-2) -- (2,-2);
\draw (0,-3) -- (1,-3);
\draw (0,-4) -- (1,-4);
\draw (0,0) -- (0,-4);
\draw (1,0) -- (1,-4);
\draw (2,0) -- (2,-2);
\draw (3,0) -- (3,-1);
\draw (4,0) -- (4,-1);

\draw (0.5,-0.5) node{1};
\draw (1.5,-0.5) node{2};
\draw (2.5,-0.5) node{3};
\draw (3.5,-0.5) node{5};
\draw (0.5,-1.5) node{3};
\draw (1.5,-1.5) node{5};
\draw (0.5,-2.5) node{4};
\draw (0.5,-3.5) node{5};
{\tiny
\draw[dotted] (.5,.5) -- ( 3.5,-2.5);
 \node[right, rotate=-45] at ( 3.5,-2.5) {$1$};
\draw[dotted] (1.5,.5) -- ( 4,-2);
 \node[right, rotate=-45] at ( 4,-2) {$2$};
\draw[dotted] (2.5,.5) -- ( 4.5,-1.5);
 \node[right, rotate=-45] at ( 4.5,-1.5) {$3$};
\draw[dotted] (3.5,.5) -- ( 5,-1);
 \node[right, rotate=-45] at ( 5,-1) {$1$};

\draw[dotted] (-.5,.5) -- ( 3,-3);
 \node[right, rotate=-45] at ( 3,-3) {$3$};
\draw[dotted] (-.5,-.5) -- ( 2.5,-3.5);
 \node[right, rotate=-45] at ( 2.5,-3.5) {$2$};
\draw[dotted] (-.5,-1.5) -- ( 2,-4);
 \node[right, rotate=-45] at ( 2,-4) {$1$};
\draw[dotted] (-.5,-2.5) -- ( 1.5,-4.5);
 \node[right, rotate=-45] at ( 1.5,-4.5) {$3$};
\draw[dotted] (-.5,-3.5) -- ( 1,-5);
 \node[right, rotate=-45] at ( 1,-5) {$2$};
}
\node at (2, -6.5) {$w = s_3s_1s_2s_1s_3$};
\end{tikzpicture}
\end{minipage}

\vspace{5pt}

\noindent The two efficient fillings of this $3$-staircase diagram.
\end{center}

Lapointe and Morse showed in \cite{lm} that this completely describes efficiently filled $k$-core tableaux.
Namely:
\begin{enumerate}
\item Any efficient filling (i.e.\ with the fewest possible symbols)
arises in this way from a unique \defi{reduced word} for $w$ (i.e.\ a word with the fewest possible
non-identity generators). Conversely, any reduced word
gives an efficient filling
(and if the word is reduced then no boxes are ever removed). See Section~8 of \cite{lm}.
\item The minimal number of symbols needed to fill $\Gamma$, which we will denote $u(\Gamma)$, is exactly
the number of boxes in $\Gamma$
whose hook length is less than $k$. See Lemma 31 of \cite{lm}.
\item \label{eff} Efficient fillings can be constructed inductively:
Suppose $\Gamma$ has a removable corner whose diagonal index has residue class $j$,
so that $s_j \cdot \Gamma$ is strictly contained in $\Gamma$.
Then we have $u(s_j \cdot \Gamma) = u(\Gamma) - 1$.
See Proposition~22 of \cite{lm}. In particular:
\begin{enumerate}
\item \label{eff-res} An efficient filling of $\Gamma$ whose largest symbol
appears in a box with diagonal index of residue $j$
restricts to an efficient filling of $s_j \cdot \Gamma$.
\item \label{eff-com} An efficient filling of $s_j \cdot \Gamma$
can be completed to an efficient filling of $\Gamma$ whose largest symbol
appears in a box with diagonal index of residue $j$.
\end{enumerate}
\end{enumerate}

\subsection{Reduction to efficient tableaux}
One consequence of this final property (\ref{eff-com}) is that we need only consider
efficient tableaux for our geometric problem.

\begin{prop} \label{efficient-only}
Let $T$ be a $k$-core tableau of shape $\Gamma$.
Then there is an efficiently filled $k$-core tableau $T'$ of shape $\Gamma$
with $W^T(X) \subseteq W^{T'}(X)$. In particular,
\[W^\Gamma(X) = \bigcup_{\text{$T$ \textbf{efficient}} \atop \text{filling of $\Gamma$}} W^T(X).\]
\end{prop}
\begin{proof}
We argue by induction on $u(\Gamma)$; the base case $u(\Gamma) = 0$ is tautological.
For the inductive step, let $t$ be the largest symbol appearing in $T$, and $j = T[t]$ (c.f. \ Definition~\ref{Tsquare}).
Let $T_\circ$ be the restriction of $T$ to $s_j \cdot \Gamma$. By our inductive hypothesis,
there is an efficient filling $T'_\circ$ of $s_j \cdot \Gamma$ with $W^{T_\circ}(X) \subseteq W^{T_\circ'}(X)$.
If $T'$ is the completion of $T'_\circ$ to a filling of $\Gamma$ using the additional symbol $t$,
then $W^T(X) \subseteq W^{T'}(X)$ as desired.
\end{proof}

\subsection{Truncations}

The following will be a convenient way of packaging the data of an efficient filling
as necessary for our regeneration theorem.

\begin{defin} \label{trunc_func}
Let $T$ be an efficiently filled $k$-core tableau, corresponding to a reduced word
$s_{j_g} \cdots s_{j_2} s_{j_1}$.
Define $T^{\leq t}$
to be the tableau formed by the boxes of $T$ with symbols up to $t$,
i.e.\ corresponding to the reduced word $s_{j_t} \cdots s_{j_2} s_{j_1}$.

In particular, for each $t$ and $\ell$, we obtain an integer
which we refer to as the $\ell$th \defi{truncation} at time $t$:
\[T^{\leq t}(\ell) = (s_{j_t} \cdots s_{j_2} s_{j_1})(\ell).\]
\end{defin}

\noindent
We now summarize several properties of the $T^{\leq t}(\ell)$. First of all,
by construction, we have:
\begin{equation} \label{atzero}
T^{\leq 0}(\ell)=\ell.
\end{equation}
Moreover, by \eqref{distinct-A},
\begin{equation} \label{distinct}
T^{\leq t}(\ell_1) \not\equiv T^{\leq t}(\ell_2) \pmod{k} \quad \text{for} \quad \ell_1 \not\equiv \ell_2 \pmod{k}.
\end{equation}

If $s_{j_t}$ is the identity (equivalently if $T[t] = *$), then $T^{\leq t}(\ell) = T^{\leq t - 1}(\ell)$ for all $\ell$.
Otherwise, $s_{j_t}$ is a simple transposition, and there are exactly two values of $\ell \in \{1, 2, \ldots, k\}$,
say $\ell_-$ and $\ell_+$,
for which $T^{\leq t}(\ell)$ changes:
\begin{equation} \label{swap}
T^{\leq t}(\ell) = \begin{cases}
T^{\leq t - 1}(\ell) & \text{if $\ell \neq \ell_\pm$;} \\
T^{\leq t - 1}(\ell_\pm) \pm 1 & \text{if $\ell = \ell_\pm$.}
\end{cases}
\end{equation}
In this case, we say that $t$ is \defi{increasing} for $\ell_+$
and \defi{decreasing} for $\ell_-$.
Combining \eqref{distinct} and \eqref{swap}, we have
\begin{equation} \label{swapmod}
T^{\leq t}(\ell_+) - T^{\leq t}(\ell_-) \equiv T^{\leq t - 1}(\ell_-) - T^{\leq t - 1}(\ell_+) \equiv 1 \pmod{k}.
\end{equation}
The following pictures illustrate the behavior of the $T^{\leq t}(\ell)$ at fixed time $t$,
and the relation between times $t - 1$ and $t$:

\begin{minipage}{.49\textwidth}
\begin{center}
\begin{tikzpicture}[scale=.55]
{\small

\draw [fill=violet!10] (0, -6) -- (1, -6) -- (1, -3) -- (3, -3) -- (3, -1) -- (6, -1) -- (6, 0) -- (0, 0) ;
\draw [very thick, color=violet] (0, -9) -- (0, -6) -- (1, -6) -- (1, -3) -- (3, -3) -- (3, -1) -- (6, -1) -- (6, 0) -- (10,0);

\draw (0,0) -- (9, 0);
\draw (0,-1) -- (9,-1);
\draw (0,-2) -- (6,-2);
\draw (0,-3) -- (4,-3);
\draw (0,-4) -- (4,-4);
\draw (0,-5) -- (2,-5);
\draw (0,-6) -- (2,-6);
\draw (0,-7) -- (1,-7);
\draw (0,-8) -- (1,-8);
\draw (0,-9) -- (1,-9);
\draw (0,0) -- (0,-9);
\draw (1,0) -- (1,-9);
\draw (2,0) -- (2,-6);
\draw (3,0) -- (3,-4);
\draw (4,0) -- (4,-4);
\draw (5,0) -- (5,-2);
\draw (6,0) -- (6,-2);
\draw (7,0) -- (7,-1);
\draw (8,0) -- (8,-1);
\draw (9,0) -- (9,-1);

\draw (0.5,-0.5) node{1};
\draw (1.5,-0.5) node{2};
\draw (2.5,-0.5) node{4};
\draw (3.5,-0.5) node{6};
\draw (4.5,-0.5) node{7};
\draw (5.5,-0.5) node{9};
\draw (6.5,-0.5) node{11};
\draw (7.5,-0.5) node{12};
\draw (8.5,-0.5) node{14};
\draw (0.5,-1.5) node{3};
\draw (1.5,-1.5) node{7};
\draw (2.5,-1.5) node{9};
\draw (3.5,-1.5) node{11};
\draw (4.5,-1.5) node{12};
\draw (5.5,-1.5) node{14};
\draw (0.5,-2.5) node{4};
\draw (1.5,-2.5) node{8};
\draw (2.5,-2.5) node{10};
\draw (3.5,-2.5) node{13};
\draw (0.5,-3.5) node{5};
\draw (1.5,-3.5) node{11};
\draw (2.5,-3.5) node{12};
\draw (3.5,-3.5) node{14};
\draw (0.5,-4.5) node{7};
\draw (1.5,-4.5) node{13};
\draw (0.5,-5.5) node{8};
\draw (1.5,-5.5) node{14};
\draw (0.5,-6.5) node{11};
\draw (0.5,-7.5) node{13};
\draw (0.5,-8.5) node{14};

\draw (1, -10) -- (11,0);

\draw[line width=1mm, color=blue] (0, -6) -- (1, -6);

\draw[line width=1mm, color=blue] (2, -3) -- (3, -3);

\draw[line width=1mm, color=blue] (5, -1) -- (6, -1);

\draw[line width=1mm, color=blue] (9, 0) -- (10, 0);

\draw[dotted, color=blue] (-.5, -4.5) -- (3, -8);
\draw[dotted, color=blue] (-.5, .5) -- (5.5, -5.5);
\draw[dotted, color=blue] (4.5, .5) -- (8, -3);
\draw[dotted, color=blue] (9.5, .5) -- (10.5, -.5);
}

{\small
\node at (3, -8) {{\color{blue}$\bullet$}};
\node[below right] at (3, -8) {{\color{violet}$T^{\leq 10}$}$(1)=${\color{blue}$-5$}};
\node at (5.5, -5.5) {{\color{blue}$\bullet$}};
\node[below right] at (5.5, -5.5) {{\color{violet}$T^{\leq 10}$}$(2)=${\color{blue}$0$}};
\node at (8, -3) {{\color{blue}$\bullet$}};
\node[below right] at (8, -3) {{\color{violet}$T^{\leq 10}$}$(3)=${\color{blue}$5$}};
\node at (10.5, -.5) {{\color{blue}$\bullet$}};
\node[below right] at (10.5, -.5) {{\color{violet}$T^{\leq 10}$}$(4)=${\color{blue}$10$}};
}

\node at (3, .75) {{\color{violet} $T^{\leq 10}$}};

\end{tikzpicture}
\end{center}
\end{minipage}
\begin{minipage}{.48\textwidth}
\begin{center}
\begin{tikzpicture}[scale=.55]

{\small
\draw (0,0) -- (8, 0);
\draw (5,-1) -- (8,-1);
\draw (3,-2) -- (5,-2);
\draw (1,-4) -- (3,-4);
\draw (0,-7) -- (1,-7);
\draw (0,0) -- (0,-8);
\draw (1,-4) -- (1,-7);
\draw (3,-2) -- (3,-4);
\draw (5,-1) -- (5,-2);
\draw (8,0) -- (8,-1);

\draw (1.5, -10.5) -- (10, -2);

\draw[densely dotted, fill=magenta!10] (1, -4) -- (1, -5) -- (2, -5) -- (2 , -4);

\draw[densely dotted, fill=magenta!10]   (3,-2) -- (4,-2) -- (4,-3) -- (3,-3) -- (3,-2);

\draw[densely dotted, fill=magenta!10]   (0,-7) -- (1,-7) -- (1,-8) -- (0,-8) -- (0,-7);

\draw[dotted, color=magenta!70!black] (-.5, -6.5) -- (2.5, -9.5);
\draw[dotted, color=green!70!black] (-.5, -5.5) -- (3, -9);
\draw[dotted] (-.5, -2.5) -- (4.5, -7.5);
\draw[dotted] (.5, .5) -- (6.5, -5.5);
\draw[dotted, color=green!70!black] (4.5, .5) -- ( 8.5, -3.5);
\draw[dotted, color=magenta!70!black] (5.5, .5) -- ( 9, -3);

\node[ left] at (-.5, -6.5) {{\color{magenta!70!black}$T[t]$}};
\node[ left] at (-.5, -2.5)  {{\color{magenta!70!black}$T[t]$}};
\node[ left] at (.5, .5)  {{\color{magenta!70!black}$T[t]$}};
\node[ left] at (4.5, .5)  {{\color{magenta!70!black}$T[t]$}};

\node at (2.5, -9.5) {{\color{magenta!70!black}$\bullet$}};
\node at (3, -9) {{\color{green!70!black}$\bullet$}};
\node at ( 8.5, -3.5) {{\color{green!70!black}$\bullet$}};
\node at ( 9, -3) {{\color{magenta!70!black}$\bullet$}};

\node[below right] at (2.5, -9.5) {{\color{magenta!70!black}$T^{\leq t}(\ell_-)$}};
\node[below right] at (3, -8.7) {{\color{green!70!black}$T^{\leq t-1}(\ell_-)$}};
\node[below right] at ( 8.5, -3.5) {{\color{green!70!black}$T^{\leq t-1}(\ell_+)$}};
\node[below right] at ( 9, -2.7) {{\color{magenta!70!black}$T^{\leq t}(\ell_+)$}};

\draw[thick] (0,-9) -- (0, -7) -- (1, -7) -- (1, -4) -- (3, -4) -- (3, -2) -- (4, -2) -- (5, -2) -- (5, -1)  -- (8, -1) -- (8, 0) -- (9,0);

\draw[line width=1mm, color=green!60!black] (0, -7) -- (1, -7);

\draw[line width=1mm, color=magenta!70!black] (0, -8) -- (1, -8);

\draw[line width=1mm, color=green!60!black] (5, -1) -- (6, -1);

\draw[line width=1mm, color=magenta!70!black] (6, -1) -- (7, -1);
}
\draw (3.5,-2.5) node{{\color{magenta!70!black}$t$}};
\draw (1.5,-4.5) node{{\color{magenta!70!black}$t$}};
\draw (0.5,-7.5) node{{\color{magenta!70!black}$t$}};

\end{tikzpicture}
\end{center}
\end{minipage}

\noindent As can be seen in the diagram above, we have the relation
\begin{equation}
T[t] \equiv T^{\leq t}(\ell_-) \pmod{k}.
\end{equation}
As can be seen in the right diagram,
$T^{\leq t - 1}(\ell_-)$ is an edge of the leftmost addable corner whose diagonal index has residue class $T[t]$,
while $T^{\leq t - 1}(\ell_+)$
lies to the right of the rightmost addable corner whose diagonal index has residue class $T[t]$.
Thus,
\begin{equation} \label{sort-0}
T^{t-1}(\ell_-) \leq T^{t-1}(\ell_+) - (k - 1) < T^{t-1}(\ell_+) \quad \text{and} \quad T^{t}(\ell_-) \leq T^{t}(\ell_+) - (k + 1) < T^{t}(\ell_+).
\end{equation}

\begin{prop}\label{main_ants}
If $\ell_1 > \ell_2$ then
\begin{equation}\label{metric}
\left\lfloor \frac{T^{\leq t}(\ell_1) - T^{\leq t}(\ell_2)}{k} \right\rfloor 
\end{equation}
is a non-decreasing function of $t$.
In particular, the truncations are ``sorted,"
i.e.
\[T^{\leq t}(\ell_1) > T^{\leq t}(\ell_2) \quad \text{if} \quad \ell_1 > \ell_2.\]
\end{prop}
\begin{proof}
We will prove this by induction on $t$. From \eqref{distinct} and \eqref{swap}, we have
\[\left\lfloor \frac{T^{\leq t}(\ell_1) - T^{\leq t}(\ell_2)}{k} \right\rfloor  = \left\lfloor \frac{T^{\leq t-1}(\ell_1) - T^{\leq t-1}(\ell_2)}{k} \right\rfloor + \begin{cases}
1 & \text{if $(\ell_1, \ell_2) = (\ell_+, \ell_-)$;} \\
-1 & \text{if $(\ell_1, \ell_2) = (\ell_-, \ell_+)$;} \\
0 & \text{otherwise.}
\end{cases}\]
It thus remains to see that $\ell_- < \ell_+$. But this follows from
\eqref{sort-0}, given our inductive hypothesis
that the truncations are sorted.
\end{proof}

\subsection{\boldmath $k$-staircases}
Let $\vec{e}$ be a splitting type;
write $d_1 > d_2 > \cdots > d_s$ for the distinct parts of $\vec{e}$,
and $m_1, m_2, \ldots, m_s$ for the corresponding multiplicities. (Note that $e_1 \leq e_2 \leq \cdots \leq e_k$ but we have $d_1 > d_2 > \cdots > d_\s$!)
The integers $1, 2, \ldots, k$ are then naturally in bijection with pairs $(j, n)$ with $1 \leq j \leq s$
and $1 \leq n \leq m_j$
(via lexicographic order).

\begin{prop} Every $k$-staircase is a $k$-core, and $u(\Gamma(\vec{e})) = u(\vec{e})$.
\end{prop}
\begin{proof}
Write $n(m) \colonequals \# \{\ell : e_\ell \geq -m \}$. The $k$-staircase $\Gamma(\vec{e})$ has the following form:
\begin{center}
\begin{tikzpicture}[scale=.55]

\draw (0,0) -- (0, -9);
\node[rotate=90] at (0, -9.5) {$\dots$};
\node[rotate=90] at (2, -9.5) {$\dots$};
\draw (0,0) -- (14, 0);
\node at (14.65, 0) {$\dots$};
\node at (14.65, -2) {$\dots$};
\draw (14, -2) -- (13, -2) -- (13, -3) -- (8, -3) -- (8, -5) -- (4, -5) -- (4, -8) -- (2, -8) -- (2, -9);

\draw[draw=none, fill=teal!5] (2, -8) -- (2, -5) -- (4, -5) -- (4, -8);
\draw[draw=none, fill=teal!5] (4, -5) -- (4, -3) -- (8, -3) -- (8, -5);
\draw[draw=none, fill=teal!5] (8, -3) -- (8, -2) -- (13, -2) -- (13, -3);

\draw[dotted, fill=teal!5] (2, -8) -- (2, -5) -- (4, -5);
\draw[dotted, fill=teal!5] (4, -5) -- (4, -3) -- (8, -3);
\draw[dotted, fill=teal!5] (8, -3) -- (8, -2) -- (13, -2);
{\tiny
\node at (3, -6.5) {$h< k$};
\node at (6, -4) {$h< k$};
\node at (10.5, -2.5) {$h< k$};

\node at (3.1, -8.3) {$n(m-1)$};
\node at (6, -5.3) {$n(m)$};
\node at (10.5, -3.3) {$n(m+1)$};

\node[right] at (4, -6.5) {$k - n(m)$};
\node[right] at (8, -4) {$k - n(m+1)$};
\node[right] at (13, -2.5) {$k-n(m+2)$};

\node at (4, -2.25) {$h > k$};

\draw[<->] (0,.5) -- (13,.5);
\node[above] at (7.5, .5) {$h^0(\O(\vec{e})(m+1))$};

\draw [<->] (-.5, 0) -- (-.5, -8);
\node[rotate=90] at (-1, -4) {$h^1(\O(\vec{e})(m-1))$};

\draw [<->] (.5, 0) -- (.5, -5);
\node[rotate=90] at (1, -2.5) {$h^1(\O(\vec{e})(m))$};

\draw[<->] (0,-.5) -- (8,-.5);
\node[below] at (4, -.5) {$h^0(\O(\vec{e})(m))$};
}
\end{tikzpicture}
\end{center}

Since $n(m)$ is a nondecreasing function of $m$,
the boxes in the shaded regions have hook length $h < k$,
and the remaining boxes have $h > k$.
In particular, no box has $h = k$, so $\Gamma(\vec{e})$
is a $k$-core. Counting up the number of boxes with $h < k$, we obtain
\begin{align*}
u(\Gamma(\vec{e})) &= \sum_m n(m) \cdot (k - n(m + 1)) \\
&= \sum_m \ \sum_{e_{\ell_1} \geq -m} \ \ \sum_{e_{\ell_2} < -(m + 1)} 1 \\
&= \sum_{\ell_1, \ell_2} \# \{m : e_{\ell_2} + 1 < -m \leq e_{\ell_1} \} \\
&= \sum_{\ell_1, \ell_2} h^1(\pp^1, \O_{\pp^1}(e_{\ell_2} - e_{\ell_1})) \\
&= h^1(\pp^1, \operatorname{End}(\O_{\pp^1}(\vec{e}))) \\
&= u(\vec{e}). \qedhere
\end{align*}
\end{proof}

\begin{rem}
The fact that $u(\Gamma(\vec{e})) = u(\vec{e})$ already establishes that for $C \to \pp^1$ a general degree $k$ genus $g$ cover, $\dim W^{\vec{e}}(C) \leq \dim W^{\Gamma(\vec{e})}(X) = g - u(\vec{e}).$
\end{rem}

\begin{prop} \label{endpos} We have
\[T^{\leq g}(j, n) = \chi(\O(\vec{e})(-d_j)) - (m_1 + \cdots + m_j) + n.\]
\end{prop}
\begin{proof}
We use the
lengths labeled in the diagram below to calculate
the diagonal index of the first horizontal segment in every residue class
along the boundary of $\Gamma(\vec{e})$:
 
 \begin{center}
 \begin{tikzpicture}[scale = 0.75]
 \node at (4, .75) {\tiny $h^0(\O(\vec{e})(-d_j))$};
 \draw[<->] (0,.5) -- (8, .5);
 \draw (0, 0) -- (9, 0);
 \node at (9.4, 0) {$\ldots$};
 \draw (0, 0) -- (0, -4.5);
 \node at (0, -4.25-.5) {$\vdots$};
  \node at (5, -4.25-.5) {$\vdots$};
 \draw[<->] (-.5,0) -- (-.5, -3);
 \node[rotate=90] at (-.75,-1.5) {\tiny $h^1(\O(\vec{e})(-d_j))$};
 \draw (5, -3) -- (5, -4.5);
 \draw[<->] (0, -3) -- (5, -3);
 \node at (2.5, -3.25) {\tiny $h^0(\O(\vec{e})(-d_j - 1))$};
 \draw (5, -3) -- (8, -3);
 \draw[<->] (5, -3.5) -- (8, -3.5);
 \node at (6.5, -3.75) {\tiny $m_1 + \ldots + m_j$};
 \draw (8, -3) -- (8, -1.5);
 \draw (8, -1.5) -- (9, -1.5);
 \node at (9.4, -1.5) {$\ldots$};
 \draw[line width=1mm, color=violet] (6.5-1.5/3,-3) -- (6.5, -3);
 \draw[<->] (5, -2.5) -- (6.5, -2.5);
 \node at (5.75, -2.3) {\tiny $n$};
 \node at (14, -2.75) {\small ${\color{violet} T^{\leq g}(\ell)} = (\text{horizontal position}) - (\text{vertical position})$};
\node at (15.25, -3.5) {\small $= h^0(\pp^1, \O(\vec{e})(-d_j-1)) + n - h^1(\pp^1, \O(\vec{e})(-d_j))$};
\node at (14.15, -4.25) {\small $= \chi(\O(\vec{e})(-d_j)) -  (m_1 + \ldots + m_j) + n$};
 \end{tikzpicture}
 \end{center}
 as desired.
\end{proof}

We conclude the section with two results on the relationship between the truncations
with same value of $j$, respectively distinct values of $j$.

\begin{cor}\label{same_layer}
If $n_1 \geq n_2$, then $T^{\leq t}(j, n_1) - T^{\leq t}(j, n_2) \leq k - 1$.
\end{cor}
\begin{proof}
By Propositions~\ref{main_ants} and~\ref{endpos},
\[\left\lfloor \frac{T^{\leq t}(j, n_1) - T^{\leq t}(j, n_2)}{k} \right\rfloor \leq \left\lfloor \frac{T^{\leq g}(j, n_1) - T^{\leq g}(j, n_2)}{k} \right\rfloor = \left\lfloor\frac{n_1 - n_2}{k} \right\rfloor = 0. \qedhere\]
\end{proof}

\begin{cor} \label{jpm}
If $t$ is decreasing for $(j_-, n_-)$ and increasing for $(j_+, n_+)$,
then $j_- < j_+$.
\end{cor}
\begin{proof}
By \eqref{sort-0}, we have $T^{\leq t}(j_-, n_-) \leq T^{\leq t}(j_+, n_+) - (k + 1)$.
Therefore by Proposition~\ref{main_ants}, we have $j_- \leq j_+$.
Moreover, by Corollary~\ref{same_layer}, we have $j_- \neq j_+$.
\end{proof}

\begin{cor} \label{distinct layers}
If $j' > j$, then $T^{\leq g}(j', n') - T^{\leq t}(j', n') \geq T^{\leq g}(j, n) - T^{\leq t}(j, n)$.
\end{cor}
\begin{proof}
It suffices to consider the case that $j' = j + 1$.
Because the truncations remains sorted (Proposition~\ref{main_ants}), we have
\begin{align}
T^{\leq t}(j + 1, m_{j + 1}) - T^{\leq t}(j + 1, n') &\geq m_{j + 1} - n' = T^{\leq g}(j + 1, m_{j + 1}) - T^{\leq g}(j + 1, n') \label{e1} \\
T^{\leq t}(j, n) - T^{\leq t}(j, 1) &\geq n - 1 = T^{\leq g}(j, n) - T^{\leq g}(j, 1). \label{e2}
\end{align}
Moreover, by Proposition~\ref{main_ants}, we also have
\[\left\lfloor \frac{T^{\leq g}(j + 1, m_{j + 1}) - T^{\leq g}(j, 1)}{k} \right\rfloor \geq \left\lfloor\frac{T^{\leq t}(j + 1, m_{j + 1}) - T^{\leq t}(j, 1)}{k}\right\rfloor.\]
By Proposition~\ref{endpos}, we have $T^{\leq g}(j + 1, m_{j + 1}) - T^{\leq g}(j, 1) \equiv -1$ mod $k$, so this implies
\begin{equation}
T^{\leq g}(j + 1, m_{j + 1}) - T^{\leq g}(j, 1) \geq T^{\leq t}(j + 1, m_{j + 1}) - T^{\leq t}(j, 1). \label{e3}
\end{equation}
Adding \eqref{e1}, \eqref{e2}, and \eqref{e3}, we obtain
\[T^{\leq t}(j, n) - T^{\leq t}(j + 1, n') \geq T^{\leq g}(j, n) - T^{\leq g}(j + 1, n'),\]
as desired.
\end{proof}

\subsection{Alternative construction of truncations: Paths on the tableau}
The truncations can alternatively be described directly
in terms of the tableau, without going through the affine symmetric group.
In this section, we give a brief exposition of this.
We omit proofs since the construction of the truncations given above
is logically sufficient for the remainder of the paper.

We call a box in a tableau \defi{decreasing} if it is filled with the first instance of a number, reading from the left.  Dually, we call
a box \defi{increasing} if it is filled with the first instance of a of a number, reading from the top.

\begin{minipage}{.3\textwidth}
\begin{center}
\begin{tikzpicture}[scale=.75]
\draw[fill=green!30]  (0,0) -- (1,0) -- (1,-1) -- (0,-1) -- (0,0);
\draw (0.5,0) -- (0.5,-0.25);
\draw (1,0) -- (0.75,-0.25);
\draw (0,0) -- (0.25,-0.25);

\draw (3, -.5) node{a decreasing box};

\draw[fill=yellow!30]  (0,-2) -- (1,-2) -- (1,-3) -- (0,-3) -- (0,-2);
\draw (0,-2) -- (0.25,-0.25-2);
\draw (0,-0.5-2) -- (0.25,-0.5-2);
\draw (0,-1-2) -- (0.225,-0.75-2);

\draw (3.1, -2-.5) node{an increasing box};

\draw[fill=green!30!yellow!30]  (0,-4) -- (1,-4) -- (1,-5) -- (0,-5) -- (0,-4);
\draw (0.5,-4) -- (0.5,-0.25-4);
\draw (1,-4) -- (0.75,-0.25-4);
\draw (0,-4) -- (0.25,-0.25-4);
\draw (0,-0.5-4) -- (0.25,-0.5-4);
\draw (0,-1-4) -- (0.225,-0.75-4);

\draw (3.1, -4-.2) node{an increasing and};
\draw (3.1, -4.5-.2) node{decreasing box};

\end{tikzpicture}
\end{center}
\end{minipage}
\begin{minipage}{.6\textwidth}

\begin{center}
\begin{tikzpicture}[scale=.6]
{\small
\draw (0,0) -- (9, 0);
\draw (0,-1) -- (9,-1);
\draw (0,-2) -- (6,-2);
\draw (0,-3) -- (4,-3);
\draw (0,-4) -- (4,-4);
\draw (0,-5) -- (2,-5);
\draw (0,-6) -- (2,-6);
\draw (0,-7) -- (1,-7);
\draw (0,-8) -- (1,-8);
\draw (0,-9) -- (1,-9);
\draw (0,0) -- (0,-9);
\draw (1,0) -- (1,-9);
\draw (2,0) -- (2,-6);
\draw (3,0) -- (3,-4);
\draw (4,0) -- (4,-4);
\draw (5,0) -- (5,-2);
\draw (6,0) -- (6,-2);
\draw (7,0) -- (7,-1);
\draw (8,0) -- (8,-1);
\draw (9,0) -- (9,-1);
\draw[fill=yellow!30]  (2,0) -- (3,0) -- (3,-1) -- (2,-1) -- (2,0);
\draw (2,0) -- (2.25,-0.25);
\draw (2,-0.5) -- (2.25,-0.5);
\draw (2,-1) -- (2.225,-0.75);
\draw[fill=yellow!30]  (4,0) -- (5,0) -- (5,-1) -- (4,-1) -- (4,0);
\draw (4,0) -- (4.25,-0.25);
\draw (4,-0.5) -- (4.25,-0.5);
\draw (4,-1) -- (4.225,-0.75);
\draw[fill=yellow!30]  (5,0) -- (6,0) -- (6,-1) -- (5,-1) -- (5,0);
\draw (5,0) -- (5.25,-0.25);
\draw (5,-0.5) -- (5.25,-0.5);
\draw (5,-1) -- (5.225,-0.75);
\draw[fill=yellow!30]  (6,0) -- (7,0) -- (7,-1) -- (6,-1) -- (6,0);
\draw (6,0) -- (6.25,-0.25);
\draw (6,-0.5) -- (6.25,-0.5);
\draw (6,-1) -- (6.225,-0.75);
\draw[fill=yellow!30]  (7,0) -- (8,0) -- (8,-1) -- (7,-1) -- (7,0);
\draw (7,0) -- (7.25,-0.25);
\draw (7,-0.5) -- (7.25,-0.5);
\draw (7,-1) -- (7.225,-0.75);
\draw[fill=yellow!30]  (8,0) -- (9,0) -- (9,-1) -- (8,-1) -- (8,0);
\draw (8,0) -- (8.25,-0.25);
\draw (8,-0.5) -- (8.25,-0.5);
\draw (8,-1) -- (8.225,-0.75);
\draw[fill=yellow!30]  (1,-2) -- (2,-2) -- (2,-3) -- (1,-3) -- (1,-2);
\draw (1,-2) -- (1.25,-2.25);
\draw (1,-2.5) -- (1.25,-2.5);
\draw (1,-3) -- (1.225,-2.75);
\draw[fill=yellow!30]  (3,-2) -- (4,-2) -- (4,-3) -- (3,-3) -- (3,-2);
\draw (3,-2) -- (3.25,-2.25);
\draw (3,-2.5) -- (3.25,-2.5);
\draw (3,-3) -- (3.225,-2.75);
\draw[fill=green!30]  (0,-2) -- (1,-2) -- (1,-3) -- (0,-3) -- (0,-2);
\draw (0,-2) -- (0.25,-2.25);
\draw (0.5,-2) -- (0.5,-2.25);
\draw (1,-2) -- (0.75,-2.25);
\draw[fill=green!30]  (0,-4) -- (1,-4) -- (1,-5) -- (0,-5) -- (0,-4);
\draw (0,-4) -- (0.25,-4.25);
\draw (0.5,-4) -- (0.5,-4.25);
\draw (1,-4) -- (0.75,-4.25);
\draw[fill=green!30]  (0,-5) -- (1,-5) -- (1,-6) -- (0,-6) -- (0,-5);
\draw (0,-5) -- (0.25,-5.25);
\draw (0.5,-5) -- (0.5,-5.25);
\draw (1,-5) -- (0.75,-5.25);
\draw[fill=green!30]  (0,-6) -- (1,-6) -- (1,-7) -- (0,-7) -- (0,-6);
\draw (0,-6) -- (0.25,-6.25);
\draw (0.5,-6) -- (0.5,-6.25);
\draw (1,-6) -- (0.75,-6.25);
\draw[fill=green!30]  (0,-7) -- (1,-7) -- (1,-8) -- (0,-8) -- (0,-7);
\draw (0,-7) -- (0.25,-7.25);
\draw (0.5,-7) -- (0.5,-7.25);
\draw (1,-7) -- (0.75,-7.25);
\draw[fill=green!30]  (0,-8) -- (1,-8) -- (1,-9) -- (0,-9) -- (0,-8);
\draw (0,-8) -- (0.25,-8.25);
\draw (0.5,-8) -- (0.5,-8.25);
\draw (1,-8) -- (0.75,-8.25);
\draw[fill=green!30]  (2,-1) -- (3,-1) -- (3,-2) -- (2,-2) -- (2,-1);
\draw (2,-1) -- (2.25,-1.25);
\draw (2.5,-1) -- (2.5,-1.25);
\draw (3,-1) -- (2.75,-1.25);
\draw[fill=green!30]  (2,-3) -- (3,-3) -- (3,-4) -- (2,-4) -- (2,-3);
\draw (2,-3) -- (2.25,-3.25);
\draw (2.5,-3) -- (2.5,-3.25);
\draw (3,-3) -- (2.75,-3.25);
\draw[fill=green!30!yellow!30]  (0,0) -- (1,0) -- (1,-1) -- (0,-1) -- (0,0);
\draw (0.5,0) -- (0.5,-0.25);
\draw (1,0) -- (0.75,-0.25);
\draw (0,0) -- (0.25,-0.25);
\draw (0,-0.5) -- (0.25,-0.5);
\draw (0,-1) -- (0.225,-0.75);
\draw[fill=green!30!yellow!30]  (1,0) -- (2,0) -- (2,-1) -- (1,-1) -- (1,0);
\draw (1.5,0) -- (1.5,-0.25);
\draw (2,0) -- (1.75,-0.25);
\draw (1,0) -- (1.25,-0.25);
\draw (1,-0.5) -- (1.25,-0.5);
\draw (1,-1) -- (1.225,-0.75);
\draw[fill=green!30!yellow!30]  (3,0) -- (4,0) -- (4,-1) -- (3,-1) -- (3,0);
\draw (3.5,0) -- (3.5,-0.25);
\draw (4,0) -- (3.75,-0.25);
\draw (3,0) -- (3.25,-0.25);
\draw (3,-0.5) -- (3.25,-0.5);
\draw (3,-1) -- (3.225,-0.75);
\draw[fill=green!30!yellow!30]  (0,-1) -- (1,-1) -- (1,-2) -- (0,-2) -- (0,-1);
\draw (0.5,-1) -- (0.5,-1.25);
\draw (1,-1) -- (0.75,-1.25);
\draw (0,-1) -- (0.25,-1.25);
\draw (0,-1.5) -- (0.25,-1.5);
\draw (0,-2) -- (0.225,-1.75);
\draw[fill=green!30!yellow!30]  (2,-2) -- (3,-2) -- (3,-3) -- (2,-3) -- (2,-2);
\draw (2.5,-2) -- (2.5,-2.25);
\draw (3,-2) -- (2.75,-2.25);
\draw (2,-2) -- (2.25,-2.25);
\draw (2,-2.5) -- (2.25,-2.5);
\draw (2,-3) -- (2.225,-2.75);
\draw[fill=green!30!yellow!30]  (0,-3) -- (1,-3) -- (1,-4) -- (0,-4) -- (0,-3);
\draw (0.5,-3) -- (0.5,-3.25);
\draw (1,-3) -- (0.75,-3.25);
\draw (0,-3) -- (0.25,-3.25);
\draw (0,-3.5) -- (0.25,-3.5);
\draw (0,-4) -- (0.225,-3.75);
\draw (0.5,-0.5) node{1};
\draw (1.5,-0.5) node{2};
\draw (2.5,-0.5) node{4};
\draw (3.5,-0.5) node{6};
\draw (4.5,-0.5) node{7};
\draw (5.5,-0.5) node{9};
\draw (6.5,-0.5) node{11};
\draw (7.5,-0.5) node{12};
\draw (8.5,-0.5) node{14};
\draw (0.5,-1.5) node{3};
\draw (1.5,-1.5) node{7};
\draw (2.5,-1.5) node{9};
\draw (3.5,-1.5) node{11};
\draw (4.5,-1.5) node{12};
\draw (5.5,-1.5) node{14};
\draw (0.5,-2.5) node{4};
\draw (1.5,-2.5) node{8};
\draw (2.5,-2.5) node{10};
\draw (3.5,-2.5) node{13};
\draw (0.5,-3.5) node{5};
\draw (1.5,-3.5) node{11};
\draw (2.5,-3.5) node{12};
\draw (3.5,-3.5) node{14};
\draw (0.5,-4.5) node{7};
\draw (1.5,-4.5) node{13};
\draw (0.5,-5.5) node{8};
\draw (1.5,-5.5) node{14};
\draw (0.5,-6.5) node{11};
\draw (0.5,-7.5) node{13};
\draw (0.5,-8.5) node{14};

}
\end{tikzpicture}
\end{center}
\end{minipage}

\vspace{10pt}

For an efficient filling, one can show that the number of decreasing boxes in a row is a nonincreasing function of the row,
and is always at most $k - 1$.
For $0\leq i \leq k$, we call the collection of boxes which are the $i$th decreasing box in each row the
\defi{$i$th decreasing cascade} of the tableau.  Dually, we define the \defi{increasing cascades}.
(The $k$th cascade is always empty.)
One then shows that the translation of the $(k+1-i)$th increasing cascade of a tableau up $k+1-i$ and over $i$ ``meshes''
with the $i$th decreasing cascade to form a ``walking path''
and consists of symbols in increasing order.

\begin{center}
\begin{minipage}{.48\textwidth}
\begin{center}
\begin{tikzpicture}[scale=.35]
{\small

\draw (0.5, .35) node{\includegraphics[width=10pt]{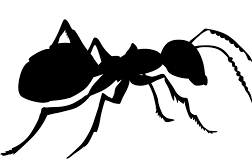}};

\draw[color=white] (0,0) -- (0, 1) -- (1,1) -- (1,0) -- (0,0);
\draw (0,0) -- (9, 0);
\draw (0,-1) -- (9,-1);
\draw (0,-2) -- (6,-2);
\draw (0,-3) -- (4,-3);
\draw (0,-4) -- (4,-4);
\draw (0,-5) -- (2,-5);
\draw (0,-6) -- (2,-6);
\draw (0,-7) -- (1,-7);
\draw (0,-8) -- (1,-8);
\draw (0,-9) -- (1,-9);
\draw (0,0) -- (0,-9);
\draw (1,0) -- (1,-9);
\draw (2,0) -- (2,-6);
\draw (3,0) -- (3,-4);
\draw (4,0) -- (4,-4);
\draw (5,0) -- (5,-2);
\draw (6,0) -- (6,-2);
\draw (7,0) -- (7,-1);
\draw (8,0) -- (8,-1);
\draw (9,0) -- (9,-1);
\draw[fill=green!30]  (0,0) -- (1,0) -- (1,-1) -- (0,-1) -- (0,0);
\draw (0,0) -- (0.5,-0.5);
\draw (0.5,0) -- (0.5,-0.5);
\draw (1,0) -- (0.5,-0.5);
\draw[fill=green!30]  (0,-1) -- (1,-1) -- (1,-2) -- (0,-2) -- (0,-1);
\draw (0,-1) -- (0.5,-1.5);
\draw (0.5,-1) -- (0.5,-1.5);
\draw (1,-1) -- (0.5,-1.5);
\draw[fill=green!30]  (0,-2) -- (1,-2) -- (1,-3) -- (0,-3) -- (0,-2);
\draw (0,-2) -- (0.5,-2.5);
\draw (0.5,-2) -- (0.5,-2.5);
\draw (1,-2) -- (0.5,-2.5);
\draw[fill=green!30]  (0,-3) -- (1,-3) -- (1,-4) -- (0,-4) -- (0,-3);
\draw (0,-3) -- (0.5,-3.5);
\draw (0.5,-3) -- (0.5,-3.5);
\draw (1,-3) -- (0.5,-3.5);
\draw[fill=green!30]  (0,-4) -- (1,-4) -- (1,-5) -- (0,-5) -- (0,-4);
\draw (0,-4) -- (0.5,-4.5);
\draw (0.5,-4) -- (0.5,-4.5);
\draw (1,-4) -- (0.5,-4.5);
\draw[fill=green!30]  (0,-5) -- (1,-5) -- (1,-6) -- (0,-6) -- (0,-5);
\draw (0,-5) -- (0.5,-5.5);
\draw (0.5,-5) -- (0.5,-5.5);
\draw (1,-5) -- (0.5,-5.5);
\draw[fill=green!30]  (0,-6) -- (1,-6) -- (1,-7) -- (0,-7) -- (0,-6);
\draw (0,-6) -- (0.5,-6.5);
\draw (0.5,-6) -- (0.5,-6.5);
\draw (1,-6) -- (0.5,-6.5);
\draw[fill=green!30]  (0,-7) -- (1,-7) -- (1,-8) -- (0,-8) -- (0,-7);
\draw (0,-7) -- (0.5,-7.5);
\draw (0.5,-7) -- (0.5,-7.5);
\draw (1,-7) -- (0.5,-7.5);
\draw[fill=green!30]  (0,-8) -- (1,-8) -- (1,-9) -- (0,-9) -- (0,-8);
\draw (0,-8) -- (0.5,-8.5);
\draw (0.5,-8) -- (0.5,-8.5);
\draw (1,-8) -- (0.5,-8.5);
}
\end{tikzpicture}

\noindent The $1$st walking path.
\end{center}

\end{minipage}
\begin{minipage}{.48\textwidth}
\begin{center}
\begin{tikzpicture}[scale=.35]
{\small

\draw (1.5, .35) node{\includegraphics[width=10pt]{ant2.png}};

\draw[color=white] (0,0) -- (0, 1) -- (1,1) -- (1,0) -- (0,0);
\draw (0,0) -- (9, 0);
\draw (0,-1) -- (9,-1);
\draw (0,-2) -- (6,-2);
\draw (0,-3) -- (4,-3);
\draw (0,-4) -- (4,-4);
\draw (0,-5) -- (2,-5);
\draw (0,-6) -- (2,-6);
\draw (0,-7) -- (1,-7);
\draw (0,-8) -- (1,-8);
\draw (0,-9) -- (1,-9);
\draw (0,0) -- (0,-9);
\draw (1,0) -- (1,-9);
\draw (2,0) -- (2,-6);
\draw (3,0) -- (3,-4);
\draw (4,0) -- (4,-4);
\draw (5,0) -- (5,-2);
\draw (6,0) -- (6,-2);
\draw (7,0) -- (7,-1);
\draw (8,0) -- (8,-1);
\draw (9,0) -- (9,-1);
\draw[fill=green!30]  (1,0) -- (2,0) -- (2,-1) -- (1,-1) -- (1,0);
\draw (1,0) -- (1.5,-0.5);
\draw (1.5,0) -- (1.5,-0.5);
\draw (2,0) -- (1.5,-0.5);
\draw[fill=green!30]  (2,-1) -- (3,-1) -- (3,-2) -- (2,-2) -- (2,-1);
\draw (2,-1) -- (2.5,-1.5);
\draw (2.5,-1) -- (2.5,-1.5);
\draw (3,-1) -- (2.5,-1.5);
\draw[fill=green!30]  (2,-2) -- (3,-2) -- (3,-3) -- (2,-3) -- (2,-2);
\draw (2,-2) -- (2.5,-2.5);
\draw (2.5,-2) -- (2.5,-2.5);
\draw (3,-2) -- (2.5,-2.5);
\draw[fill=green!30]  (2,-3) -- (3,-3) -- (3,-4) -- (2,-4) -- (2,-3);
\draw (2,-3) -- (2.5,-3.5);
\draw (2.5,-3) -- (2.5,-3.5);
\draw (3,-3) -- (2.5,-3.5);
\draw[fill=yellow!30]  (2,0) -- (3,0) -- (3,-1) -- (2,-1) -- (2,0);
\draw (2,0) -- (2.5,-0.5);
\draw (2,-0.5) -- (2.5,-0.5);
\draw (2,-1) -- (2.5,-0.5);

\draw[fill=yellow!20]  (0,-3) -- (1,-3) -- (1,-4) -- (0,-4) -- (0,-3);
\draw[densely dotted] (0,-3) -- (0.5,-3.5);
\draw[densely dotted] (0,-3.5) -- (0.5,-3.5);
\draw[densely dotted] (0,-4) -- (0.5,-3.5);

\draw[dashed, ->] (.5, -2.75) -- (1.75, -1.25);
\draw[color=white]  (12,1) -- (13,1) -- (13,0) -- (12,0) -- (12,1);
}
\end{tikzpicture}

\noindent The $2$nd walking path.
\end{center}

\end{minipage}

\vspace{5pt}

\end{center}

\begin{center}
\begin{minipage}{.48\textwidth}
\begin{center}
\begin{tikzpicture}[scale=.35]
{\small

\draw (2.5, .35) node{\includegraphics[width=10pt]{ant2.png}};

\draw (0,0) -- (9, 0);
\draw (0,-1) -- (9,-1);
\draw (0,-2) -- (6,-2);
\draw (0,-3) -- (4,-3);
\draw (0,-4) -- (4,-4);
\draw (0,-5) -- (2,-5);
\draw (0,-6) -- (2,-6);
\draw (0,-7) -- (1,-7);
\draw (0,-8) -- (1,-8);
\draw (0,-9) -- (1,-9);
\draw (0,0) -- (0,-9);
\draw (1,0) -- (1,-9);
\draw (2,0) -- (2,-6);
\draw (3,0) -- (3,-4);
\draw (4,0) -- (4,-4);
\draw (5,0) -- (5,-2);
\draw (6,0) -- (6,-2);
\draw (7,0) -- (7,-1);
\draw (8,0) -- (8,-1);
\draw (9,0) -- (9,-1);
\draw[fill=green!30]  (3,0) -- (4,0) -- (4,-1) -- (3,-1) -- (3,0);
\draw (3,0) -- (3.5,-0.5);
\draw (3.5,0) -- (3.5,-0.5);
\draw (4,0) -- (3.5,-0.5);
\draw[fill=yellow!30]  (3,1) -- (4,1) -- (4,0) -- (3,0) -- (3,1);
\draw (3,1) -- (3.5,0.5);
\draw (3,0.5) -- (3.5,0.5);
\draw (3,0) -- (3.5,0.5);
\draw[fill=yellow!30]  (4,0) -- (5,0) -- (5,-1) -- (4,-1) -- (4,0);
\draw (4,0) -- (4.5,-0.5);
\draw (4,-0.5) -- (4.5,-0.5);
\draw (4,-1) -- (4.5,-0.5);
\draw[fill=yellow!30]  (5,0) -- (6,0) -- (6,-1) -- (5,-1) -- (5,0);
\draw (5,0) -- (5.5,-0.5);
\draw (5,-0.5) -- (5.5,-0.5);
\draw (5,-1) -- (5.5,-0.5);
\draw[fill=yellow!30]  (6,0) -- (7,0) -- (7,-1) -- (6,-1) -- (6,0);
\draw (6,0) -- (6.5,-0.5);
\draw (6,-0.5) -- (6.5,-0.5);
\draw (6,-1) -- (6.5,-0.5);

\draw[fill=yellow!20]  (0,-1) -- (1,-1) -- (1,-2) -- (0,-2) -- (0,-1);
\draw[densely dotted] (0,-1) -- (0.5,-1.5);
\draw[densely dotted] (0,-1.5) -- (0.5,-1.5);
\draw[densely dotted] (0,-2) -- (0.5,-1.5);
\draw[fill=yellow!20]  (1,-2) -- (2,-2) -- (2,-3) -- (1,-3) -- (1,-2);
\draw[densely dotted] (1,-2) -- (1.5,-2.5);
\draw[densely dotted] (1,-2.5) -- (1.5,-2.5);
\draw[densely dotted] (1,-3) -- (1.5,-2.5);
\draw[fill=yellow!20]  (2,-2) -- (3,-2) -- (3,-3) -- (2,-3) -- (2,-2);
\draw[densely dotted] (2,-2) -- (2.5,-2.5);
\draw[densely dotted] (2,-2.5) -- (2.5,-2.5);
\draw[densely dotted] (2,-3) -- (2.5,-2.5);
\draw[fill=yellow!20]  (3,-2) -- (4,-2) -- (4,-3) -- (3,-3) -- (3,-2);
\draw[densely dotted] (3,-2) -- (3.5,-2.5);
\draw[densely dotted] (3,-2.5) -- (3.5,-2.5);
\draw[densely dotted] (3,-3) -- (3.5,-2.5);

\draw[dashed, ->] (1.5, -1.75) -- (3.75, -1.25);

}
\end{tikzpicture}

\noindent The $3$rd walking path.
\end{center}

\end{minipage}
\begin{minipage}{.48\textwidth}
\begin{center}
\begin{tikzpicture}[scale=.35]
{\small

\draw (3.5, .35) node{\includegraphics[width=10pt]{ant2.png}};

\draw (0,0) -- (9, 0);
\draw (0,-1) -- (9,-1);
\draw (0,-2) -- (6,-2);
\draw (0,-3) -- (4,-3);
\draw (0,-4) -- (4,-4);
\draw (0,-5) -- (2,-5);
\draw (0,-6) -- (2,-6);
\draw (0,-7) -- (1,-7);
\draw (0,-8) -- (1,-8);
\draw (0,-9) -- (1,-9);
\draw (0,0) -- (0,-9);
\draw (1,0) -- (1,-9);
\draw (2,0) -- (2,-6);
\draw (3,0) -- (3,-4);
\draw (4,0) -- (4,-4);
\draw (5,0) -- (5,-2);
\draw (6,0) -- (6,-2);
\draw (7,0) -- (7,-1);
\draw (8,0) -- (8,-1);
\draw (9,0) -- (9,-1);
\draw[fill=yellow!20]  (0,0) -- (1,0) -- (1,-1) -- (0,-1) -- (0,0);
\draw[densely dotted] (0,0) -- (0.5,-0.5);
\draw[densely dotted] (0,-0.5) -- (0.5,-0.5);
\draw[densely dotted] (0,-1) -- (0.5,-0.5);
\draw[fill=yellow!20]  (1,0) -- (2,0) -- (2,-1) -- (1,-1) -- (1,0);
\draw[densely dotted] (1,0) -- (1.5,-0.5);
\draw[densely dotted] (1,-0.5) -- (1.5,-0.5);
\draw[densely dotted] (1,-1) -- (1.5,-0.5);
\draw[fill=yellow!20]  (2,0) -- (3,0) -- (3,-1) -- (2,-1) -- (2,0);
\draw[densely dotted] (2,0) -- (2.5,-0.5);
\draw[densely dotted] (2,-0.5) -- (2.5,-0.5);
\draw[densely dotted] (2,-1) -- (2.5,-0.5);
\draw[fill=yellow!20]  (3,0) -- (4,0) -- (4,-1) -- (3,-1) -- (3,0);
\draw[densely dotted] (3,0) -- (3.5,-0.5);
\draw[densely dotted] (3,-0.5) -- (3.5,-0.5);
\draw[densely dotted] (3,-1) -- (3.5,-0.5);
\draw[fill=yellow!20]  (4,0) -- (5,0) -- (5,-1) -- (4,-1) -- (4,0);
\draw[densely dotted] (4,0) -- (4.5,-0.5);
\draw[densely dotted] (4,-0.5) -- (4.5,-0.5);
\draw[densely dotted] (4,-1) -- (4.5,-0.5);
\draw[fill=yellow!20]  (5,0) -- (6,0) -- (6,-1) -- (5,-1) -- (5,0);
\draw[densely dotted] (5,0) -- (5.5,-0.5);
\draw[densely dotted] (5,-0.5) -- (5.5,-0.5);
\draw[densely dotted] (5,-1) -- (5.5,-0.5);
\draw[fill=yellow!20]  (6,0) -- (7,0) -- (7,-1) -- (6,-1) -- (6,0);
\draw[densely dotted] (6,0) -- (6.5,-0.5);
\draw[densely dotted] (6,-0.5) -- (6.5,-0.5);
\draw[densely dotted] (6,-1) -- (6.5,-0.5);
\draw[fill=yellow!20]  (7,0) -- (8,0) -- (8,-1) -- (7,-1) -- (7,0);
\draw[densely dotted] (7,0) -- (7.5,-0.5);
\draw[densely dotted] (7,-0.5) -- (7.5,-0.5);
\draw[densely dotted] (7,-1) -- (7.5,-0.5);
\draw[fill=yellow!20]  (8,0) -- (9,0) -- (9,-1) -- (8,-1) -- (8,0);
\draw[densely dotted] (8,0) -- (8.5,-0.5);
\draw[densely dotted] (8,-0.5) -- (8.5,-0.5);
\draw[densely dotted] (8,-1) -- (8.5,-0.5);

\draw[fill=yellow!30]  (4,1) -- (5,1) -- (5,0) -- (4,0) -- (4,1);
\draw (4,1) -- (4.5,0.5);
\draw (4,0.5) -- (4.5,0.5);
\draw (4,0) -- (4.5,0.5);
\draw[fill=yellow!30]  (5,1) -- (6,1) -- (6,0) -- (5,0) -- (5,1);
\draw (5,1) -- (5.5,0.5);
\draw (5,0.5) -- (5.5,0.5);
\draw (5,0) -- (5.5,0.5);
\draw[fill=yellow!30]  (6,1) -- (7,1) -- (7,0) -- (6,0) -- (6,1);
\draw (6,1) -- (6.5,0.5);
\draw (6,0.5) -- (6.5,0.5);
\draw (6,0) -- (6.5,0.5);
\draw[fill=yellow!30]  (7,1) -- (8,1) -- (8,0) -- (7,0) -- (7,1);
\draw (7,1) -- (7.5,0.5);
\draw (7,0.5) -- (7.5,0.5);
\draw (7,0) -- (7.5,0.5);
\draw[fill=yellow!30]  (8,1) -- (9,1) -- (9,0) -- (8,0) -- (8,1);
\draw (8,1) -- (8.5,0.5);
\draw (8,0.5) -- (8.5,0.5);
\draw (8,0) -- (8.5,0.5);
\draw[fill=yellow!30]  (9,1) -- (10,1) -- (10,0) -- (9,0) -- (9,1);
\draw (9,1) -- (9.5,0.5);
\draw (9,0.5) -- (9.5,0.5);
\draw (9,0) -- (9.5,0.5);
\draw[fill=yellow!30]  (10,1) -- (11,1) -- (11,0) -- (10,0) -- (10,1);
\draw (10,1) -- (10.5,0.5);
\draw (10,0.5) -- (10.5,0.5);
\draw (10,0) -- (10.5,0.5);
\draw[fill=yellow!30]  (11,1) -- (12,1) -- (12,0) -- (11,0) -- (11,1);
\draw (11,1) -- (11.5,0.5);
\draw (11,0.5) -- (11.5,0.5);
\draw (11,0) -- (11.5,0.5);
\draw[fill=yellow!30]  (12,1) -- (13,1) -- (13,0) -- (12,0) -- (12,1);
\draw (12,1) -- (12.5,0.5);
\draw (12,0.5) -- (12.5,0.5);
\draw (12,0) -- (12.5,0.5);

}
\end{tikzpicture}

\noindent The $4$th walking path.
\end{center}

\end{minipage}

\end{center} 

\vspace{10pt}

We think of these walking paths as giving instructions to a collection of $k$ ants starting at $(0, \ell)$ for $1 \leq \ell \leq k$.
At time $t$, if the symbol $t$ appears as the next box in the $\ell$th walking path, the $\ell$th ant steps onto the corresponding box.
The diagonal index of the $\ell$th ant at time $t$ recovers $T^{\leq t}(\ell)$;
the symbol $t$ is decreasing (respectively increasing) for $\ell$ if $t$ is a decreasing (respectively increasing) box in the $\ell$th
walking path.

\section{Set-Theoretic Regeneration \label{regeneration-st}}

In this section, we show that every point of $W^{\Gamma(\vec{e})}(X)$
arises as a limit of line bundles with splitting type $\vec{e}$.
In particular, the fiber over $0$ of the closure of $W^{\vec{e}}(\mathcal{X}^* / B^*)$
coincides \emph{set-theoretically} with $W^{\Gamma(\vec{e})}(X)$, and therefore also with $W^{\vec{e}}(X)_{\text{red}}$.
For this task, we will use the language of limit linear series, as developed by
Eisenbud and Harris \cite{lls}.

Then, in the following section, we will show that $W^{\vec{e}}(X)$ is reduced.
Since the closure of $W^{\vec{e}}(\mathcal{X}^* / B^*)$
is a priori contained scheme-theoretically in $W^{\vec{e}}(X)$,
this will upgrade our set-theoretic regeneration theorem to a scheme-theoretic regeneration theorem. In other words, this
will show that the fiber over $0$ of the
closure of $W^{\vec{e}}(\mathcal{X}^* / B^*)$ is equal to $W^{\Gamma(\vec{e})}(X)$ as schemes.

Recall that, for a splitting type $\vec{e}$, we write $d_1 > \cdots  > d_\s$ for the distinct entries of $\vec{e}$, and $m_1, \ldots, m_\s$ for the corresponding multiplicities.
(Note that $e_1 \leq e_2 \leq \cdots \leq e_k$ but $d_1 > d_2 > \cdots > d_\s$!)
If $f\colon C \to \pp^1$ is a smooth degree $k$ cover, the locus $W^{\vec{e}}(C)$ can be described (set-theoretically) as follows.
A line bundle $L$ is in $W^{\vec{e}}(C)$ if and only if $L$ possesses a collection for $j = 1, \ldots, \s-1$ of linear series $V_j \subset H^0(C,L(-d_j))$ with $\dim V_j = h^0(\pp^1,\O(\vec{e})(-d_j))$.
Moreover, the $V_j$ may be chosen so that the image of the natural map
\begin{equation} \label{im}
 V_{j-1} \otimes H^0(\pp^1, \O(d_{j-1} - d_{j})) \rightarrow H^0(C, L(-d_{j-1})) \otimes H^0(\pp^1, \O(d_{j-1} - d_{j})) \rightarrow H^0(C, L(-d_{j}))
 \end{equation}
is contained in $V_{j}$. We call such a collection $\{V_1, \ldots, V_\s\}$ satisfying \eqref{im} an \defi{$\vec{e}$-nested linear series} and $V_j$ the linear series at \defi{layer $j$}. 
By convention, we set $V_0 \colonequals \{0\}$.

Now suppose that $p \in C$ is a point of total ramification for $f$, and
let $\Im_{j}(V_{j-1})$ denote the image of \eqref{im}.
Call the ramification indices of $V_{j}$ at $p$ that are \textit{not} ramification indices of $\Im_{j}(V_{j-1})$ at $p$ the \defi{new ramification indices for layer $j$ at $p$}.
An $\vec{e}$-nested linear series will be called \defi{non-colliding at $p$} if the new ramification indices at $p$ are distinct from eachother, and from all ramification indices at lower layers, modulo $k$. 
For $j' < j$, if $V_{j'}$ has a section vanishing to order $a$ at $p$, then there are sections vanishing to orders $a, a+k, \ldots, a+(d_{j'} - d_{j})k$ in $\Im_{j}(V_{j-1})$.

\begin{lem} \label{max-dim}
If $\{V_1, \ldots, V_{s}\}$ is a $\vec{e}$-nested linear series that is non-colliding at some $p \in C$, then
\[\dim \Im_j(V_{j-1}) = h^0(\pp^1, \O_{\pp^1}(\vec{e})(-d_j)) - m_j.\]
Hence, the number of new ramification indices at layer $j$
is exactly $m_j$.
\end{lem}
\begin{proof}
We induct on $j$. When $j = 1$, we have $\dim \Im_{1}(V_0) = 0$ by definition. Suppose that for all $j' < j$, there are $m_{j'}$ new sections at layer $j'$, say $\sigma_{j', 1}, \ldots, \sigma_{j', m_{j'}}$, of distinct vanishing orders mod~$k$. Then $\Im_j(V_{j-1})$ is spanned by
the image of
\[\bigoplus_{j' < j} \langle \sigma_{j', 1}, \ldots, \sigma_{j', m_{j'}} \rangle \otimes H^0(\pp^1, \O(d_{j'} - d_j)) \rightarrow H^0(C, L(-d_j)).\]
Considering orders of vanishing at $p$, we see that the map above is injective, so
\[\dim \Im_j(V_{j-1}) = \sum_{j' < j} m_{j'} h^0(\pp^1, \O(d_{j'} - d_j)) = h^0(\pp^1, \O(\vec{e})(-d_j)) - m_{j}. \qedhere\]
\end{proof}

The above notions extend readily to limit linear series
on our chain curve $X$: we call a collection of limit linear series $\{\mathcal{V}_1, \ldots, \mathcal{V}_\s\}$ an \defi{$\vec{e}$-nested limit linear series}
if for each component $E^i \rightarrow \pp^1$, the collection of aspects $\{\mathcal{V}_1(E^i), \ldots, \mathcal{V}_\s(E^i)\}$ form an $\vec{e}$-nested linear series;
we say this limit linear series is \defi{non-colliding} if $\{\mathcal{V}_1(E^i), \ldots, \mathcal{V}_\s(E^i)\}$
is non-colliding at $p^{i-1}$ and $p^i$.

In what follows, new ramification indices will be denoted in bold: $\na{i}{j}{n}$ for $n = 1, \ldots, m_j$ will be the new ramification indices in layer $j$ at $p^i$ on the component $E^{i+1}$.
In terms of the $\na{i}{j}{n}$, the ramification indices of $V_j(E^{i+1})$ at $p^i$ are
\begin{equation} \label{ramj}
\{\na{i}{j'}{n} + \delta k : \delta = 0, \ldots, d_j - d_{j'}, n = 1, \ldots, m_{j'}, j' = 1, \ldots, j\}.
\end{equation}
We define $\nbb{i}{j,n} \colonequals d-d_jk- \naa{i}{j,n}$; these represent the new ramification indices
in layer $j$ at $p^i$ on the component $E^i$
(if our limit linear series is refined).
\begin{center}
\begin{tikzpicture}
\draw (.5, -.25) .. controls (.5, -.3) and (2, -.5) .. (3, 0);
\draw (2, 0) .. controls (4, -1) and (6, -1) .. (8, 0);
\draw (7, 0) .. controls (8, -.5) and (9.5, -.3) .. (9.5, -.25);
\filldraw (2.44, -0.205) circle[radius=0.03];
\filldraw (7.56, -0.205) circle[radius=0.03];
\draw (2.65, .15) node{$p^{i-1}$};
\draw (7.9, .15) node{$p^{i}$};
\draw (.5, -0.65) node{$E^{i-1}$};
\draw (5, -0.95) node{$E^{i}$};
\draw (9.5, -0.65) node{$E^{i+1}$};

\draw[dashed] (2.44, -.1) -- (2.44, -2);
\node[left] at (2.44, -1.25) {$\nb{i-1}{j}{n}$};
\node[right] at (2.44, -1.25) {$\na{i-1}{j}{n}$};

\draw[dashed] (7.56, -.1) -- (7.56, -2);
\node[left] at (7.56, -1.25) {$\nb{i}{j}{n}$};
\node[right] at (7.56, -1.25) {$\na{i}{j}{n}$};
\end{tikzpicture}
\end{center}
 For ease of notation, we will sometimes replace $(j, n)$ by its corresponding lexigraphical order
\[\ell \colonequals \ell(j, n) = m_1 + \ldots + m_{j-1} + n,\]
and so write $\na{i}{j}{n} = \naa{i}{\ell}$ and $\nb{i}{j}{n} = \nbb{i}{\ell}$.

Let $T$ be any efficient tableau of shape $\Gamma(\vec{e})$.
Our argument for the regeneration theorem will
proceed in two basic steps. 
First, we show that a general $\vec{e}$-positive line bundle
in $W^T(X)$
arises from a refined, non-colliding $\vec{e}$-nested limit linear series.
Then, we prove a regeneration theorem for refined, non-colliding $\vec{e}$-nested limit linear series.

\subsection{From tableaux to limit linear series}

We explain how to construct nested limit linear series from tableaux. We will need to know our proposed new ramification indices increase across layers, as established in the following lemma.

\begin{lem} \label{bincreasing}
Let $T$ be an efficient tableau of shape $\Gamma(\vec{e})$. Define
\begin{equation} \label{defab}
\na{i}{j}{n} \colonequals T^{\leq i}(j, n) + i - 1 \quad (\text{so} \ \nb{i}{j}{n} = d-d_jk- \na{i}{j}{n} = d - d_jk - i + 1 - T^{\leq i}(j,n)).
\end{equation}
If $j' < j$, then
$\na{i}{j'}{n'} < \na{i}{j}{n}$ and $\nb{i}{j'}{n'} < \nb{i}{j}{n}$.
\end{lem}
\begin{proof}
For each $i$, Lemma \ref{main_ants} says $\naa{i}{1} < \naa{i}{2} < \naa{i}{3} < \cdots$. 
To obtain the statement for the $\nb{i}{j}{n}$, we first apply Proposition~\ref{endpos} to obtain
\[\nb{g}{j}{n} = d - d_jk - g + 1 - [\chi(\O(\vec{e})(-d_j)) - (m_1 + \cdots + m_j) + n] = m_1 + \ldots + m_j - n.\]
In particular,
\begin{equation} \label{bless1}
\nb{g}{j'}{n'} = m_1 + \ldots + m_{j'} - n' < m_1 + \ldots + m_{j} - n = \nb{g}{j}{n}.
\end{equation}
We then rewrite Lemma \ref{distinct layers} in terms of the $\boldsymbol{b}$'s to obtain
\begin{equation} \label{bless2}
\nb{i}{j'}{n'} - \nb{g}{j'}{n'} \leq \nb{i}{j}{n} - \nb{g}{j}{n}.
\end{equation}
The claim now follows from adding \eqref{bless1} and \eqref{bless2}.
\end{proof}

\begin{rem}
Although $\nb{i}{j'}{n'} < \nb{i}{j}{n}$ for $j' < j$ (increasing across layers), we have $\nb{i}{j}{1} > \nb{i}{j}{2} > \cdots > \nb{i}{j}{m_j}$ (decreasing within a layer).
\end{rem}

Given a tableau $T$ of shape $\Gamma(\vec{e})$, we now show that a general line bundle in $W^T(X)$ posses a unique $\vec{e}$-nested limit linear series with the proposed ramification.

\begin{lem} \label{hasit}
Let $T$ be a tableau of shape $\Gamma(\vec{e})$, and let $\naa{i}{\ell}$ and $\nbb{i}{\ell}$ be as defined in the previous lemma.
A general line bundle $L$ in $W^T(X)$ possesses a unique $\vec{e}$-nested limit linear series whose new ramification indices
at $p^i$ are are (exactly) $\naa{i}{\ell}$ for the $E^{i + 1}$-aspects and $\nbb{i}{\ell}$ for the $E^i$-aspects.
This $\vec{e}$-nested limit linear series is refined and non-colliding.
\end{lem} 

\begin{proof} Equation \eqref{distinct} ensures that the proposed ramification indices are non-colliding.
Equation \eqref{defab} ensures the limit linear series is refined (c.f.\ \eqref{ramj}).

Fix $i$; we will build an $\vec{e}$-nested linear series on $E^i$, which will be the $E^i$-aspect
of our desired $\vec{e}$-nested limit linear series.
If $T[i] = *$, then $L^i$ is a general degree $d$
line bundle. Moreover, for any~$(j, n)$, we have $T^{\leq i - 1}(j, n) = T^{\leq i}(j, n)$,
so
\[\na{i-1}{j}{n} + \nb{i}{j}{n} = \na{i-1}{j}{n} + d - d_jk - \na{i}{j}{n} = (T^{\leq i - 1}(j, n) + i - 2) + d - d_jk - (T^{\leq i}(j, n) + i - 1) = d - d_jk - 1.\]
Thus, $L(-d_j)^i$
has a unique (up to rescaling) section $\sigma_{j, n}$ vanishing to orders exactly $\naa{i-1}{j, n}$ at $p^{i - 1}$
and $\nbb{i}{j, n}$ at $p^i$.
The unique linear series on $E^i$ at layer $j$ having the prescribed ramification is therefore
\[V_j(E^i) = \bigoplus_{(j', n') : j' \leq j} H^0(\O_{\pp^1}(d_{j'} - d_j)) \cdot \sigma_{j', n'}.\]

We now suppose that $i$ appears in $T$.
Let $\ell_\pm$ be the indices such that $i$ is decreasing for $\ell_-$ and increasing for $\ell_+$,
and write $\ell_\pm = (j_\pm, n_\pm)$.
By Corollary~\ref{jpm}, we have $j_- < j_+$.
The $\naa{i-1}{\ell}$ and $\naa{i}{\ell}$ are related via:
\begin{align}
\naa{i-1}{\ell_\pm} = T^{\leq i - 1}(\ell_\pm) + i - 2 = T^{\leq i}(\ell_\pm) + i - 2 \mp 1 &= \naa{i}{\ell_{\pm}} - 1 \mp 1 \label{ih} \\
\naa{i-1}{\ell} = T^{\leq i - 1}(\ell) + i - 2 = T^{\leq i}(\ell) + i - 2 &= \naa{i}{\ell} - 1 \ \text{for $\ell \neq \ell_\pm$}. \label{invh}
\end{align}
Equivalently,
\begin{equation}
\naa{i-1}{\ell} + \nbb{i}{\ell} = \begin{cases}
d - d_jk - 1 \mp 1 &\text{if $\ell = \ell_\pm$} \\
d - d_jk - 1 & \text{if $\ell \neq \ell_\pm$} \\
\end{cases} \label{rule}
\end{equation}
Futhermore, by Equation \eqref{swapmod},
\begin{equation} \label{sr}
\naa{i-1}{\ell_-} \equiv \naa{i-1}{\ell_+} + 1 \quad \text{and} \quad \naa{i}{\ell_+} \equiv \naa{i}{\ell_-} + 1 \quad \pmod{k}.
\end{equation}
By definition of $W^T(X)$, we have
\begin{equation} \label{li}
L^i \simeq \O_{E^i}(d_{j_-})(\naa{i-1}{\ell_-} p^{i - 1} + \nbb{i}{\ell_-} p^i) \simeq \O_{E^i}(d_{j_+})((\naa{i-1}{\ell_+} + 1) p^{i - 1} + (\nbb{i}{\ell_+} + 1) p^i).
\end{equation}

We now build the layers inductively (starting with the tautological case $V_0(E^i) = \{0\}$).
After we have built layer $V_{j - 1}(E^i)$, write $\ell = (j, n)$, and let
\[S_\ell \colonequals H^0(E^i, L^i(-d_j)(-\naa{i-1}{\ell}  p^{i-1}- \nbb{i}{\ell}p^i ) )\subset H^0(E^i, L^i(-d_j))\]
be the subspace of sections having vanishing order at least $\naa{i - 1}{\ell}$ at $p^{i-1}$ and $\nbb{i}{\ell}$ at $p^{i}$.
We must show that $S_\ell$ posseses a section vanishing to order exactly $\naa{i-1}{\ell}$ at $p^{i-1}$ and $\nbb{i}{\ell}$ at $p^{i}$,
and that this section is essentially unique, in the sense that,
along with $\Im_j(V_{j-1}(E^i)) \cap S_\ell$, it generates $S_\ell$.

For $\ell \neq \ell_{\pm}$, the line bundle
$L^i(-d_j)(-\naa{i - 1}{\ell}  p^{i-1}- \nbb{i}{\ell}p^i )$
has degree $1$.
Equation \eqref{distinct} implies
$\naa{i-1}{\ell} \not\equiv \naa{i-1}{\ell_-}$ mod $k$,
so this line bundle is not equal to $\O_{E^i}(p^i)$.
Similarly, since $\naa{i - 1}{\ell} \not\equiv \naa{i - 1}{\ell_+}$ mod $k$,
this line bundle is not equal to $\O_{E^i}(p^{i - 1})$.
Thus $\dim S_\ell = 1$ and $S_\ell$ consists of sections of the exact vanishing order desired. 

When $\ell = \ell_-$, we have
$L^i(-d_{j_-})(-\naa{i - 1}{\ell_-}  p^{i-1}- \nbb{i}{\ell_-}p^i ) = \O_{E^i}$.
Thus
$\dim S_{\ell_-} = 1$, and the section with the required vanishing orders corresponds to the constant section of $L^i(-d_{j_-})$. 

Finally, when $\ell = \ell_+$, we have
 $L^i(-d_{j_+})(-\naa{i - 1}{\ell_+}  p^{i-1}- \nbb{i}{\ell_+}p^i ) = \O_{E^i}(p^{i-1} + p^{i})$ and
  $\dim S_{\ell_+} = 2$. 
However, we shall show that
$\Im_{j}(V_{(j_+)-1}(E^i))$ contains the $1$-dimensional subspace of sections
\[H^0(\O_{E^i}) \subset H^0(\O_{E^i}(p^{i-1} + p^{i})) = S_{\ell_+} \subset H^0(E^i, L^i(-d_{j_+}))\]
that vanish to order $\naa{i - 1}{\ell_+} +1$ at $p^{i-1}$ and $\nbb{i}{\ell_+} + 1$ at $p^i$.
By \eqref{ramj}, this follows in turn from
\begin{align*}
\naa{i - 1}{\ell_+} +1 \equiv \naa{i-1}{\ell_-} \quad &\text{and} \quad \nbb{i}{\ell_+} + 1 \equiv \nbb{i}{\ell_-} \qquad \pmod{k} \\
\naa{i-1}{\ell_+} + 1 \geq \naa{i-1}{\ell_-} \quad &\text{and} \quad \nbb{i}{\ell_+} + 1 \geq \nbb{i}{\ell_-}.
\end{align*}
The first line follows from \eqref{sr}.
The second line follows from Lemma \ref{bincreasing} because $j_- < j_+$.
 \end{proof}

 \begin{landscape}
\begin{minipage}{.6\textwidth}
\begin{tikzpicture}[scale=.70]
{\small 
\draw (0,0) -- (12, 0);
\draw (0,-1) -- (12,-1);
\draw (0,-2) -- (8,-2);
\draw (0,-3) -- (5,-3);
\draw (0,-4) -- (5,-4);
\draw (0,-5) -- (2,-5);
\draw (0,-6) -- (2,-6);
\draw (0,-7) -- (1,-7);
\draw (0,-8) -- (1,-8);
\draw (0,-9) -- (1,-9);
\draw (0,-10) -- (1,-10);
\draw (0,0) -- (0,-10);
\draw (1,0) -- (1,-10);
\draw (2,0) -- (2,-6);
\draw (3,0) -- (3,-4);
\draw (4,0) -- (4,-4);
\draw (5,0) -- (5,-4);
\draw (6,0) -- (6,-2);
\draw (7,0) -- (7,-2);
\draw (8,0) -- (8,-2);
\draw (9,0) -- (9,-1);
\draw (10,0) -- (10,-1);
\draw (11,0) -- (11,-1);
\draw (12,0) -- (12,-1);
\draw[fill=yellow!30]  (3,0) -- (4,0) -- (4,-1) -- (3,-1) -- (3,0);
\draw (3,0) -- (3.25,-0.25);
\draw (3,-0.5) -- (3.25,-0.5);
\draw (3,-1) -- (3.25,-0.75);
\draw[fill=yellow!30]  (4,0) -- (5,0) -- (5,-1) -- (4,-1) -- (4,0);
\draw (4,0) -- (4.25,-0.25);
\draw (4,-0.5) -- (4.25,-0.5);
\draw (4,-1) -- (4.25,-0.75);
\draw[fill=yellow!30]  (5,0) -- (6,0) -- (6,-1) -- (5,-1) -- (5,0);
\draw (5,0) -- (5.25,-0.25);
\draw (5,-0.5) -- (5.25,-0.5);
\draw (5,-1) -- (5.25,-0.75);
\draw[fill=yellow!30]  (6,0) -- (7,0) -- (7,-1) -- (6,-1) -- (6,0);
\draw (6,0) -- (6.25,-0.25);
\draw (6,-0.5) -- (6.25,-0.5);
\draw (6,-1) -- (6.25,-0.75);
\draw[fill=yellow!30]  (8,0) -- (9,0) -- (9,-1) -- (8,-1) -- (8,0);
\draw (8,0) -- (8.25,-0.25);
\draw (8,-0.5) -- (8.25,-0.5);
\draw (8,-1) -- (8.25,-0.75);
\draw[fill=yellow!30]  (9,0) -- (10,0) -- (10,-1) -- (9,-1) -- (9,0);
\draw (9,0) -- (9.25,-0.25);
\draw (9,-0.5) -- (9.25,-0.5);
\draw (9,-1) -- (9.25,-0.75);
\draw[fill=yellow!30]  (10,0) -- (11,0) -- (11,-1) -- (10,-1) -- (10,0);
\draw (10,0) -- (10.25,-0.25);
\draw (10,-0.5) -- (10.25,-0.5);
\draw (10,-1) -- (10.25,-0.75);
\draw[fill=yellow!30]  (11,0) -- (12,0) -- (12,-1) -- (11,-1) -- (11,0);
\draw (11,0) -- (11.25,-0.25);
\draw (11,-0.5) -- (11.25,-0.5);
\draw (11,-1) -- (11.25,-0.75);
\draw[fill=yellow!30]  (3,-1) -- (4,-1) -- (4,-2) -- (3,-2) -- (3,-1);
\draw (3,-1) -- (3.25,-1.25);
\draw (3,-1.5) -- (3.25,-1.5);
\draw (3,-2) -- (3.25,-1.75);
\draw[fill=yellow!30]  (4,-2) -- (5,-2) -- (5,-3) -- (4,-3) -- (4,-2);
\draw (4,-2) -- (4.25,-2.25);
\draw (4,-2.5) -- (4.25,-2.5);
\draw (4,-3) -- (4.25,-2.75);
\draw[fill=green!30]  (0,-2) -- (1,-2) -- (1,-3) -- (0,-3) -- (0,-2);
\draw (0,-2) -- (0.25,-2.25);
\draw (0.5,-2) -- (0.5,-2.25);
\draw (1,-2) -- (0.75,-2.25);
\draw[fill=green!30]  (0,-3) -- (1,-3) -- (1,-4) -- (0,-4) -- (0,-3);
\draw (0,-3) -- (0.25,-3.25);
\draw (0.5,-3) -- (0.5,-3.25);
\draw (1,-3) -- (0.75,-3.25);
\draw[fill=green!30]  (0,-4) -- (1,-4) -- (1,-5) -- (0,-5) -- (0,-4);
\draw (0,-4) -- (0.25,-4.25);
\draw (0.5,-4) -- (0.5,-4.25);
\draw (1,-4) -- (0.75,-4.25);
\draw[fill=green!30]  (0,-7) -- (1,-7) -- (1,-8) -- (0,-8) -- (0,-7);
\draw (0,-7) -- (0.25,-7.25);
\draw (0.5,-7) -- (0.5,-7.25);
\draw (1,-7) -- (0.75,-7.25);
\draw[fill=green!30]  (0,-8) -- (1,-8) -- (1,-9) -- (0,-9) -- (0,-8);
\draw (0,-8) -- (0.25,-8.25);
\draw (0.5,-8) -- (0.5,-8.25);
\draw (1,-8) -- (0.75,-8.25);
\draw[fill=green!30]  (0,-9) -- (1,-9) -- (1,-10) -- (0,-10) -- (0,-9);
\draw (0,-9) -- (0.25,-9.25);
\draw (0.5,-9) -- (0.5,-9.25);
\draw (1,-9) -- (0.75,-9.25);
\draw[fill=green!30]  (1,-2) -- (2,-2) -- (2,-3) -- (1,-3) -- (1,-2);
\draw (1,-2) -- (1.25,-2.25);
\draw (1.5,-2) -- (1.5,-2.25);
\draw (2,-2) -- (1.75,-2.25);
\draw[fill=green!30]  (2,-2) -- (3,-2) -- (3,-3) -- (2,-3) -- (2,-2);
\draw (2,-2) -- (2.25,-2.25);
\draw (2.5,-2) -- (2.5,-2.25);
\draw (3,-2) -- (2.75,-2.25);
\draw[fill=green!30]  (2,-3) -- (3,-3) -- (3,-4) -- (2,-4) -- (2,-3);
\draw (2,-3) -- (2.25,-3.25);
\draw (2.5,-3) -- (2.5,-3.25);
\draw (3,-3) -- (2.75,-3.25);
\draw[fill=green!30]  (3,-3) -- (4,-3) -- (4,-4) -- (3,-4) -- (3,-3);
\draw (3,-3) -- (3.25,-3.25);
\draw (3.5,-3) -- (3.5,-3.25);
\draw (4,-3) -- (3.75,-3.25);
\draw[fill=green!30!yellow!30]  (0,0) -- (1,0) -- (1,-1) -- (0,-1) -- (0,0);
\draw (0.5,0) -- (0.5,-0.25);
\draw (1,0) -- (0.75,-0.25);
\draw (0,0) -- (0.25,-0.25);
\draw (0,-0.5) -- (0.25,-0.5);
\draw (0,-1) -- (0.225,-0.75);
\draw[fill=green!30!yellow!30]  (1,0) -- (2,0) -- (2,-1) -- (1,-1) -- (1,0);
\draw (1.5,0) -- (1.5,-0.25);
\draw (2,0) -- (1.75,-0.25);
\draw (1,0) -- (1.25,-0.25);
\draw (1,-0.5) -- (1.25,-0.5);
\draw (1,-1) -- (1.225,-0.75);
\draw[fill=green!30!yellow!30]  (2,0) -- (3,0) -- (3,-1) -- (2,-1) -- (2,0);
\draw (2.5,0) -- (2.5,-0.25);
\draw (3,0) -- (2.75,-0.25);
\draw (2,0) -- (2.25,-0.25);
\draw (2,-0.5) -- (2.25,-0.5);
\draw (2,-1) -- (2.225,-0.75);
\draw[fill=green!30!yellow!30]  (7,0) -- (8,0) -- (8,-1) -- (7,-1) -- (7,0);
\draw (7.5,0) -- (7.5,-0.25);
\draw (8,0) -- (7.75,-0.25);
\draw (7,0) -- (7.25,-0.25);
\draw (7,-0.5) -- (7.25,-0.5);
\draw (7,-1) -- (7.225,-0.75);
\draw[fill=green!30!yellow!30]  (0,-1) -- (1,-1) -- (1,-2) -- (0,-2) -- (0,-1);
\draw (0.5,-1) -- (0.5,-1.25);
\draw (1,-1) -- (0.75,-1.25);
\draw (0,-1) -- (0.25,-1.25);
\draw (0,-1.5) -- (0.25,-1.5);
\draw (0,-2) -- (0.225,-1.75);
\draw[fill=green!30!yellow!30]  (1,-1) -- (2,-1) -- (2,-2) -- (1,-2) -- (1,-1);
\draw (1.5,-1) -- (1.5,-1.25);
\draw (2,-1) -- (1.75,-1.25);
\draw (1,-1) -- (1.25,-1.25);
\draw (1,-1.5) -- (1.25,-1.5);
\draw (1,-2) -- (1.225,-1.75);
\draw[fill=green!30!yellow!30]  (2,-1) -- (3,-1) -- (3,-2) -- (2,-2) -- (2,-1);
\draw (2.5,-1) -- (2.5,-1.25);
\draw (3,-1) -- (2.75,-1.25);
\draw (2,-1) -- (2.25,-1.25);
\draw (2,-1.5) -- (2.25,-1.5);
\draw (2,-2) -- (2.225,-1.75);
\draw[fill=green!30!yellow!30]  (0,-5) -- (1,-5) -- (1,-6) -- (0,-6) -- (0,-5);
\draw (0.5,-5) -- (0.5,-5.25);
\draw (1,-5) -- (0.75,-5.25);
\draw (0,-5) -- (0.25,-5.25);
\draw (0,-5.5) -- (0.25,-5.5);
\draw (0,-6) -- (0.225,-5.75);
\draw[fill=green!30!yellow!30]  (0,-6) -- (1,-6) -- (1,-7) -- (0,-7) -- (0,-6);
\draw (0.5,-6) -- (0.5,-6.25);
\draw (1,-6) -- (0.75,-6.25);
\draw (0,-6) -- (0.25,-6.25);
\draw (0,-6.5) -- (0.25,-6.5);
\draw (0,-7) -- (0.225,-6.75);
\draw (0.5,-0.5) node{1};
\draw (1.5,-0.5) node{3};
\draw (2.5,-0.5) node{5};
\draw (3.5,-0.5) node{7};
\draw (4.5,-0.5) node{8};
\draw (5.5,-0.5) node{10};
\draw (6.5,-0.5) node{11};
\draw (7.5,-0.5) node{12};
\draw (8.5,-0.5) node{16};
\draw (9.5,-0.5) node{17};
\draw (10.5,-0.5) node{19};
\draw (11.5,-0.5) node{20};
\draw (0.5,-1.5) node{2};
\draw (1.5,-1.5) node{4};
\draw (2.5,-1.5) node{6};
\draw (3.5,-1.5) node{9};
\draw (4.5,-1.5) node{16};
\draw (5.5,-1.5) node{17};
\draw (6.5,-1.5) node{19};
\draw (7.5,-1.5) node{20};
\draw (0.5,-2.5) node{7};
\draw (1.5,-2.5) node{8};
\draw (2.5,-2.5) node{10};
\draw (3.5,-2.5) node{11};
\draw (4.5,-2.5) node{18};
\draw (0.5,-3.5) node{9};
\draw (1.5,-3.5) node{16};
\draw (2.5,-3.5) node{17};
\draw (3.5,-3.5) node{19};
\draw (4.5,-3.5) node{20};
\draw (0.5,-4.5) node{11};
\draw (1.5,-4.5) node{18};
\draw (0.5,-5.5) node{14};
\draw (1.5,-5.5) node{20};
\draw (0.5,-6.5) node{15};
\draw (0.5,-7.5) node{16};
\draw (0.5,-8.5) node{18};
\draw (0.5,-9.5) node{20};

\draw[white] (0,-10) -- (0, -13);
}
\end{tikzpicture}
\end{minipage}
\begin{minipage}{.6\textwidth}
\begin{tikzpicture}[scale=.70]
{\small 
\draw (0,0) -- (12, 0);
\draw (0,-1) -- (12,-1);
\draw (0,-2) -- (8,-2);
\draw (0,-3) -- (5,-3);
\draw (0,-4) -- (5,-4);
\draw (0,-5) -- (2,-5);
\draw (0,-6) -- (2,-6);
\draw (0,-7) -- (1,-7);
\draw (0,-8) -- (1,-8);
\draw (0,-9) -- (1,-9);
\draw (0,-10) -- (1,-10);
\draw (0,0) -- (0,-10);
\draw (1,0) -- (1,-10);
\draw (2,0) -- (2,-6);
\draw (3,0) -- (3,-4);
\draw (4,0) -- (4,-4);
\draw (5,0) -- (5,-4);
\draw (6,0) -- (6,-2);
\draw (7,0) -- (7,-2);
\draw (8,0) -- (8,-2);
\draw (9,0) -- (9,-1);
\draw (10,0) -- (10,-1);
\draw (11,0) -- (11,-1);
\draw (12,0) -- (12,-1);
\draw[fill=yellow!30]  (4,1) -- (5,1) -- (5,0) -- (4,0) -- (4,1);
\draw (4,1) -- (4.25,0.75);
\draw (4,0.5) -- (4.25,0.5);
\draw (4,0) -- (4.25,0.25);
\draw[fill=yellow!30]  (5,1) -- (6,1) -- (6,0) -- (5,0) -- (5,1);
\draw (5,1) -- (5.25,0.75);
\draw (5,0.5) -- (5.25,0.5);
\draw (5,0) -- (5.25,0.25);
\draw[fill=yellow!30]  (6,1) -- (7,1) -- (7,0) -- (6,0) -- (6,1);
\draw (6,1) -- (6.25,0.75);
\draw (6,0.5) -- (6.25,0.5);
\draw (6,0) -- (6.25,0.25);
\draw[fill=yellow!30]  (7,1) -- (8,1) -- (8,0) -- (7,0) -- (7,1);
\draw (7,1) -- (7.25,0.75);
\draw (7,0.5) -- (7.25,0.5);
\draw (7,0) -- (7.25,0.25);
\draw[fill=yellow!30]  (8,0) -- (9,0) -- (9,-1) -- (8,-1) -- (8,0);
\draw (8,0) -- (8.25,-0.25);
\draw (8,-0.5) -- (8.25,-0.5);
\draw (8,-1) -- (8.25,-0.75);

\draw[fill=green!30]  (0,-2) -- (1,-2) -- (1,-3) -- (0,-3) -- (0,-2);
\draw (0,-2) -- (0.25,-2.25);
\draw (0.5,-2) -- (0.5,-2.25);
\draw (1,-2) -- (0.75,-2.25);
\draw[fill=green!30]  (0,-3) -- (1,-3) -- (1,-4) -- (0,-4) -- (0,-3);
\draw (0,-3) -- (0.25,-3.25);
\draw (0.5,-3) -- (0.5,-3.25);
\draw (1,-3) -- (0.75,-3.25);
\draw[fill=green!30]  (0,-4) -- (1,-4) -- (1,-5) -- (0,-5) -- (0,-4);
\draw (0,-4) -- (0.25,-4.25);
\draw (0.5,-4) -- (0.5,-4.25);
\draw (1,-4) -- (0.75,-4.25);
\draw[fill=green!30]  (0,-7) -- (1,-7) -- (1,-8) -- (0,-8) -- (0,-7);
\draw (0,-7) -- (0.25,-7.25);
\draw (0.5,-7) -- (0.5,-7.25);
\draw (1,-7) -- (0.75,-7.25);
\draw[fill=green!30]  (0,-8) -- (1,-8) -- (1,-9) -- (0,-9) -- (0,-8);
\draw (0,-8) -- (0.25,-8.25);
\draw (0.5,-8) -- (0.5,-8.25);
\draw (1,-8) -- (0.75,-8.25);
\draw[fill=green!30]  (0,-9) -- (1,-9) -- (1,-10) -- (0,-10) -- (0,-9);
\draw (0,-9) -- (0.25,-9.25);
\draw (0.5,-9) -- (0.5,-9.25);
\draw (1,-9) -- (0.75,-9.25);
\draw[fill=green!30]  (2,-2) -- (3,-2) -- (3,-3) -- (2,-3) -- (2,-2);
\draw (2,-2) -- (2.25,-2.25);
\draw (2.5,-2) -- (2.5,-2.25);
\draw (3,-2) -- (2.75,-2.25);
\draw[fill=green!30]  (0,0) -- (1,0) -- (1,-1) -- (0,-1) -- (0,0);
\draw (0.5,0) -- (0.5,-0.25);
\draw (1,0) -- (0.75,-0.25);
\draw (0,0) -- (0.25,-0.25);
\draw (0,-0.5) -- (0.25,-0.5);
\draw (0,-1) -- (0.225,-0.75);
\draw[fill=green!30]  (2,0) -- (3,0) -- (3,-1) -- (2,-1) -- (2,0);
\draw (2.5,0) -- (2.5,-0.25);
\draw (3,0) -- (2.75,-0.25);
\draw (2,0) -- (2.25,-0.25);
\draw (2,-0.5) -- (2.25,-0.5);
\draw (2,-1) -- (2.225,-0.75);
\draw[fill=green!30]  (7,0) -- (8,0) -- (8,-1) -- (7,-1) -- (7,0);
\draw (7.5,0) -- (7.5,-0.25);
\draw (8,0) -- (7.75,-0.25);
\draw (7,0) -- (7.25,-0.25);
\draw (7,-0.5) -- (7.25,-0.5);
\draw (7,-1) -- (7.225,-0.75);
\draw[fill=green!30]  (0,-1) -- (1,-1) -- (1,-2) -- (0,-2) -- (0,-1);
\draw (0.5,-1) -- (0.5,-1.25);
\draw (1,-1) -- (0.75,-1.25);
\draw (0,-1) -- (0.25,-1.25);
\draw (0,-1.5) -- (0.25,-1.5);
\draw (0,-2) -- (0.225,-1.75);
\draw[fill=yellow!30]  (3,-2) -- (4,-2) -- (4,-3) -- (3,-3) -- (3,-2);
\draw (3,-2) -- (3.25,-2.25);
\draw (3,-2.5) -- (3.25,-2.5);
\draw (3,-3) -- (3.25,-2.75);
\draw[fill=green!30]  (2,-1) -- (3,-1) -- (3,-2) -- (2,-2) -- (2,-1);
\draw (2.5,-1) -- (2.5,-1.25);
\draw (3,-1) -- (2.75,-1.25);
\draw (2,-1) -- (2.25,-1.25);
\draw (2,-1.5) -- (2.25,-1.5);
\draw (2,-2) -- (2.225,-1.75);
\draw[fill=green!30]  (3,-3) -- (4,-3) -- (4,-4) -- (3,-4) -- (3,-3);
\draw (3,-3) -- (3.25,-3.25);
\draw (3.5,-3) -- (3.5,-3.25);
\draw (4,-3) -- (3.75,-3.25);
\draw[fill=green!30]  (0,-5) -- (1,-5) -- (1,-6) -- (0,-6) -- (0,-5);
\draw (0.5,-5) -- (0.5,-5.25);
\draw (1,-5) -- (0.75,-5.25);
\draw (0,-5) -- (0.25,-5.25);
\draw (0,-5.5) -- (0.25,-5.5);
\draw (0,-6) -- (0.225,-5.75);
\draw[fill=green!30]  (0,-6) -- (1,-6) -- (1,-7) -- (0,-7) -- (0,-6);
\draw (0.5,-6) -- (0.5,-6.25);
\draw (1,-6) -- (0.75,-6.25);
\draw (0,-6) -- (0.25,-6.25);
\draw (0,-6.5) -- (0.25,-6.5);
\draw (0,-7) -- (0.225,-6.75);
\draw (4.5,.5) node{2};
\draw (5.5,.5) node{4};
\draw (6.5,.5) node{6};
\draw (7.5,.5) node{9};
\draw (0.5,-0.5) node{1};
\draw (2.5,-0.5) node{5};
\draw (7.5,-0.5) node{12};
\draw (8.5,-0.5) node{18};
\draw (0.5,-1.5) node{2};
\draw (2.5,-1.5) node{6};
\draw (0.5,-2.5) node{7};
\draw (2.5,-2.5) node{10};
\draw (3.5,-2.5) node{14};
\draw (0.5,-3.5) node{9};
\draw (3.5,-3.5) node{19};
\draw (0.5,-4.5) node{11};
\draw (0.5,-5.5) node{14};
\draw (0.5,-6.5) node{15};
\draw (0.5,-7.5) node{16};
\draw (0.5,-8.5) node{18};
\draw (0.5,-9.5) node{20};

\node at (0.5, 0.3) {$\ell=1$};
\node at (2.5, 0.3) {{ \color{violet}$\ell=3$}};
\node at (6, 1.3) {{\color{blue}$\ell=4$}};

}
{\scriptsize

\draw[gray] (6.5 - 5, -6.5-5) -- (6.5 + 4, -6.5+4);
\draw[gray] (6.5 - 5 + 10/8, -6.5-5-10/8) -- (6.5 + 5+ 10/8, -6.5+5-10/8);
\draw[gray] (6.5 - 5 + 14/8, -6.5-5-14/8) -- (6.5 + 5+ 14/8, -6.5+5-14/8);
\draw[gray] (6.5 - 5 + 20/8, -6.5-5-20/8) -- (6.5 + 5+ 20/8, -6.5+5-20/8);

\node[rotate=45, right] at (6.5 + 4, -6.5+4) {$t=0$};
\node[rotate=45, right] at (6.5 + 5+ 10/8, -6.5+5-10/8)  {$t=10$};
\node[rotate=45, right] at (6.5 + 5+ 14/8, -6.5+5-14/8)  {$t=14$};
\node[rotate=45,  right] at (6.5 + 5+ 20/8, -6.5+5-20/8) {$t=20$};

\draw[fill=gray!40] (7.0,-6.0) circle (7pt);
\draw (7.0,-6.0)node{$1$};
\draw[fill=violet!40] (8.0,-5.0) circle (7pt);
\draw (8.0,-5.0)node{$3$};
\draw[fill=blue!40] (8.5,-4.5) circle (7pt);
\draw (8.5,-4.5)node{$4$};

\draw[fill=gray!40] (6.25,-9.25) circle (7pt);
\draw (6.25,-9.25)node{$1$};
\draw[fill=violet!40] (7.75,-7.75) circle (7pt);
\draw (7.75,-7.75)node{$3$};
\draw[fill=blue!40] (11.75,-3.75) circle (7pt);
\draw (11.75,-3.75)node{$4$};

\draw[fill=gray!40] (5.75,-10.75) circle (7pt);
\draw (5.75,-10.75)node{$1$};
\draw[fill=violet!40] (8.75,-7.75) circle (7pt);
\draw (8.75,-7.75)node{$3$};
\draw[fill=blue!40] (11.75,-4.75) circle (7pt);
\draw (11.75,-4.75)node{$4$};

\draw[fill=gray!40] (4.5,-13.5) circle (7pt);
\draw (4.5,-13.5)node{$1$};
\draw[fill=violet!40] (9.0,-9.0) circle (7pt);
\draw (9.0,-9.0)node{$3$};
\draw[fill=blue!40] (13.0,-5.0) circle (7pt);
\draw (13.0,-5.0)node{$4$};

}
\end{tikzpicture}
\end{minipage}

\vspace{10pt}
 
 \small
 \setlength{\tabcolsep}{0.15em}
 \begin{tabular}{llr|lr|lr|lr|lr|lr|lr|lr|lr|lr|lr|lr|lr|lr|lr|lr|lr|lr|lr|lr}
&\multicolumn{2}{c}{{\ \ \ 1 \ \ \ } } & \multicolumn{2}{c}{ \ \ \  2 \ \ \ } & \multicolumn{2}{c}{\ \ \ 3 \ \ \ } & \multicolumn{2}{c}{ \ \ \ 4 \ \ \ } & \multicolumn{2}{c}{\ \ \ 5 \ \ \ } & \multicolumn{2}{c}{\ \ \ 6 \ \ \ } 
& \multicolumn{2}{c}{7} & \multicolumn{2}{c}{8}
& \multicolumn{2}{c}{9} & \multicolumn{2}{c}{10}
& \multicolumn{2}{c}{11} & \multicolumn{2}{c}{12} & \multicolumn{2}{c}{13} & \multicolumn{2}{c}{14} & \multicolumn{2}{c}{15} & \multicolumn{2}{c}{16} 
& \multicolumn{2}{c}{17} & \multicolumn{2}{c}{18}
& \multicolumn{2}{c}{19} & \multicolumn{2}{c}{20}
\\
\\[-12pt]
 \multicolumn{3}{l}  {\textit{Layer 1}} 
 \\[-6pt]
 \\
$\ell = 1$ &\cellcolor{green!30}$0$ & \cellcolor{green!30} $10$ 
& \cellcolor{green!30}$0$ & \cellcolor{green!30} $10$ 
& \cellcolor{white}$0$ & \cellcolor{white} $9$ 
& \cellcolor{white}$1$ & \cellcolor{white} $8$ 
& \cellcolor{white}$2$ & \cellcolor{white} $7$ 
& \cellcolor{white}$3$ & \cellcolor{white} $6$ 
& \cellcolor{green!30}$4$ & \cellcolor{green!30} $6$ 
& \cellcolor{white}$4$ & \cellcolor{white} $5$ 
& \cellcolor{green!30}$5$ & \cellcolor{green!30} $5$ 
& \cellcolor{white}$5$ & \cellcolor{white} $4$ 
& \cellcolor{green!30}$6$ & \cellcolor{green!30} $4$ 
& \cellcolor{white}$6$ & \cellcolor{white} $3$ 
& \cellcolor{white}$7$ & \cellcolor{white} $2$ 
& \cellcolor{green!30}$8$ & \cellcolor{green!30} $2$ 
& \cellcolor{green!30}$8$ & \cellcolor{green!30} $2$ 
& \cellcolor{green!30}$8$ & \cellcolor{green!30} $2$ 
& \cellcolor{white}$8$ & \cellcolor{white} $1$ 
& \cellcolor{green!30}$9$ & \cellcolor{green!30} $1$ 
& \cellcolor{white}$9$ & \cellcolor{white} $0$ 
& \cellcolor{green!30}$10$ & \cellcolor{green!30} $0$  \\
\\[-6pt]
 \multicolumn{3}{l}  {\color{red!50!violet} \textit{Layer 2}} 
 \\[-6pt]
\\
${\color{red} \ell = 2}$ & \cellcolor{white}{\color{red} $1$} & \cellcolor{white} {\color{red} $18$} 
& \cellcolor{white}{\color{red} $2$} & \cellcolor{white} {\color{red} $17$} 
& \cellcolor{green!30}{\color{red} $3$} & \cellcolor{green!30} {\color{red} $17$} 
& \cellcolor{green!30}{\color{red} $3$} & \cellcolor{green!30} {\color{red} $17$} 
& \cellcolor{white}{\color{red} $3$} & \cellcolor{white} {\color{red} $16$} 
& \cellcolor{white}{\color{red} $4$} & \cellcolor{white} {\color{red} $15$} 
& \cellcolor{white}{\color{red} $5$} & \cellcolor{white} {\color{red} $14$} 
& \cellcolor{green!30}{\color{red} $6$} & \cellcolor{green!30} {\color{red} $14$} 
& \cellcolor{white}{\color{red} $6$} & \cellcolor{white} {\color{red} $13$} 
& \cellcolor{white}{\color{red} $7$} & \cellcolor{white} {\color{red} $12$} 
& \cellcolor{white}{\color{red} $8$} & \cellcolor{white} {\color{red} $11$} 
& \cellcolor{white}{\color{red} $9$} & \cellcolor{white} {\color{red} $10$} 
& \cellcolor{white}{\color{red} $10$} & \cellcolor{white} {\color{red} $9$} 
& \cellcolor{white}{\color{red} $11$} & \cellcolor{white} {\color{red} $8$} 
& \cellcolor{yellow!30}{\color{red} $12$} & \cellcolor{yellow!30} {\color{red} $6$} 
& \cellcolor{white}{\color{red} $14$} & \cellcolor{white} {\color{red} $5$} 
& \cellcolor{green!30}{\color{red} $15$} & \cellcolor{green!30} {\color{red} $5$} 
& \cellcolor{white}{\color{red} $15$} & \cellcolor{white} {\color{red} $4$} 
& \cellcolor{white}{\color{red} $16$} & \cellcolor{white} {\color{red} $3$} 
& \cellcolor{white}{\color{red} $17$} & \cellcolor{white} {\color{red} $2$} 
\\
\\
${\color{violet} \ell = 3}$ & \cellcolor{white}{\color{violet} $2$} & \cellcolor{white} {\color{violet} $17$} 
& \cellcolor{white}{\color{violet} $3$} & \cellcolor{white} {\color{violet} $16$} 
& \cellcolor{white}{\color{violet} $4$} & \cellcolor{white} {\color{violet} $15$} 
& \cellcolor{white}{\color{violet} $5$} & \cellcolor{white} {\color{violet} $14$} 
& \cellcolor{green!30}{\color{violet} $6$} & \cellcolor{green!30} {\color{violet} $14$} 
& \cellcolor{green!30}{\color{violet} $6$} & \cellcolor{green!30} {\color{violet} $14$} 
& \cellcolor{white}{\color{violet} $6$} & \cellcolor{white} {\color{violet} $13$} 
& \cellcolor{white}{\color{violet} $7$} & \cellcolor{white} {\color{violet} $12$} 
& \cellcolor{white}{\color{violet} $8$} & \cellcolor{white} {\color{violet} $11$} 
& \cellcolor{green!30}{\color{violet} $9$} & \cellcolor{green!30} {\color{violet} $11$} 
& \cellcolor{white}{\color{violet} $9$} & \cellcolor{white} {\color{violet} $10$} 
& \cellcolor{white}{\color{violet} $10$} & \cellcolor{white} {\color{violet} $9$} 
& \cellcolor{white}{\color{violet} $11$} & \cellcolor{white} {\color{violet} $8$} 
& \cellcolor{yellow!30}{\color{violet} $12$} & \cellcolor{yellow!30} {\color{violet} $6$} 
& \cellcolor{white}{\color{violet} $14$} & \cellcolor{white} {\color{violet} $5$} 
& \cellcolor{white}{\color{violet} $15$} & \cellcolor{white} {\color{violet} $4$} 
& \cellcolor{white}{\color{violet} $16$} & \cellcolor{white} {\color{violet} $3$} 
& \cellcolor{white}{\color{violet} $17$} & \cellcolor{white} {\color{violet} $2$} 
& \cellcolor{green!30}{\color{violet} $18$} & \cellcolor{green!30} {\color{violet} $2$} 
& \cellcolor{white}{\color{violet} $18$} & \cellcolor{white} {\color{violet} $1$} 
\\
\\[-6pt]
 \multicolumn{3}{l}  {\color{blue} \textit{Layer 3}} 
 \\[-6pt]
\\
${\color{blue} \ell = 4}$ \qquad \qquad & \cellcolor{white}{\color{blue} $3$} & \cellcolor{white} {\color{blue} $26$} 
& \cellcolor{yellow!30}{\color{blue} $4$} & \cellcolor{yellow!30} {\color{blue} $24$} 
& \cellcolor{white}{\color{blue} $6$} & \cellcolor{white} {\color{blue} $23$} 
& \cellcolor{yellow!30}{\color{blue} $7$} & \cellcolor{yellow!30} {\color{blue} $21$} 
& \cellcolor{white}{\color{blue} $9$} & \cellcolor{white} {\color{blue} $20$} 
& \cellcolor{yellow!30}{\color{blue} $10$} & \cellcolor{yellow!30} {\color{blue} $18$} 
& \cellcolor{white}{\color{blue} $12$} & \cellcolor{white} {\color{blue} $17$} 
& \cellcolor{white}{\color{blue} $13$} & \cellcolor{white} {\color{blue} $16$} 
& \cellcolor{yellow!30}{\color{blue} $14$} & \cellcolor{yellow!30} {\color{blue} $14$} 
& \cellcolor{white}{\color{blue} $16$} & \cellcolor{white} {\color{blue} $13$} 
& \cellcolor{white}{\color{blue} $17$} & \cellcolor{white} {\color{blue} $12$} 
& \cellcolor{green!30}{\color{blue} $18$} & \cellcolor{green!30} {\color{blue} $12$} 
& \cellcolor{white}{\color{blue} $18$} & \cellcolor{white} {\color{blue} $11$} 
& \cellcolor{white}{\color{blue} $19$} & \cellcolor{white} {\color{blue} $10$} 
& \cellcolor{white}{\color{blue} $20$} & \cellcolor{white} {\color{blue} $9$} 
& \cellcolor{white}{\color{blue} $21$} & \cellcolor{white} {\color{blue} $8$} 
& \cellcolor{white}{\color{blue} $22$} & \cellcolor{white} {\color{blue} $7$} 
& \cellcolor{yellow!30}{\color{blue} $23$} & \cellcolor{yellow!30} {\color{blue} $5$} 
& \cellcolor{white}{\color{blue} $25$} & \cellcolor{white} {\color{blue} $4$} 
& \cellcolor{white}{\color{blue} $26$} & \cellcolor{white} {\color{blue} $3$} 
\end{tabular}

 \end{landscape}

\noindent
\begin{minipage}{0.56\textwidth}

\begin{example}[$g = 20, \vec{e} = (-6, -4, -2, -2, 0)$]
Let $T$ be the tableau on the previous page.
The locus $W^T(X)$ is $1$-dimensional, corresponding to the fact that $13$ does not appear in $T$, and so $L^{13}$ may be any degree $10$ line bundle on $E^{13}$. We assume $L^{13}$ is not a linear combination of the nodes. The table below the tableau on the previous page lists the new ramification indices $\naa{i}{\ell}$ and $\nbb{i}{\ell}$. The table to the right lists all ramification indices for $V_j(E^{14})$ at each layer $j$. The ramification indices at $p^{13}$ are on the left and those for $p^{14}$ are on the right.  The new ramification indices are bold, and images of a section in higher layers are given the same color.

\hspace{\myindent}The number $14$ is decreasing for the first truncation ($\ell_- = 1 = (1, 1)$), so we have $L^{14} = \O_{E^{14}}(\mathbf{8} p^{13} + \mathbf{2} p^{14})$.

\hspace{\myindent}There are two new sections in layer $2$, drawn in red ($\ell = 2 = (2, 1)$) and purple ($\ell = 3 = (2, 2)$).
The number $14$ is increasing for the third truncation ($\ell_+ = 3 = (2, 2)$), and the old ramification indices that are
one more than the new ramification indices are as follows:
\begin{align*}
{\color{violet} \naa{13}{2,2}} + 1 = {\color{violet} \mathbf{12}} + 1  &= \circled{13}= \mathbf{8} + 5 = \naa{13}{1,1} + k \\
{\color{violet} \nbb{14}{2,2}} + 1 = {\color{violet} \mathbf{6}} + 1  &= \circled{7} = \mathbf{2} + 5= \nbb{14}{1,1} + k
\end{align*}

\hspace{\myindent}There is one new section in layer $3$ \mbox{($\ell = 4 = (3, 1)$)}, colored blue.
\end{example}
\end{minipage}
\hfill
\begin{minipage}{0.4\textwidth}
\begin{center}

\hspace{.17in}
\begin{tikzpicture}
\draw (2, 0) .. controls (3, -1) and (4, -1) .. (5, 0);
\filldraw (2.5, -0.42) circle[radius=0.03];
\filldraw (4.5, -0.42) circle[radius=0.03];
\draw (2.4, -0.7) node{$p^{13}$};
\draw (4.5, -0.7) node{$p^{14}$};
\draw (3.5, 0) node{$E^{14}$};
\end{tikzpicture}

\vspace{.1in}
\hspace{-0.715in}
 \setlength{\tabcolsep}{0.35em}
\begin{tabular}{c | c c c |}
&&&\\[-6pt]
\textit{Layer $j = 1$} & $\mathbf{8}$ && $\mathbf{2}$\\ 
&&& \\[-6pt]
\hline
&&& \\[-6pt]
{\color{red!50!violet} \textit{Layer $j = 2$}} & $8$ && $2$ \\
& ${\color{red} \mathbf{11}}$ && ${\color{violet} \mathbf{6}}$ \\
& ${\color{violet} \mathbf{12}}$ && $\circled{7}$ \\
& $\circled{13}$ && ${\color{red} \mathbf{8}}$ \\
& $18$ && $12$  \\
&&& \\[-6pt]
\hline
&&& \\[-6pt]
{\color{blue} \textit{Layer $j = 3$}} & $8$ && $2$ \\
& ${\color{red} 11}$ && ${\color{violet} 6}$ \\
& ${\color{violet} 12}$ && $7$ \\
& $13$ && ${\color{red} 8}$ \\
& ${\color{red} 16}$ && ${\color{blue} \mathbf{10}}$ \\
& ${\color{violet} 17}$ && ${\color{violet} 11}$ \\ 
& $18$ && $12$ \\
& ${\color{blue} \mathbf{19}}$ && {\color{red} $13$} \\
& ${\color{red} 21}$ &&  {\color{violet} $16$} \\
& ${\color{violet} 22}$ && $17$ \\
& $23$ && ${\color{red} 18}$ \\
& $28$ && $22$
\end{tabular}
\end{center}
\end{minipage}

\subsection{Regeneration for refined, non-colliding nested limit linear series}

\begin{thm}[Regeneration] \label{regen}
Suppose $\naa{i}{\ell}$ and $\nbb{i}{\ell}$ are the new ramification indices at $p^i$ for a refined, non-colliding $\vec{e}$-nested limit linear series. Let $ \mathcal{X} \rightarrow \P \rightarrow B$ be the family of curves constructed in Section \ref{our_degen}.
There is a quasi-projective scheme $\widetilde{W}$ over $B$ whose general fiber is contained in the space of $\vec{e}$-nested linear series on $\mathcal{X}^*$
and whose special fiber is the space of $\vec{e}$-nested limit linear series on $X$ with the (strictly) specified ramification. Every component of $\widetilde{W}$ has dimension at least $\dim \Pic^d(\mathcal{X}/B) - u(\vec{e})$. 
\end{thm}

\begin{proof}
We will construct a variety $\widetilde{W}^{\mathrm{fr}}$ which parametrizes compatible framings of nested linear series and surjects onto the desired $\widetilde{W}$. 
Set $\chi \colonequals \deg(\O(\vec{e})) + k = d - g + 1$. We retain notation as above so $d_1 > \cdots > d_\s$ are the distinct degrees appearing in $\vec{e}$ and $m_1, \ldots, m_\s$ are the corresponding multiplicities. 
Let $\Pic \colonequals \Pic^d(\mathcal{X}/B)$ and label maps as in the diagram below.

\begin{center}
\begin{tikzcd}
\X \times_B \Pic \arrow{r} \arrow{dd}[swap]{\pi} & \X \arrow{d}{\mathfrak{f}}\\
& \P \arrow{d}{\varphi} \\
\Pic \arrow{r}[swap]{\eta} & B
\end{tikzcd}
\end{center}

Let $\mathcal{L}$ be the universal limit line bundle on $\X \times_B \Pic \xrightarrow{\pi} \Pic$.
Recall that $\mathcal{L}^i = \mathcal{L}_{(0, \ldots, 0, d, 0, \ldots 0)}$
has degree $d$ on component $E^i$ and degree $0$ on all other components. 
In addition, recall that $\O_{\P}(n)^i = \O_{\P}(n)_{(0, \ldots, 0, n, 0, \ldots, 0)}$ is isomorphic to $\O_{\pp^1}(n)$ on the smooth fibers of $\P \rightarrow B$; on the central fiber, it has degree $n$ on $P^i$ and degree $0$ on all other components.
(See Example~\ref{limit-Om}.)

Let $D$ be an effective divisor of relative degree $N$ on $\mathcal{X} \rightarrow B$ so that $D$ meets each component of the central fiber with sufficiently large degree. By slight abuse of notation we denote by $D$ the pullback of this divisor to $\mathcal{X} \times_B \Pic$.
By cohomology and base change, $\pi_* \mathcal{L}(-d_j)^i(D)$ is a vector bundle
on $\Pic^{d-d_jk+N}(\mathcal{X}/B)$,
which we identify with $\Pic$ via tensoring with $\O_{\X}(-d_j)(D)$;
its rank is $N+\chi - d_jk$.

For each component $E^i$ of the central fiber, we are going to build a tower
\[G^i_{\s-1} \xrightarrow{\psi^i_{\s-1}} G^i_{\s-2} \xrightarrow{\psi^i_{\s-2}} \cdots \rightarrow G^i_1 \xrightarrow{\psi^i_1} G^i_0 \colonequals \Pic,\]
where each $G^i_j$ is a Grassmann bundle $\Gr(m_j, \mathcal{Q}^i_j)$ for $\Q^i_j$ a vector bundle over an open $U^i_{j-1} \subset G^i_{j-1}$. We will write $\mathcal{S}^i_j$ for the tautological subbundle on $G^i_j$.
This tower of Grassmann bundles will parametrize $\vec{e}$-nested linear series on $E^i$. The bundle $\mathcal{S}^i_j$ will correspond to the space of ``new sections at layer $j$'' (i.e.\ sections at layer $j$ modulo those coming from lower layers).

We build the tower by constructing the $\Q^i_j$ inductively.  By convention, set $G^i_0 \colonequals \Pic$ and $U^i_0 \colonequals \Pic$.
We begin by defining $\Q^i_1$ as the push forward $\mathcal{Q}^i_1 \colonequals \pi_*\mathcal{L}(-d_1)^i(D)$, which we have seen above is a vector bundle.

Now suppose, by induction, that we have defined $\mathcal{Q}^i_{j-1}$ 
as a quotient
\[(\psi^i_1 \circ \cdots \circ \psi^i_{j-2})^*\pi_*\mathcal{L}(- d_{j-1})^i(D) \rightarrow \Q^i_{j-1},\]
defined on some open $U^i_{j-2} \subset G^i_{j-2}$. 
Let $(\mathcal{S}^i_{j-1})'$ be the pullback of the tautological bundle,
i.e.\ the bundle so that the diagram below is a fiber square:
\begin{center}
\begin{tikzcd}
(\psi^i_1 \circ \cdots \circ \psi^i_{j-1})^*\pi_*\mathcal{L}(- d_{j-1})^i(D) \arrow{r} &(\psi^i_{j-1})^*\Q^i_{j-1} \\
(\mathcal{S}_{j-1}^i)' \arrow[hook]{u} \arrow{r} & \arrow[hook]{u} \mathcal{S}^i_{j-1}.
\end{tikzcd}
\end{center}
The bundle $(\mathcal{S}^i_{j-1})'$ will correspond to the full linear series at layer $j-1$ (not just the new sections).

The layer $j$ comparison maps of \eqref{im} fit together in our family as the map
\begin{center}
\begin{tikzcd}
(\mathcal{S}^i_{j-1})' \otimes (\eta \circ \psi_1^i \circ \cdots \circ \psi^i_{j-1})^*\varphi_* \O_{\P}(d_{j-1} - d_{j})^i  \arrow{d} \\
 (\psi^i_{1} \circ \cdots \circ \psi^i_{j-1})^*\pi_*\mathcal{L}(-d_{j-1})^i(D) \otimes (\eta \circ \psi_1^i \circ \cdots \circ \psi^i_{j-1})^*\varphi_* \O_{\P}(d_{j-1} - d_{j})^i   \arrow{d} \\
 (\psi^i_{1} \circ \cdots \circ \psi^i_{j-1})^*\pi_*\mathcal{L}(-d_j)^i(D).
\end{tikzcd}
\end{center}
As in the proof of Lemma \ref{max-dim}, the above composition has rank at most
$h^0(\O(\vec{e})(-d_j)) - m_j$.
Define $U^{i}_{j-1} \subset G^i_{j-1}$ to be the open where the composition has exactly this rank. On $U^{i}_{j-1}$,
define $\mathcal{Q}^i_{j}$ to be the cokernel of the compositon.
Its rank $q_j \colonequals \rk \mathcal{Q}^i_{j}$ is
\begin{align}
q_j &= \rk(\pi_*\mathcal{L}(-d_j)^i(D)) - (h^0(\O(\vec{e})(-d_j)) - m_j) \notag \\
&= N+ \chi(\O(\vec{e})(-d_j)) - (h_0(\O(\vec{e})(-d_j)) - m_j) \notag \\
&= N + m_j -h^1(\O(\vec{e})(-d_j)). \label{qi}
\end{align}

Now we introduce a space parametrizing ``lifted projective frames'' over our tower of Grassmann bundles.
Recalling that $\Q^i_j$ is a quotient of $(\psi^i_1 \circ \cdots \circ \psi^i_{j-1})^* \pi_* \L(-d_j)^i(D)$, let $\widetilde{G}^i_j \rightarrow G^i_j$ be the space of lifts of $m_j$-dimensional subspaces of $\Q^i_j$ to $m_j$-dimensional subspaces of $(\psi^i_1 \circ \cdots \circ \psi^i_{j-1})^* \pi_* \L(-d_j)^i(D)$
(so $\widetilde{G}^i_j$ is an open inside $\Gr(m_j, (\psi^i_1 \circ \cdots \circ \psi^i_{j-1})^* \pi_* \L(-d_j)^i(D))$).
Let $\widetilde{\mathcal{S}}^i_j$ denote the tautological bundle on $\widetilde{G}^i_j$, and let $\Fr(\widetilde{\mathcal{S}}^i_j) \rightarrow \widetilde{G}^i_j$ be the space of projective frames of $\widetilde{\mathcal{S}}^i_j$.

For $p$ a node on a component $E^i$, we inductively define
\[F^{i,p}_1 \colonequals  \Fr(\widetilde{\mathcal{S}}^i_1) \quad \text{and} \quad F^{i,p}_j \colonequals F^{i,p}_{j-1} \times_{G^i_{j-1}} \Fr(\widetilde{\mathcal{S}}^i_j),\]
which encodes framing information for component $i$ of all sections up to layer $j$. Note that $F^{i, p}_j$ does not depend on $p$ --- however, we include the $p$ to highlight that we will be imposing conditions on these frames at the node $p$.
We define a ``master frame space"
\[F\colonequals F^{1, p^1}_{s-1} \times F^{2, p^1}_{s-1} \times F^{2, p^2}_{s-1} \times F^{3, p^2}_{s-1} \times F^{3, p^3}_{s-1} \times \cdots   \times F^{g-1, p^{g-2}}_{s-1} \times F^{g-1, p^{g-1}}_{s-1} \times F^{g, p^{g-1}}_{s-1}\]
where all products are over $\Pic$. This maps to
\[G \colonequals G^1_{s-1} \times G^2_{s-1} \times G^2_{s-1} \times G^3_{s-1} \times G^3_{s-1} \times \cdots \times  G^{g-1}_{s-1} \times G^{g-1}_{s-1} \times G^{g}_{s-1}.\]

Next we compute $\dim F$.
The relative dimension of $\Fr(\widetilde{\mathcal{S}}^i_j)$ over $G^i_j$ is
\begin{equation} \label{FoG}
m_j \cdot (\rk \pi_* \L(-d_j)^i(D) - 1) = m_j \cdot [h^0(\O(\vec{e})(-d_j)) - 1].
\end{equation}
Since each $\psi^i_{j}$ is relative dimension $m_j(q_j - m_j)$,
we have
\[\dim F_{s-1}^{i,p} = \dim \Pic + \sum_{j=1}^{\s-1} m_j \cdot [h^0(\O(\vec{e})(-d_j)) - 1] + m_j(q_j - m_j) .\]
All together, we find that
\begin{align}
\dim F &= \dim \Pic + (2g - 2) \left(\sum_{j=1}^{\s-1}  m_j \cdot [h^0(\O(\vec{e})(-d_j)) - 1] + m_j(q_j - m_j) \right) \notag \\
&= \dim \Pic + (2g - 2) \left(\sum_{j=1}^{\s-1}   m_j \cdot [h^0(\O(\vec{e})(-d_j)) - 1] + m_j(N - h^1(\O(\vec{e})(-d_j)))\right) \notag \\
&= \dim \Pic + (2g - 2) \left(\sum_{j=1}^{\s-1} m_j (N + \chi (\O(\vec{e})(-d_j)) - 1)\right). \label{dimF}
\end{align}

We now construct a subvariety $\widetilde{W}^{\mathrm{fr}} \subset F$ that parametrizes compatible projective frames of $\vec{e}$-nested limit linear series. 
The image of $\widetilde{W}^{\mathrm{fr}}$ in $G$ will be the desired variety $\widetilde{W}$.

We will impose three types of conditions on frames; the first two will be pull-backs of conditions on $G$. 

\begin{enumerate}
\item First, we require that the spaces of sections corresponding to the same component are equal.  \item Second we require that the space of sections vanish along $D$.  
\item Finally, we impose compatibility conditions for the two frames labeled with the same node.
\end{enumerate}

\subsection{Compatibility along components} \label{comp1}

The first of these conditions is represented by restricting to a diagonal.  This
imposes
\begin{align}
\codim \left( G^1_{s-1} \times G^2_{s-1} \times \cdots \times G^g_{s-1} \hookrightarrow G \right) &=
\dim G^2_{s-1} + \ldots + \dim G^{g-1}_{s-1} \notag\\
&= (g - 2) \sum_{j=1}^{s-1} m_j(q_j - m_j) \notag \\
&= (g - 2) \sum_{j=1}^{s-1} m_j (N - h^1(\O(\vec{e})(-d_j))) \label{diag_cond}
\end{align}
conditions.

\subsection{\boldmath Vanishing along $D$}

After \'{e}tale base change, we may assume that each component of $D$ meets only one $E^i$. Let $D^i$ be the union of components of $D$ meeting $E^i$, and let $Z^i_j \subset G^i_j$
be the locus where the nested linear series vanishes on $D^i$ (up to layer $j$).
To determine an upper bound on the codimension of $Z^i_{s-1} \subset G^i_{\s-1}$, we count the number of equations needed to describe $Z^i_{j} \subset (\psi^i_{j})^{-1}(Z^i_{j-1})$ at each layer.

\begin{center}
\begin{tikzcd}[scale=0.7]
Z^i_{\s-1}  \arrow[thick, color=violet, hook]{d} \\
(\psi^i_{\s-1})^{-1}(Z^i_{\s-2}) \arrow{r} \arrow[hook]{dddd} & \cdots \\
&&Z_3^i \arrow[thick, color=violet,hook]{d} \\
&&(\psi^i_{3})^{-1}(Z_2^i) \arrow{r} \arrow[hook]{dd} & Z_2^i \arrow[thick, color=violet,hook]{d} \\
&&& (\psi_2^i)^{-1}(Z_1^i) \arrow[hook]{d} \arrow{r} & Z_1^i \arrow[thick, color=violet,hook]{d} \\
G^i_{\s-1} \arrow{r}[swap]{\psi^i_{s-1}} & \cdots \arrow{r}[swap]{\psi^i_4} & G^i_3 \arrow{r}[swap]{\psi^i_3} & G_2^i \arrow{r}[swap]{\psi^i_2} & G_1^i
\end{tikzcd}
\end{center}

We start with the locus $Z^i_1 \subset G^i_1$, which is defined by the vanishing of the composition
\[\mathcal{S}^i_{1} \rightarrow (\psi^i_1)^*\Q^i_1 = (\psi^i_1)^*\pi_*\mathcal{L}(-d_1)^i(D) \rightarrow (\psi^i_1)^*\pi_*(\mathcal{L}(-d_1)^i(D) \otimes \O_{D^i}). \]
This represents $m_1 \deg(D^i)$ equations.

On $(\psi^i_j)^{-1} Z^i_{j-1}$, evaluation $(\psi_1^i \circ \cdots \circ \psi_{j-1}^i)^* \pi_*(\L(-d_j)^i(D)) \to (\psi_1^i \circ \cdots \circ \psi_{j-1}^i)^* \pi_*(\L(-d_j)^i(D) \otimes \O_{D^i})$ factors through $\Q^i_j$.
Therefore, $Z_j^i \subset (\psi_{j-1}^i)^{-1}(Z^i_{j-1})$ is the locus where the composition
\[\mathcal{S}^i_{j} \rightarrow (\psi^i_j)^* \Q^i_{j} \rightarrow (\psi_1^i \circ \cdots \circ \psi_{j}^i)^* \pi_*(\L(-d_j)^i(D) \otimes \O_{D^i})\]
vanishes. This represents $m_{j} \deg(D^i)$ equations.

Totaling over the layers, we see that every component of $Z_{\s-1}^i$ has codimension at most
\[\codim (Z_{\s-1}^i \subset G^i_{\s-1}) \leq (m_1 + \ldots + m_{s-1}) \deg(D^i).\]
Taking the product over all components, every component of $Z_{\s-1}^1 \times_{\Pic} \cdots \times_{\Pic} Z_{\s-1}^g$
has codimension at most
\begin{equation}\label{van_D}\codim(Z_{\s-1}^1 \times_{\Pic} \cdots \times_{\Pic} Z_{\s-1}^g \subset G^1_{s-1} \times G^2_{s-1} \times \cdots \times G^g_{s-1}) \leq (m_1 + \ldots + m_{s-1})N.\end{equation}

\subsection{Compatibility at nodes} \label{comp2}
For each node $p^i$ we describe equations on $F^{i,p^i}_{s-1} \times F^{i+1,p^i}_{s-1}$
that impose ramification conditions on both frames at $p^i$ in the central fiber, and say the two frames are equal up to translation by old sections on the general fiber.

Let $\sigma^{i}_{j, 1}, \ldots, \sigma^{i}_{j,m_j}$ be coordinates on the $\Fr(\widetilde{\mathcal{S}}^i_j)$ component of $F^{i,p^{i}}_{s-1}$ (the universal framing of $\widetilde{\mathcal{S}}^i_j$ associated to $p^i$). 
Similarly, let $\lambda^{i}_{j, 1}, \ldots, \lambda^{i}_{j,m_j}$ be coordinates on the $\Fr(\widetilde{\mathcal{S}}^{i+1}_j)$ component of $F^{i+1, p^{i}}_{s-1}$ (the universal framing of $\widetilde{\mathcal{S}}^{i+1}_j$ associated to $p^i$).
Let  $\tau^{i}$ be the constant section of $\O_{\mathcal{X}}(X^{\leq i})$ (that vanishes to the left of $p^i$), and let $\mu^i$ be the constant section of $\O_{\mathcal{X}}(X^{>i})$ (that vanishes to the right of $p^i$).
\begin{center}
\begin{tikzpicture}
\draw (2.5+.5-.2, 1.5+.3) -- (5-1+.4+.1, 1.5+.3+.3)node[right]{\tiny $\lambda^{i}_{j,2}$};
\draw (2.5+.5-.2, 1.5+.3) -- (5-1+.4+.1, 1.5-.3+.3)node[right]{\tiny $\lambda^{i}_{j,1}$};
\draw (2.5+.5-.2, 1.5+.3) -- (5-1.3+.3+.1, 1.5+1.2+.3)node[above]{\tiny $\lambda^{i}_{j,n}$};
\filldraw (5-.5+.3-.2, 2.4+.3) circle[radius=0.015];
\filldraw (5-.45+.3-.2, 2.25+.3) circle[radius=0.015];
\filldraw (5-.55+.3-.2, 2.55+.3) circle[radius=0.015];

\draw (2.5-.5, 1.5-.3) -- (1-.4-.3, 1.5+.3-.3)node[left]{\tiny $\sigma^i_{j,2}$};
\draw (2.5-.5, 1.5-.3) -- (1-.4-.3, 1.5-.3-.3)node[left]{\tiny $\sigma^i_{j,1}$};
\draw (2.5-.5, 1.5-.3) -- (1.3-.3-.3, 1.5+1.2-.3)node[above]{\tiny $\sigma^i_{j,n}$};
\filldraw (.5-.3, 2.4-.3) circle[radius=0.015];
\filldraw (.45-.3, 2.25-.3) circle[radius=0.015];
\filldraw (.55-.3, 2.55-.3) circle[radius=0.015];
\draw (0, 0) .. controls (1, -1) and (2, -1) .. (3, 0);
\draw (2, 0) .. controls (3, -1) and (4, -1) .. (5, 0);
\filldraw (2.5, -0.42) circle[radius=0.03];
\draw (2.5, 0) node{$p^i$};
\draw (1.5, -0.95) node{$E^i$};
\draw (3.5, -0.95) node{$E^{i+1}$};
\draw[->] (2.5, -1.65) -- (5, -1.65);
\node at (3.75,-2) {$\mu^i = 0$};
\draw(2.5, -1.55) -- (2.5, -1.75);
\draw[->] (2.5, -2.25) -- (0, -2.25);
\node at (1.25, - 2.6) {$\tau^i = 0$};
\draw (2.5, -2.15) -- (2.5, -2.35);
\node at (-.5, -.35) {$\cdots$};
\node at (5.5, -.35) {$\cdots$};
\end{tikzpicture}
\end{center}
Suppose that for layer $j$, the specified new ramification indices at $p^{i}$ on $E^{i+1}$ are $\na{i}{j}{1}, \na{i}{j}{2}, \ldots, \na{i}{j}{m_j}$ (desired ramification of the $\lambda^{i}_{j,n}$'s), and that the new ramification indices at $p^i$ on $E^i$ are
the $\nb{i}{j}{n} = d - d_jk - \na{i}{j}{n}$ (desired ramification of the $\sigma^{i}_{j,n}$'s).

We define a closed subvariety $Y^i \subset F^{i,p^i}_{s-1} \times F^{i+1,p^{i}}_{s-1}$ by the
conditions (for $j = 1, \ldots, s-1$, and $n = 1, \ldots, m_j$):
\begin{equation} \label{condition_j}
\sigma^i_{j,n} \otimes (\tau^i)^{\na{i}{j}{n}} = \lambda^i_{j,n} \otimes (\mu^i)^{\nb{i}{j}{n}},
\end{equation}
viewed as elements of the projectivization of
\[\pi_*\L(-d_j)^i(D + \na{i}{j}{n}X^{\leq i}) \cong \pi_*\L(-d_j)^{i+1}(D + \nb{i}{j}{n}X^{> i}).\]
The isomorphism above comes from the fact that 
\[\L(-d_j)^{i} \simeq \L(-d_j)^{i+1}(-(d-d_jk) X^{\leq i})\]
 and $\O(X^{\leq i}) \cong \O(-X^{> i})$ and $\na{i}{j}{n} + \nb{i}{j}{n} = d-d_jk$.

Away from the central fiber, $\tau^i$ and $\mu^i$ are non-zero, so \eqref{condition_j} says the new sections $\sigma^i_{j,n}$ and $\lambda^i_{j,n}$ are equal (up to scaling).

In the central fiber,
$\tau^i|_{X^{>i}}$ is not a zero divisor and vanishes only at $p^i$, while $\mu^i|_{X^{\leq i}}$ is not a zero a divisor
and vanishes only at $p^i$. Thus
condition~\eqref{condition_j} says that $\sigma^i_{j,n}$ and $\lambda^i_{j,n}$
are determined by
the restrictions $\sigma^i_{j,n}|_{X^{\leq i}}$ and $\lambda^i_{j,n}|_{X^{> i}}$,
which can be anything of the desired vanishing order at $p^i$:
\begin{gather}
\sigma^i_{j,n}|_{X^{\leq i}} \text{ vanishes to order at least $\nb{i}{j}{n}$ at $p^i$} \label{want10}\\
\lambda^i_{j,n}|_{X^{> i}} \text{ vanishes to order at least $\na{i}{j}{n}$ at $p^i$} \label{want20} \\
\lambda^i_{j,n}|_{X^{\leq i}} = \sigma^i_{j,n}|_{X^{ \leq i}} \cdot \frac{\tau^i|_{X^{\leq i}}^{\otimes \na{i}{j}{n}}}{\mu^i|_{X^{\leq i}}^{\otimes \nb{i}{j}{n}}}  \qquad \text{and} \qquad \sigma^i_{j,n}|_{X^{>i}} = \lambda^i_{j,n}|_{X^{>i}} \cdot \frac{\mu^i|_{X^{> i}}^{\otimes \nb{i}{j}{n}}}{\tau^i|_{X^{> i}}^{\otimes \na{i}{j}{n}}}. \label{want30}
\end{gather}
Since $\tau^i$ vanishes on $X^{\leq i}$, and $\mu^i$ vanishes on $X^{> i}$,
in most cases \eqref{want30} simplifies to:
$\lambda^i_{j,n}|_{X^{\leq i}} = 0$ (unless $\na{i}{j}{n} = 0$), and
$\sigma^i_{j,n}|_{X^{>i}} = 0$ (unless $\nb{i}{j}{n} = 0$).
In particular, when $j > 1$, both $\na{i}{j}{n}$ and $\nb{i}{j}{n}$ are positive,
so \eqref{want30} simplifies to   
$\lambda^i_{j,n}|_{X^{\leq i}} = 0$ and
$\sigma^i_{j,n}|_{X^{>i}} = 0$.

For each $(j, n)$, equation \eqref{condition_j} represents
\[\rk \pi_*\L(-d_j)^i(D + \na{i}{j}{n} X^{\leq i}) - 1 = \rk \pi_*\L(-d_j)^i(D) - 1 = N + \chi(\pp^1, \O(\vec{e})(-d_j)) - 1\]
conditions.
Thus, every component of $Y^i$ has codimension at most
\begin{align}
\codim(Y^i \subset F^{i, p^i}_{s-1} \times F^{i + 1, p^{i}}_{s-1}) &\leq \sum_{j=1}^{s-1} m_j(N + \chi(\pp^1, \O(\vec{e})(-d_j)) - 1), \notag \\
\intertext{and so every component of $Y^1 \times Y^2 \times \cdots \times Y^{g-1}$ has codimension at most}
\codim(Y^1 \times Y^2 \times \cdots \times Y^{g-1} \subset F) &\leq (g - 1) \sum_{j=1}^{s-1} m_j(N + \chi(\pp^1, \O(\vec{e})(-d_j)) - 1). \label{compnodes}
\end{align}

Imposing all the conditions of Sections \ref{comp1} -- \ref{comp2} defines a closed subvariety of $F$. Let us additionally remove the locus over the central fiber where any ramification index of the frame is larger than the specified ramification index. We call the resulting quasiprojective variety $ \widetilde{W}^{\mathrm{fr}}$. The image of $ \widetilde{W}^{\mathrm{fr}} \rightarrow G$ is the desired $\widetilde{W}$.

\subsection{Fiber dimensions}
Finally, let us count the dimension of fibers $\widetilde{W}^{\mathrm{fr}} \rightarrow G$. 
At each layer $j$, the master frame space $F$ parameterizes $2g - 2$ lifted projective frames
of dimension $m_j$.

On the general fiber, our equations specify that these frames are equal in $g - 1$ pairs. The fiber
dimension $\widetilde{W}^{\mathrm{fr}} \to G$ is thus equal to (c.f.\ \eqref{FoG}):
\begin{equation} \label{fib_gen} (g - 1) \sum_{j = 1}^{s-1} m_j (h^0(\O(\vec{e})(-d_j)) - 1).\end{equation}

On the special fiber,
let $\{\sigma^i_{j, n}, \lambda^i_{j,n}\}$ denote a point of $\widetilde{W}^{\mathrm{fr}}$.
The other points in the same fiber are then obtained by applying linear transformations $\Sigma^i$ and $\Lambda^i$
to the $\sigma^i_{j, n}$ and $\lambda^i_{j,n}$, of a particular form we will now explain.
Let $x^i$ and $y^i$ be sections of $\O_P(1)^i$ that vanish at $q^{i-1} = f(p^{i-1})$ and $q^{i} = f(p^i)$ respectively;
via pullback, we think of them as sections of $\O_X(1)^i$ on $X$ vanishing at $p^{i-1}$ and $p^i$ respectively.
Then to $\sigma^i_{j, n}$, the linear transformation $\Sigma^i$ may add any section from a previous layer whose
image in layer $j$ has higher vanishing order; explicitly, we may add $\sigma^i_{j', n'}$
times any monomial $(x^i)^\delta \cdot (y^i)^{d_{j'} - d_j - \delta}$ (with $0 \leq \delta \leq d_{j'} - d_j$) such that
\begin{equation} \label{mon-sigma}
\nb{i}{j'}{n'} + (d_{j'} - d_{j} -\delta) k > \nb{i}{j}{n}.
\end{equation}
Similarly, to $\lambda^i_{j, n}$, the linear transformation $\Lambda^i$ may add $\lambda^i_{j', n'}$
times any monomial $(x^i)^\delta \cdot (y^i)^{d_{j'} - d_j - \delta}$ (with $0 \leq \delta \leq d_{j'} - d_j$) such that
\begin{equation} \label{mon-lambda}
\na{i}{j'}{n'} + \delta k > \na{i}{j}{n}.
\end{equation}
The fiber dimension is therefore the total number of monomials satisfying
\eqref{mon-sigma}, plus the number satisfying \eqref{mon-lambda}.
Recall that $\nb{i}{j'}{n'} + \na{i}{j'}{n'} = d - d_{j'}k$ and $\nb{i}{j}{n} + \na{i}{j}{n} = d - d_j k$,
and $\na{i}{j'}{n'} \not\equiv \na{i}{j}{n} \pmod k$ unless $(j', n') = (j, n)$.
Therefore,
every monomial satisfies exactly one of \eqref{mon-sigma} or \eqref{mon-lambda},
except when $(j', n') = (j, n)$.
The fiber dimension is therefore
\begin{equation}\label{fib_spec} (g - 1) \sum_{j = 1}^{s - 1} m_j \left(-1 + \sum_{j' \leq j} m_{j'} (d_{j'} - d_j + 1)\right) = (g - 1) \sum_{j = 1}^{s - 1} m_j (h^0(\O(\vec{e})(-d_j)) - 1) .\end{equation}
Note that this is the same as the fiber dimension on the general fiber.

\subsection{Final dimension estimate}
Recalling the dimension of $F$ and totaling the equations imposed in Sections \ref{comp1} -- \ref{comp2}, we find
that every component $\widetilde{W}'$ of $\widetilde{W}$ has dimension
\begin{align*}
\dim \widetilde{W}' &\geq \dim F - \text{(number of defining equations)} - \text{(fiber dimension)} \notag \\
&= \dim \Pic + (2g - 2) \left(\sum_{j=1}^{\s-1} m_j (N + \chi (\O(\vec{e})(-d_j)) - 1)\right). \hspace{-1in} & \text{($\dim F$, c.f.\ \eqref{dimF})} \notag \\
&\qquad - (g - 2) \sum_{j=1}^{s-1} m_j (N - h^1(\O(\vec{e})(-d_j))) & \hspace{-1in} \text{(diagonal condition, c.f.\ \eqref{diag_cond})} \\
&\qquad - N\sum_{j=1}^{\s-1} m_j & \hspace{-0.08in}\text{(vanishing along $D$, c.f.\ \eqref{van_D})} \notag \\
&\qquad - (g - 1)\sum_{j=1}^{s-1} m_j(N + \chi(\O(\vec{e})(-d_j)) - 1) & \hspace{-1in}\text{(compatibility at nodes, c.f.\ \eqref{compnodes})} \notag \\
&\qquad - (g - 1)\sum_{j=1}^{s-1} m_j(h^0(\O(\vec{e})(-d_j)) - 1) &  \hspace{-1in} \text{(fiber dimension, c.f.\eqref{fib_gen} and \eqref{fib_spec})} \notag \\
&= \dim \Pic - \sum_{j=1}^{\s -1} m_jh^1(\O(\vec{e})(-d_j)) \notag \\
&= \dim \Pic -u(\vec{e}). & \qedhere
\end{align*}
\end{proof}

The regeneration theorem allows us to show that the $\vec{e}$-nested limit linear series built from tableaux as in Lemma~\ref{hasit}
arise as limits from smooth curves, implying that $W^{\Gamma(\vec{e})}(X)$
is contained in the closure of $W^{\vec{e}}(\mathcal{X}^*/B^*)$.

\begin{cor} \label{reg-st}
Let $T$ be an efficiently filled tableau of shape $\Gamma(\vec{e})$. Then $W^T(X)$ is contained in the closure of $W^{\vec{e}}(\mathcal{X}^*/B^*)$.
\end{cor}
\begin{proof}
Let $\widetilde{W}$ be the quasiprojective scheme with ramification corresponding to $T$ as constructed in Theorem~\ref{regen}. By Lemma \ref{hasit}, a generic line bundle in $W^T(X)$ is in the image of $\widetilde{W}$ in $\Pic^d(X)$. Moreover, the uniqueness statement in Lemma \ref{hasit} together with the fact that $\dim W^T(X) = g - u(\vec{e})$
shows that the restriction of $\widetilde{W}$ to the central fiber
has an irreducible component $Y$ of dimension $g - u(\vec{e})$ that dominates $W^T(X)$. By  Theorem~\ref{regen}, the dimension of $Y$ is less than the dimension of any component of $\widetilde{W}$. Let $Y'$ be an irreducible component of $\widetilde{W}$ containing $Y$, which necessarily dominates the base. The closure of the image of $Y'$ in $\Pic^d(\mathcal{X}/B)$ contains $W^T(X)$ and is contained in the closure of $W^{\vec{e}}(\mathcal{X}^*/B^*)$.
\end{proof}

\section{Reducedness and Cohen--Macaulayness}\label{sec:red_CM}

\subsection{Reducedness}
At the central fiber $X = E^1 \cup_{p^1} \cdots \cup_{p^{g-1}} E^g$, the schemes $W^{\vec{e}}(X)$ are defined as intersections of determinantal loci of the form
\[W^r_{\vec{d}}(X) = \{L \in \Pic^{\vec{d}}(X) : h^0(X, L) \geq r+1\},\]
for various degree distributions $\vec{d} = (d^1, \ldots, d^g)$ and integers $r$.

We will first show that any such determinental locus is a union of preimages of reduced points
under various projection maps $\Pic^{\vec{d}}(X) \to \prod_{i \in S} \Pic^{d^i}(E^i)$
(for $S \subseteq \{1, 2, \ldots, g\}$).
Then, we will show that the class of such varieties is closed under intersection,
thus establishing that $W^{\vec{e}}(X)$ is reduced. Combined with Corollary~\ref{reg-st},
this will show that
\[W^{\vec{e}}(X) = W^\Gamma(X).\]

Let $\pi\colon X \times \Pic^{\vec{d}}(X) \rightarrow \Pic^{\vec{d}}(X)$ be the projection and let $\mathcal{L} \colonequals \mathcal{L}_{\vec{d}}$ be a universal line bundle on $X \times \Pic^{\vec{d}}(X)$. Additionally, let $D$ be a divisor on $X$, contained in the smooth locus, and
such that $D^i \colonequals D \cap E^i$ is of sufficiently high degree.
Recall that, in terms of the natural map
\[\phi\colon \pi_*\mathcal{L}(D) \rightarrow \pi_*(\mathcal{L}(D)|_D),\]
the loci of interest are
\[W^r_{\vec{d}}(X) = \{L\in \Pic^{\vec{d}}(X): \dim \ker \phi|_L \geq r + 1\}.\]
The scheme structure is given by the $(n+1) \times (n+1)$ minors of $\phi$, where 
\[n = \rk \pi_*\mathcal{L}(D) - (r+1) = d - g - r + \deg(D).\]

We will describe the rank loci of $\phi$ in terms of evaluation maps on the normalization of $X$. Let $\nu\colon E^1 \sqcup \cdots \sqcup E^g \rightarrow X$ be the normalization and let $\pi^i\colon E^i \times \Pic^{d^i}(E^i) \rightarrow \Pic^{d^i} (E^i)$ be the projection. 
 In addition, let $\pr^i\colon \Pic^{\vec{d}}(X) \rightarrow \Pic^{d^i}(E^i)$ be the projection and let $\mathcal{L}_i$ be a universal line bundle on $E^i \times \Pic^{d^i}(E^i)$ such that $\L|_{E^i \times \Pic^{\vec{d}}(X)} = (\Id \times \pr^i)^* \L_i$.
 These maps fit into the diagram:
 \begin{center}
 \begin{tikzcd}
 & \bigsqcup_{i=1}^g E^i \times \Pic^{\vec{d}}(X) \arrow{dr}{\nu \times \Id} \\
E^i \times \Pic^{\vec{d}}(X) \arrow{ur} \arrow{rr}{\subset} \arrow{d}[swap]{\Id \times \pr^i} &&X \times \Pic^{\vec{d}}(X) \arrow{d}{\pi}\\
E^i \times \Pic^{d^i}(E^i) \arrow{dr}[swap]{\pi^i} && \Pic^{\vec{d}}(X) \arrow{dl}{\pr^i} \\
 &\Pic^{d^i}(E^i)
 \end{tikzcd}
 \end{center}
There is a commuting diagram of vector bundles on $\Pic^{\vec{d}}(X)$:
\begin{center}
\begin{tikzcd}
&& \bigoplus_{i=1}^g (\pr^i)^*(\pi^i)_{*} \L_i(D^i) \arrow[equal]{d} \\
0 \arrow{r} &\pi_*\mathcal{L}(D) \arrow{r} \arrow{d}[swap]{\phi} & \pi_*(\nu \times \Id)_*(\nu \times \Id)^*\mathcal{L}(D) \arrow{r}{\eta} \arrow{d}{\psi} & \bigoplus_{i=1}^{g-1} \pi_* (\mathcal{L}(D)|_{p^i}) \arrow{r} &0\\ 
&\pi_*(\mathcal{L}(D)|_D) \arrow{r}{\sim} & \pi_*(\nu \times \Id)_*( \nu \times \Id)^*\mathcal{L}(D)|_{\nu^{-1}(D)} \arrow[equal]{d} \\
&& \bigoplus_{i=1}^g (\pr^i)^*(\pi^{i})_* \L_i(D^i)|_{D^i}
\end{tikzcd}
\end{center}
The top row is exact by our assumption that $D$ is sufficiently positive.
Since sequences of vector bundles split locally, the rank loci of $\phi$ are corresponding rank loci of $\psi \oplus \eta$:
\begin{equation} \label{psi-eta}
\{L\in \Pic^{\vec{d}}(X): \rk \phi \leq n\} = \{L\in \Pic^{\vec{d}}(X): \rk (\psi \oplus \eta) \leq n + (g - 1)\}.
\end{equation}
The restriction of $\psi \oplus \eta$ to each summand $(\pr^i)^*(\pi^{i})_* \L_i(D^i)$ is $(\pr^i)^*\ev^i$, where
\[\mathrm{ev}^i\colon (\pi^{i})_* \L_i(D^i) \rightarrow (\pi^{i})_*(\L_i(D^i)|_{p^{i-1} \cup p^i \cup D^i})\]
is the map that evaluates a section on $E^i$ at the points of $D^i$ and the nodes $p^{i-1}$ and $p^i$ (or just $p^1$ on $E^1$ or just $p^{g-1}$ on $E^g$). The matrix for $\psi \oplus \eta$ is almost block diagonal.

\begin{center}
\begin{tikzpicture}[scale=0.7]
\draw (0, 0) -- (0, -6-.5);
\draw (0, 0) -- (6+.5, 0);
\draw (2,0) -- (2, -4);
\draw[color=violet]  (0, -2+.25) rectangle  (4, -2);
\draw (4,-2) -- (4, -6);
\draw[color = violet] (2, -4+.25) rectangle (6, -4);
\draw (6, -4) -- (6, -6);
\draw[color = violet] (4, -6) -- (6.5, -6);
\draw[color = violet] (4, -6) -- (4, -6+.25);
\draw[color = violet] (4, -6+.25) -- (6.5, -6+.25);
\node at (7, -6 + .25/2 -.05) {\color{violet} $\cdots$};
\node at (7, -8 - .25/2 -.05) {\color{violet} $\cdots$};
\node at (7 ,-7) {$\ddots$};
\node at (7, -.05) {$\cdots$};
\node at (0, -7) {$\vdots$};
\draw[color = violet] (8-.5, -8) -- (10, -8);
\draw[color = violet] (8-.5, -8-.25) -- (10, -8-.25);
\draw (8, -8) -- (8, -10);
\draw (0, -8+.5) -- (0, -10);
\draw (0, -10) -- (10, -10);
\draw (8-.5, 0) -- (10, 0);
\draw (10, 0) -- (10, -6-.5);
\node at (10, -7) {$\vdots$};
\draw (10, -8+.5) -- (10, -10);
\draw (0-1, 0+.25) .. controls (-.7-1,-2/3+.25/3) and (-.7-1,-4/3+.5/3) .. (0-1,-2);
\draw (0-1, -2+.25) .. controls (-.7-1,-2-2/3+.25/3) and (-.7-1,-2-4/3+.5/3) .. (0-1,-4);
\draw (0-1, -4+.25) .. controls (-.7-1,-4-2/3+.25/3) and (-.7-1,-4-4/3+.5/3) .. (0-1,-6);
\draw (0-1, -8) .. controls (-.7-1,-8-2/3-.25/3) and (-.7-1,-8-4/3-.5/3) .. (0-1,-10-.25);
\node at (-1.25, -7) {$\vdots$};
\node at (-1.1,-2+.125) {\tiny \color{violet} $\bullet$};
\node at (-1.1-.4,-2+.125) {\small \color{violet} $p^1$};
\node at (-1.1,-4+.125) {\tiny \color{violet} $\bullet$};
\node at (-1.1-.4,-4+.125) {\small \color{violet} $p^2$};
\node at (-1.1,-6+.125) {\tiny \color{violet} $\bullet$};
\node at (-1.1-.4,-6+.125) {\small \color{violet} $p^3$};
\node at (-1.1,-8-.125) {\tiny \color{violet} $\bullet$};
\node at (-1.1-.5,-8-.125) {\small \color{violet} $p^{g-1}$};
\node at (1, -1) {\small $\mathrm{ev}^1$};
\node at (3, -3) {\small $\mathrm{ev}^2$};
\node at (5, -5) {\small $\mathrm{ev}^3$};
\node at (9, -9.125) {\small $\mathrm{ev}^g$};
\node at (8, -2) {\huge $0$};
\node at (2, -8) {\huge $0$};
\end{tikzpicture}
\end{center}

The following lemma helps us describe the rank loci of such almost block diagonal matrices scheme-theoretically.

\begin{lem} \label{AB}
Let $\mathrm{Mat}(x, y)$ denote the affine space of $x \times y$ matrices.
Suppose that $S$ and $T$ are given schemes, and
$A\colon S \rightarrow \mathrm{Mat}(x_1, y_1)$ and $B \colon T \rightarrow \mathrm{Mat}(x_2, y_2)$ are two families of matrices. 
Let $M \colon S \times T \rightarrow \mathrm{Mat}(x_1 + x_2, y_1 + y_2 - 1)$ be the family of matrices where the upper left entries are given by $A$, the lower right entries are given by $B$, and the rest of the entries are $0$.
\begin{center}
\begin{tikzpicture}
\draw (0, 0) rectangle (4, -4.25);
\node at (1, -1) {$A_-$};
\node at (3, -3.25) {$B^-$};
\draw[violet, thick] (0, -2) rectangle (4, -2.25);
\draw (2, 0) -- (2, -4.25);
\draw [decorate,decoration={brace,amplitude=4pt},xshift=0.25cm,yshift=0pt]
      (4, -2) -- (4,-4.25) node [midway,right,xshift=.1cm] {$B$};
 \draw [decorate,decoration={brace,amplitude=4pt,mirror},xshift=-0.25cm,yshift=0pt]
(0,0) -- (0,-2.25) node [midway,left,xshift=-0.1cm]  {$A$};
\node at (3, -1) {$0$};
\node at (1, -3.25) {$0$};
\end{tikzpicture}
\end{center}
Let $A_-\colon S \rightarrow \mathrm{Mat}(x_1-1, y_1)$ be the composition of $A$ with the map that removes the bottom row. Similarly, let
$B^-\colon T \rightarrow \mathrm{Mat}(x_2, y_2-1)$ be the composition of $B$ with the map that removes the top row.
Let
\begin{align*}
S_a &= \{s \in S: \rk A(s) \leq a\} & T_b &= \{t \in T: \rk B(t) \leq b\} \\ 
S_a^- &= \{s \in S : \rk A_-(s)\leq a\} & T_b^- &= \{t \in T : \rk B^-(t)\leq b\}, 
\end{align*}
defined scheme-theoretically by the vanishing of appropriately sized minors.
Then
\begin{equation} \label{schemey}
\{(s,t) \in S \times T :\rk M(s,t) \leq n\} = \bigcup_{a + b= n} (S_a \times T_b)
\cup \bigcup_{a + b=n-1} (S_a^- \times T_b^-)
 \end{equation}
as schemes.
In particular, if all rank loci of $A, A_-, B,$ and $B^-$ are reduced, then all rank loci of $M$ are reduced.
\end{lem}
\begin{proof}
The scheme structure on the left hand side of \eqref{schemey} is defined by the $(n + 1) \times (n+1)$ minors of~$M$. 
Any such minor contains $a$ columns meeting $A$ and $b$ columns meeting $B$ for some $a + b = n+1$. 

First consider a square submatrix of $M$ that uses the common row (colored violet in the diagram).
If its determinant is nonzero, then the submatrix contains at least $a-1$ rows meeting $A_-$ and at least $b-1$ rows meeting $B^-$. 
Either there are $a$ rows meeting $A_-$, and its determinant is an $a \times a$ minor of $A_-$ times a $b \times b$ minor of $B$; or there are $b$ rows meeting $B^-$, and its determinant is an $a \times a$ minor of $A$ times a $b \times b$ minor of $B^-$. 
Now consider a square submatrix of $M$ that does not use the common row. If its determinant is nonzero, then it is a product of an $a \times a$ minor of $A_-$ and a $b \times b$ minor of $B^-$.
Ranging over all minors with this distribution of columns we obtain elements that generate
\[ I(S_{a-1}^-) \cdot I(T_{b-1}) + I(S_{a-1}) \cdot I(T_{b-1}^-) = I((S_{a-1}^- \times T \cup S \times T_{b-1}) \cap (S_{a-1} \times T \cup S \times T_{b-1}^-)),\]
where $S_{-1} = S_{-1}^- = T_{-1} = T_{-1}^- = \varnothing$ by convention.

We therefore find that
\[\{(s,t) \in S \times T :\rk M(s,t) \leq n\} =  \bigcap_{a + b = n+1} (S_{a-1}^- \times T \cup S \times T_{b-1}) \cap (S_{a-1} \times T \cup S \times T_{b-1}^-).\]
Because $A_-$ is obtained from $A$ by removing a single row, the rank loci of $A$ and $A_-$ are nested
\[\varnothing = S_{-1}^- \subseteq S_0 \subseteq S_0^- \subseteq S_1 \subseteq S_1^- \subseteq S_2 \subseteq S_2^- \subseteq \cdots \subseteq S_{n-1} \subseteq S_{n-1}^- \subseteq S_{n} \subseteq S,\]
and similarly for $B$ and $B^-$
\[\varnothing = T_{-1}^- \subseteq T_0 \subseteq T_0^- \subseteq T_1 \subseteq T_1^- \subseteq T_2 \subseteq T_2^- \subseteq \cdots \subseteq T_{n-1} \subseteq T_{n-1}^- \subseteq T_{n} \subseteq T. \]
The claim now follows from the following general fact regarding such intersections.
\end{proof}
\begin{lem}
Given nested sequences of schemes
\[\emptyset = Y_0 \subseteq Y_1 \subseteq \cdots \subseteq Y_m \subseteq Y \qquad \text{and} \qquad \emptyset = Z_0 \subseteq Z_1 \subseteq \cdots \subseteq Z_m \subseteq Z,\]
we have
\[\bigcap_{i+j = m} (Y_i \times Z) \cup (Y \times Z_j) = \bigcup_{i+j=m + 1} Y_i \times Z_j \]
as schemes.
\end{lem}
\begin{proof}
In general, intersection does not distribute across unions scheme-theoretically. However, we first show that
if $A_1 \subset A_2$ and $B$ are subschemes of any scheme, then 
\[A_2 \cap (A_1 \cup B) = A_1 \cup (A_2 \cap B).\]
as schemes. That is, intersection distributes across union if appropriate containments are satisfied. It suffices to prove the statement in the affine case, where this becomes a statement about ideals. Suppose $I_1 \colonequals I(A_1) \supset I_2 \colonequals I(A_2)$ and $J \colonequals I(B)$. Then we must show
\[ I_2 + (I_1 \cap J) = I_1 \cap  (I_2 +  J). \]
If $a \in I_2$ and $b \in I_1 \cap J$, then it's clear that $a + b \in I_1$ and $a + b \in I_2 + J$.
Now suppose we have $a \in I_2$ and $b \in J$ so that $a+b \in I_1$. Since $I_2 \subset I_1$, it follows that $b \in I_1$. Hence $a + b \in I_2 + I_1 \cap J$.

To prove the lemma, we induct on $m$. The case $m = 0$ is immediate ($\emptyset = \emptyset$). Suppose we know the result for chains of length one less.
We want to study
\begin{align*}
\bigcap_{i+j = m} (Y_i \times Z) \cup (Y \times Z_j) &= (Y \times Z_m) \cap [(Y_1 \times Z) \cup (Y \times Z_{m-1})] \cap \bigcap_{\substack{i+j = m \\ i \geq 2}} (Y_i \times Z) \cup (Y \times Z_j) \\
&= [(Y \times Z_{m-1}) \cup (Y_1 \times Z_m)] \cap \bigcap_{\substack{i+j = m \\ i \geq 2}} (Y_i \times Z) \cup (Y \times Z_j) \\
&= (Y_1 \times Z_m) \cup \left[(Y \times Z_{m-1}) \cap \bigcap_{\substack{i+j = m \\ i \geq 2}} (Y_i \times Z) \cup (Y \times Z_j) \right].
\end{align*}
Now the term in large square brackets is the intersection of complementary unions of the chains
\[\emptyset = Y_0 \subseteq Y_2 \subseteq Y_3 \subseteq \cdots \subseteq Y_m \subseteq Y \qquad \text{and} \qquad \emptyset = Z_0 \subseteq Z_1 \subseteq \cdots \subseteq Z_{m-1} \subseteq Z,\]
which have length one less. The result now follows by induction.
\end{proof}

\noindent
We now study the rank loci of evaluation maps on elliptic curves.

\begin{lem} \label{one}
Let $E$ be an elliptic curve and $\mathcal{L}$ a universal line bundle on $\pi\colon E \times \Pic^d (E) \rightarrow \Pic^d(E)$
with $d \geq 1$. Suppose $D$ is an effective divisor and 
\[\ev\colon \pi_*\mathcal{L} \rightarrow \pi_*(\mathcal{L}|_{D \times \Pic^d(E)}) \]
 is the evaluation map. For any $n$, the scheme
 \[S_n = \{L \in \Pic^d(E) : \rk \ev|_L \leq n\}\]
 is either empty, all of $\Pic^d(E)$, or a single reduced point. In particular, it is reduced.
\end{lem}
\begin{proof}
We have that $S_n$ is empty or all of $\Pic^d(E)$ unless $\deg D = d$ and $n = d - 1$. When $\deg D = d$, the evaluation map is between vector bundles that both have rank $d$. The locus where the map drops rank is cut out by the determinant, which is a section of $\det (\pi_*\mathcal{L})^\vee \otimes \det (\pi_*(\mathcal{L}|_{D\times \Pic^d(E)}))$. Set theoretically, the vanishing of this section is supported on $L = \O(D) \in \Pic^d(E)$.
To see it is reduced, we compute its degree:
\begin{align*}
\deg \left(\det (\pi_*\mathcal{L})^\vee \otimes \det (\pi_*(\mathcal{L}|_{D\times \Pic^d(E)}))\right) &= - \deg c_1(\pi_*\mathcal{L}) + \deg c_1(\pi_*(\mathcal{L}|_{D\times \Pic^d(E)})) \\
\intertext{Using Grothendieck--Riemann--Roch (noting that the relative Todd class is trivial since $E$ is an elliptic curve):}
&=-\deg \mathrm{ch}_2(\mathcal{L}) + \deg \ch_2(\mathcal{L}|_{D\times \Pic^d(E)}) \\
\intertext{Using the additivity of Chern characters in exact sequences:}
&= - \deg \ch_2(\mathcal{L}(-D \times \Pic^d(E))) \\  
&= - \frac{1}{2} \deg c_1(\mathcal{L}(-D \times \Pic^d(E)))^2 \\
\intertext{Given an identification $E \times \Pic^d(E) \cong E \times E$, the line bundle $\mathcal{L}(-D \times \Pic^d(E))$ can be represented by the diagonal  $\Delta$ minus a fiber $f$. Since $\Delta^2 = 0$ (by adjunction) and $f^2 = 0$, we obtain $(\Delta - f)^2 = -2$. Thus:}
&= 1. \qedhere
\end{align*}
\end{proof}

\begin{lem} \label{uis}
Let $\mathcal{S}$ be the collection of subvarieties of $\Pic^{\vec{d}}(X)$ that are unions of reduced preimages of points via projections $(\prod_{i \in S} \pr^i) \colon \Pic^{\vec{d}}(X) \rightarrow \prod_{i \in S} \Pic^{d^i}(E^i)$ for some $S \subseteq \{1, \ldots, g\}$. Then $\mathcal{S}$ is closed under union and intersection.
\end{lem}
\begin{proof}
It is clear that $\mathcal{S}$ is closed under union. The intersection statement is clear set-theoretically, so it suffices to show that the intersection of two elements of $\mathcal{S}$ is reduced.

Suppose $A, B \in \mathcal{S}$ and $p \in A \cap B$. 
Choose \'etale coordinates $x_i$ on $\Pic^{d^i}(E^i)$ near $\pr^i(p)$.
Then \'etale-locally, $A$ and $B$ are reduced unions of coordinate linear spaces.
Equivalently, $I(A)$ and $I(B)$ are reduced monomial ideals
(i.e.\ generated by monomials in the $x_i$ whose exponents are $0$ or $1$).
The class of such ideals is closed under addition.
\end{proof}

\noindent
Putting together Lemmas \ref{AB}, \ref{one}, and \ref{uis}, we deduce the desired reducedness property.
\begin{thm} \label{thm:reduced}
For any $r \geq 0$ and degree distribution $\vec{d}$, the scheme $W^r_{\vec{d}}(X)$
is a union of preimages of reduced points via projections $\Pic^{\vec{d}}(X) \rightarrow \prod_{i \in S} \Pic^{d_i}(E_i)$ for subsets $S \subseteq \{1, \ldots, g\}$. It follows that $W^{\vec{e}}(X)$ is reduced.
\end{thm}
\begin{proof}
By \eqref{psi-eta},
\[W^r_{\vec{d}}(X) = \{L \in \Pic^{\vec{d}}(X) : \rk (\eta \oplus \psi)|_L \leq d + \deg(D) - (r+1)\}.\]
Applying Lemma \ref{one} with $D = D^i \cup p^i \cup p^{i-1}$ (respectively $D = D^i \cup p^i$ or $D^i \cup p^{i-1}$ or $D^i$), we see that the rank loci of the maps $\ev^i$ (respectively $\ev^i$ with top row or bottom row or top and bottom row removed) are all in $\mathcal{S}$.
Repeated application of Lemma \ref{AB} then shows that the rank loci of $\eta \oplus \psi$ are all in $\mathcal{S}$. Thus, $W^{\vec{e}}(X)$ is an intersection of subschemes in $\mathcal{S}$, so $W^{\vec{e}}(X)$ is in $\mathcal{S}$, and in particular is reduced.
\end{proof}

\begin{cor} \label{reg-scheme} We have $W^{\vec{e}}(X) = W^{\Gamma(\vec{e})}(X)$ is reduced, and equal scheme-theoretically
to the closure of $W^{\vec{e}}(\mathcal{X}^*/B^*)$
in the central fiber.
\end{cor}
\begin{proof}
We have shown the following containments. (In order: by construction c.f.\ Definition~\ref{def-epos} and following discussion, by Theorem~\ref{thm:reduced}, by Proposition~\ref{li-value}, and by Corollary~\ref{reg-st} respectively.)
\[\overline{W^{\vec{e}}(\mathcal{X}^*/B^*)}|_0 \subseteq W^{\vec{e}}(X) = W^{\vec{e}}(X)_{\text{red}} \subseteq W^{\Gamma(\vec{e})}(X) \subseteq \overline{W^{\vec{e}}(\mathcal{X}^*/B^*)}|_0.\]
Therefore all containments are equalities.
\end{proof}

\subsection{Cohen--Macaulayness}

For any $k$-convex diagram $\Gamma$, we will prove that $W^\Gamma(X)$ is Cohen--Macaulay by inducting on $g$.
 The key to running our induction is the following standard fact:
 \begin{lem}[See, for example, Proposition~4.1 of \cite{cmv}] \label{std}
Suppose $A$ and $B$ are Cohen--Macaulay and $A \cap B $ is codimension $1$ in both $A$ and $B$. Then $A \cup B$ is Cohen--Macaulay if and only if $A \cap B$ is Cohen--Macaulay.
 \end{lem}

\begin{thm} \label{CM}
Given any $k$-convex shape $\Gamma$, the scheme $W^\Gamma(X)$ is Cohen--Macaulay.
In particular, $W^{\vec{e}}(X)$ is Cohen--Macaulay, and therefore $W^{\vec{e}}(C)$ is Cohen--Macaulay for a general degree~$k$
genus~$g$ cover
$f \colon C \to \pp^1$.
\end{thm}
\begin{proof}
We will induct on $g$. For the base case $g = 1$, we know $W^\Gamma(X)$ is either all of $\Pic^d(E^1)$, or a single reduced point. 
For the inductive step, suppose we are given $L \in W^\Gamma(X)$.
Let
\[\pi\colon \Pic^d(X) \cong \prod_{i=1}^g \Pic^d(E^i) \rightarrow \Pic^d(X^{\leq g - 1}) \cong \prod_{i < g} \Pic^d(E^i)\]
be the projection map, i.e.\ $(M^1, \ldots, M^g) \mapsto (M^1, \ldots, M^{g-1})$. Let $\iota_L\colon \Pic^d(X') \rightarrow \Pic^d(X)$ be the
section which sends $(M^1, \ldots, M^{g-1}) \mapsto (M^1, \ldots, M^{g-1}, L^g)$.
Define
\[A \colonequals \bigcup_{T : g \in T} W^T(X) \qquad \text{and} \qquad B \colonequals \bigcup_{T : g \notin T} W^T(X). \]

If $L \in A$, then $L^g$ is a linear combination of $p^{g-1}$ and $p^g$, and
$\Gamma$ has a removable corner in the corresponding residue $j$.
Recall that $u(s_j \cdot \Gamma) = u(\Gamma) - 1$
(c.f.\ Section~\ref{tab-word} item~\eqref{eff}).

Then $A = \iota_L(W^{s_j \cdot \Gamma}(X^{\leq g - 1}))$ in a neighborhood of $L$, so $A$ is Cohen--Macaulay by induction.
Meanwhile, we have $B = \pi^{-1}(W^{\Gamma}(X^{\leq g - 1}))$, so $B$ is Cohen--Macaulay by induction.
Moreover, $A \cap B = \iota_L(W^{\Gamma}(X^{\leq g - 1}))$ in a neighborhood of $L$,
so is Cohen--Macaulay by induction.
Since $A \cap B$ is codimension $1$ in both $A$ and $B$,
Lemma~\ref{std} implies that $W^\Gamma(X) = A \cup B$ is Cohen--Macaulay.

Since Cohen--Macaulayness is an open condition, $W^{\vec{e}}(C)$ is Cohen--Macaulay for
a general degree~$k$ cover $f \colon C \to \pp^1$.
\end{proof}

\begin{rem}
As explained in the introduction, this also establishes Cohen-Macaulayness, and hence reducedness of universal splitting loci (Corollary \ref{cor:ESconj}). The authors suspect that a more direct argument may be given by just degenerating the $\pp^1$ (without considering another curve). Reducedness confirms that the scheme structure on splitting loci obtained from Fitting supports --- i.e. non-transverse intersections of determinantal loci --- is the ``correct" scheme structure.
Cohen-Macaulyness supports the perspective that, despite this failure of transversality, splitting loci ought to behave like determinantal loci (see e.g.\ \cite{P1} where analogues of the Porteous formula were given).
\end{rem}

\section{Connectedness}\label{sec:conn}

In this section, we show $W^{\vec{e}}(C)$ is connected when $g > u(\vec{e})$, where $f \colon C \to \pp^1$
is a general degree $k$ cover.
Since ``geometrically connected and geometrically reduced'' is an open condition
in flat proper families \cite[Theorem 12.2.4(vi)]{ega4},
and we have already shown that $W^{\vec{e}}(X)$ is reduced (c.f.\ Theorem~\ref{thm:reduced}),
it suffices to see that $W^{\vec{e}}(X) = W^{\Gamma(\vec{e})}(X)$ is
connected. In other words, we want to show the transitivity of the following equivalence relation:

\begin{defin}
We say two $k$-core tableau $T$ and $T'$ with the same shape $\Gamma$
\defi{meet} if $W^T(X)$ and $W^{T'}(X)$ intersect.
We say that $T$ and $T'$ are \defi{connected} if $W^T(X)$ and $W^{T'}(X)$ are in the same connected
component of $W^\Gamma(X)$.
\end{defin}

By definition, ``connected'' is the equivalence relation generated by ``meet.''
Unwinding Definition~\ref{WT}, $T$ and $T'$ meet if and only if for every $t$,
either:
\[T[t] = *, \quad T'[t] = *, \quad \text{or} \quad T[t] \equiv T'[t] \pmod{k}.\]
The simplest example of two tableaux that are connected is the following.

\begin{example}\label{swap_unused}
Suppose that $T$ is a tableau filled with a subset of the symbols $\{1, \dots, g\}$ that does not include the symbol $N$.  Let $N'$ be the smallest symbol greater than (respectively largest symbol less than) $N$ appearing in $T$.  Then $T$ meets the tableau obtained from $T$ by replacing the symbol $N'$ with $N$. 
Applying this repeatedly, 
every tabuleau is connected to the tableau obtained by relabeling symbols via an order preserving map from the subset of symbols used to any other subset of $u$ symbols. 
\end{example}

By Example \ref{swap_unused}, it suffices to show that all efficiently-filled tableau filled with symbols $\{1, \ldots, u\}$ are connected when $g > u$.
To do this, we use the relations in the affine symmetric group
\[s_j s_{j'} = s_{j'} s_j \ \text{(for $j - j'\neq \pm 1$)} \quad \text{and} \quad s_j s_{j + 1} s_j = s_{j + 1} s_j s_{j + 1}.\]
These relations give rise to two basic moves, known as the \defi{braid moves} for the affine symmetric group,
between reduced words:
\begin{align*}
s_{j_u} \cdots s_{j_1} &\leftrightarrow s_{j_u} \cdots s_{j_{i + 2}} s_{j_i} s_{j_{i + 1}} s_{j_{i - 1}} \cdots s_{j_1} \quad \text{(for $j_i - j_{i+1} \neq \pm 1$)} && F^i \\
s_{j_u} \cdots s_{j_1} &\leftrightarrow s_{j_u} \cdots s_{j_{i + 3}} s_{j_{i + 1}} s_{j_i} s_{j_{i + 1}} s_{j_{i - 1}} \cdots s_{j_1} \quad \text{(for $j_i = j_{i+1} \pm 1 = j_{i + 2}$)} && S^i \\
\end{align*}

Given our identification of reduced words with efficient $k$-core tableaux,
this is equivalent to the following moves on tableaux:

\begin{defin} Let $T$ be an efficiently filled $k$-core tableau. We define the two \defi{braid moves}
as follows:
\begin{description}
\item[F]
If $T[i] - T[i + 1] \not\equiv \pm 1$ mod $k$, define the \defi{flip} $F^i T$ by
\[(F^i T)[t] = \begin{cases}
T[i + 1] & \text{if $t = i$;} \\
T[i] & \text{if $t = i + 1$;} \\
T[t] & \text{otherwise.}
\end{cases}
\]
\item[S]
Similarly, if $T[i] \equiv T[i + 2]$ and $T[i + 1] \equiv T[i] \pm 1$ mod $k$,
define the \defi{shuffle} $S^i T$ by
\[(S^i T)[t] = \begin{cases}
T[i] & \text{if $t = i + 1$;} \\
T[i + 1] & \text{if $t \in \{i, i + 2\}$;} \\
T[t] & \text{otherwise.}
\end{cases}
\]
\end{description}
\end{defin}

Note that both braid moves are involutions, i.e.\ $F^i F^i T = T$ and $S^i S^i T = T$.
It is known that we can get from any reduced word to any other reduced word --- or equivalently from any efficient
tableau to any other efficient tableau --- by applying a sequence of braid moves
(c.f.\ Theorem~3.3.1(ii) of \cite{ccg}).
Therefore, it suffices to prove that $T$ and $F^i T$ (when defined),
respectively $T$ and $S^i T$ (when defined), are connected.

\begin{lem}
Suppose that $T$ is an efficient filling of a $k$-core $\Gamma$
with symbols $\{1, \ldots, u\}$, and that $F^i T$ is defined.
If $g > u(\Gamma)$, then $T$ and $F^i T$ are connected.
\end{lem}
\begin{proof}
Start with $T$ and relabel symbols according to $\{1, \ldots, u\} \mapsto \{1, \ldots, i-1, i+1, \ldots, u+1\}$ (see Example \ref{swap_unused}). 
In this filling, all $(i+2)$'s may be replaced with $i$'s because 
no $i+1$ appears to the upper-left of an $i+2$.  After this replacement, we further
relabel symbols according to $\{1, \ldots, i+1, i+3, \ldots, u+1\} \mapsto \{1, \ldots, u\}$.
\begin{center}
\begin{tikzpicture}[scale = .4]
\node at (4, 1) {Start with $T$ \ldots};
\draw (0, 0) -- (0, -8);
\draw (1, 0) -- (1, -4);
\draw (2, 0) -- (2, -2);
\draw (3, 0) -- (3, -1);
\draw (4, 0) -- (4, -1); 
\draw (0, 0) -- (8, 0);
\draw (0, -4) -- (1, -4);
\draw (0, -3) -- (2, -3);
\draw (0, -2) -- (2, -2);
\draw (0, -1) -- (3, -1);
\draw [color = blue, fill = blue!20] (0, -6) rectangle (1, -5);
\node [color = blue, scale = 0.7] at (0.5, -5.5) {\tiny $i+1$};
\draw [color = blue, fill = blue!20] (0+2, -6+3) rectangle (1+2, -5+3);
\node [color = blue, scale = 0.7] at (0.5+2, -5.5+3) {\tiny $i+1$};
\draw [color = blue, fill = blue!20] (0+2+3, -6+3+2) rectangle (1+2+3, -5+3+2);
\node [color = blue, scale = 0.5] at (0.5+2+3, -5.5+3+2) {$i+1$};

\draw [color = red, fill = red!20] (0+1, -6+1) rectangle (1+1, -5+1);
\node [color = red] at (0.5+1, -5.5+1) {\tiny $i$};
\draw [color = red, fill = red!20] (0+1+2, -6+1+3) rectangle (1+1+2, -5+1+3);
\node [color = red] at (0.5+1+2, -5.5+1+3) {\tiny $i$};
\draw[thick] (0, -5) -- (1,-5) -- (1, -4) -- (2, -4) -- (2, -2) -- (3, -2) -- (3, -1) -- (5, -1) -- (5, 0);
\node[rotate = 45, scale = 0.7] at (5,-4) {symbols $> i+1$};
\end{tikzpicture}
\hspace{.2in}
\begin{tikzpicture}[scale = .4]
\node at (4.4, 1) {relabel symbols \ldots};
\draw (0, 0) -- (0, -8);
\draw (1, 0) -- (1, -4);
\draw (2, 0) -- (2, -2);
\draw (3, 0) -- (3, -1);
\draw (4, 0) -- (4, -1); 
\draw (0, 0) -- (8, 0);
\draw (0, -4) -- (1, -4);
\draw (0, -3) -- (2, -3);
\draw (0, -2) -- (2, -2);
\draw (0, -1) -- (3, -1);
\draw [color = violet, fill = violet!20] (0, -6) rectangle (1, -5);
\node [color = violet, scale = 0.7] at (0.5, -5.5) {\tiny $i+2$};
\draw [color = violet, fill = violet!20] (0+2, -6+3) rectangle (1+2, -5+3);
\node [color = violet, scale = 0.7] at (0.5+2, -5.5+3) {\tiny $i+2$};
\draw [color = violet, fill = violet!20] (0+2+3, -6+3+2) rectangle (1+2+3, -5+3+2);
\node [color = violet, scale = 0.5] at (0.5+2+3, -5.5+3+2) {$i+2$};

\draw [color = blue, fill = blue!20] (0+1, -6+1) rectangle (1+1, -5+1);
\node [scale = .7, color = blue] at (0.5+1, -5.5+1) {\tiny $i+1$};
\draw [color = blue, fill = blue!20] (0+1+2, -6+1+3) rectangle (1+1+2, -5+1+3);
\node [scale = .7, color = blue] at (0.5+1+2, -5.5+1+3) {\tiny $i+1$};
\draw[thick] (0, -5) -- (1,-5) -- (1, -4) -- (2, -4) -- (2, -2) -- (3, -2) -- (3, -1) -- (5, -1) -- (5, 0);
\node[rotate = 45, scale = 0.7] at (5,-4) {symbols $> i+2$};
\end{tikzpicture}
\hspace{.2in}
\begin{tikzpicture}[scale = .4]
\node at (4, 1) {replace $i+2$ with $i$};
\draw (0, 0) -- (0, -8);
\draw (1, 0) -- (1, -4);
\draw (2, 0) -- (2, -2);
\draw (3, 0) -- (3, -1);
\draw (4, 0) -- (4, -1); 
\draw (0, 0) -- (8, 0);
\draw (0, -4) -- (1, -4);
\draw (0, -3) -- (2, -3);
\draw (0, -2) -- (2, -2);
\draw (0, -1) -- (3, -1);
\draw [color = red, fill = red!20] (0, -6) rectangle (1, -5);
\node [color = red] at (0.5, -5.5) {\tiny $i$};
\draw [color = red, fill = red!20] (0+2, -6+3) rectangle (1+2, -5+3);
\node [color = red] at (0.5+2, -5.5+3) {\tiny $i$};
\draw [color = red, fill = red!20] (0+2+3, -6+3+2) rectangle (1+2+3, -5+3+2);
\node [color = red] at (0.5+2+3, -5.5+3+2) {\tiny $i$};

\draw [color = blue, fill = blue!20] (0+1, -6+1) rectangle (1+1, -5+1);
\node [scale = .7, color = blue] at (0.5+1, -5.5+1) {\tiny $i+1$};
\draw [color = blue, fill = blue!20] (0+1+2, -6+1+3) rectangle (1+1+2, -5+1+3);
\node [scale = .7, color = blue] at (0.5+1+2, -5.5+1+3) {\tiny $i+1$};
\node[rotate = 45, scale = 0.7] at (5,-4) {symbols $> i+2$};
\draw[thick] (0, -5) -- (1,-5) -- (1, -4) -- (2, -4) -- (2, -2) -- (3, -2) -- (3, -1) -- (5, -1) -- (5, 0);
\end{tikzpicture}
\hspace{.2in}
\begin{tikzpicture}[scale = .4]
\node at (4, 1) {\ldots relabel symbols};
\draw (0, 0) -- (0, -8);
\draw (1, 0) -- (1, -4);
\draw (2, 0) -- (2, -2);
\draw (3, 0) -- (3, -1);
\draw (4, 0) -- (4, -1); 
\draw (0, 0) -- (8, 0);
\draw (0, -4) -- (1, -4);
\draw (0, -3) -- (2, -3);
\draw (0, -2) -- (2, -2);
\draw (0, -1) -- (3, -1);
\draw [color = red, fill = red!20] (0, -6) rectangle (1, -5);
\node [color = red] at (0.5, -5.5) {\tiny $i$};
\draw [color = red, fill = red!20] (0+2, -6+3) rectangle (1+2, -5+3);
\node [color = red] at (0.5+2, -5.5+3) {\tiny $i$};
\draw [color = red, fill = red!20] (0+2+3, -6+3+2) rectangle (1+2+3, -5+3+2);
\node [color = red] at (0.5+2+3, -5.5+3+2) {\tiny $i$};

\draw [color = blue, fill = blue!20] (0+1, -6+1) rectangle (1+1, -5+1);
\node [scale = .7, color = blue] at (0.5+1, -5.5+1) {\tiny $i+1$};
\draw [color = blue, fill = blue!20] (0+1+2, -6+1+3) rectangle (1+1+2, -5+1+3);
\node [scale = .7, color = blue] at (0.5+1+2, -5.5+1+3) {\tiny $i+1$};
\node[rotate = 45, scale = 0.7] at (5,-4) {symbols $> i+1$};
\draw[thick] (0, -5) -- (1,-5) -- (1, -4) -- (2, -4) -- (2, -2) -- (3, -2) -- (3, -1) -- (5, -1) -- (5, 0);
\end{tikzpicture}
\end{center}
The resulting tableau is $F^i T$. 
Each tableau in this sequence is connected to the previous, so $T$ and $F^i T$ are connected.
\end{proof}

\begin{lem}
Suppose that $T$ is an efficient filling of a $k$-core $\Gamma$ with symbols $\{1, \ldots, u\}$, and that $S^i T$ is defined.
If $g > u(\Gamma)$, then $T$ and $S^i T$ are connected.
\end{lem}
\begin{proof}
Because $S^i$ is an involution, after possibly replacing $T$ with $S^i T$,
it suffices to treat the case $T[i + 1] - T[i] \equiv 1 \pmod k$.
Let $j = T[i] = T[i + 2]$.
Note that $T^{\leq i - 1}$ has addable corners with diagonal indices of residues $j$ and $j + 1$
(because $S^i T$ is defined).
In other words, the boundary segments of $T^{\leq i}$ neighboring boxes with diagonal index of residue class $j$ or $j+1$ must
proceed through:
\begin{center}
\begin{tikzpicture}[scale=.8]
\draw (0, 0)node{$\bullet$} -- (0, 1)node{$\bullet$} -- (0,2)node{$\bullet$} -- (0,3)node{$\bullet$} ;
\draw (2, 1)node{$\bullet$} -- (2, 2)node{$\bullet$} -- (2,3)node{$\bullet$} -- (3,3)node{$\bullet$} ;
\draw (5, 2)node{$\bullet$} -- (5, 3)node{$\bullet$} -- (6,3)node{$\bullet$} -- (7,3)node{$\bullet$} ;
\draw (9, 3)node{$\bullet$} -- (10, 3)node{$\bullet$} -- (11,3)node{$\bullet$} -- (12,3)node{$\bullet$} ;
\end{tikzpicture}
\end{center}

To see that $T$ and $S^i T$ are connected when $g > u$,
consider the following sequence of tableaux. First relabel $T$ according to $\{1, \ldots, u\} \mapsto \{1, \ldots, i-1, i+1, \ldots, u+1\}$ to get filling (b). Then place $i$ in the positions of $T^{\leq i - 1}$'s addable corner with diagonal index of
residue $j+1$ to get filling (c). Now all instances of $i+3$ may be replaced with $i+1$ to give filling (d).
Tableau (d) is a relabeling of $S^i T$ by $\{1, \ldots, u\} \mapsto \{1, \ldots, i+2, i+4, \ldots, u+1\}$.
\begin{center}
\begin{tikzpicture}[scale = 0.8]
\node at (4.5, 4.7) {(a) Start with $T$ \ldots};
\draw[violet, fill = violet!20] (2, 1) rectangle (3, 2);
\node[color = violet] at (2.5, 1.5) {$i+2$};
\draw[blue, fill=blue!20] (2, 2) rectangle (3, 3);
\node[color = blue] at (2.5, 2.5) {$i+1$};
\draw[red, fill=red!20] (2+3, 3) rectangle (3+3, 4);
\node[color = red] at (2.5+3, 3.5) {$i$};
\draw[blue, fill=blue!20] (2+3+1, 3) rectangle (3+3+1, 4);
\node[color = blue] at (2.5+3+1, 3.5) {$i+1$};
\draw (2, 1)node{$\bullet$} -- (2, 2)node{$\bullet$} -- (2,3)node{$\bullet$} -- (3,3)node{$\bullet$} ;
\draw (5, 3)node{$\bullet$} -- (5, 4)node{$\bullet$} -- (6,4)node{$\bullet$} -- (7,4)node{$\bullet$} ;
\node at (4, 3.5) {$\iddots$};
\draw[->] (8, 2) -- (9, 2);
\end{tikzpicture}
\hspace{.1in}
\begin{tikzpicture}[scale = 0.8]
\node at (4.5, 4.7) {(b) relabel \ldots};
\draw[orange!50!black, fill = orange!20] (2, 1) rectangle (3, 2);
\node[color = orange!50!black] at (2.5, 1.5) {$i+3$};
\draw[violet, fill=violet!20] (2, 2) rectangle (3, 3);
\node[color = violet] at (2.5, 2.5) {$i+2$};
\draw[blue, fill=blue!20] (2+3, 3) rectangle (3+3, 4);
\node[color = blue] at (2.5+3, 3.5) {$i+1$};
\draw[violet, fill=violet!20] (2+3+1, 3) rectangle (3+3+1, 4);
\node[color = violet] at (2.5+3+1, 3.5) {$i+2$};
\draw (2, 1)node{$\bullet$} -- (2, 2)node{$\bullet$} -- (2,3)node{$\bullet$} -- (3,3)node{$\bullet$} ;
\draw (5, 3)node{$\bullet$} -- (5, 4)node{$\bullet$} -- (6,4)node{$\bullet$} -- (7,4)node{$\bullet$} ;
\node at (4, 3.5) {$\iddots$};
\draw[->] (8, 2) -- (9, 2);
\end{tikzpicture}
\end{center}
\vspace{.1in}
\begin{center}
\begin{tikzpicture}[scale = 0.8]
\node at (4.5, 4.7) {(c) place $i$ \ldots};
\draw[orange!50!black, fill = orange!20] (2, 1) rectangle (3, 2);
\node[color = orange!50!black] at (2.5, 1.5) {$i+3$};
\draw[red, fill=red!20] (2, 2) rectangle (3, 3);
\node[color = red] at (2.5, 2.5) {$i$};
\draw[blue, fill=blue!20] (2+3, 3) rectangle (3+3, 4);
\node[color = blue] at (2.5+3, 3.5) {$i+1$};
\draw[violet, fill=violet!20] (2+3+1, 3) rectangle (3+3+1, 4);
\node[color = violet] at (2.5+3+1, 3.5) {$i+2$};
\draw (2, 1)node{$\bullet$} -- (2, 2)node{$\bullet$} -- (2,3)node{$\bullet$} -- (3,3)node{$\bullet$} ;
\draw (5, 3)node{$\bullet$} -- (5, 4)node{$\bullet$} -- (6,4)node{$\bullet$} -- (7,4)node{$\bullet$} ;
\node at (4, 3.5) {$\iddots$};
\draw[->] (8, 2) -- (9, 2);
\end{tikzpicture}
\hspace{.1in}
\begin{tikzpicture}[scale = 0.8]
\node at (4.5, 4.7) {(d) remove $i+3$};
\draw[blue, fill = blue!20] (2, 1) rectangle (3, 2);
\node[color = blue] at (2.5, 1.5) {$i+2$};
\draw[red, fill=red!20] (2, 2) rectangle (3, 3);
\node[color = red] at (2.5, 2.5) {$i$};
\draw[blue, fill=blue!20] (2+3, 3) rectangle (3+3, 4);
\node[color = blue] at (2.5+3, 3.5) {$i+1$};
\draw[violet, fill=violet!20] (2+3+1, 3) rectangle (3+3+1, 4);
\node[color = violet] at (2.5+3+1, 3.5) {$i+2$};
\draw (2, 1)node{$\bullet$} -- (2, 2)node{$\bullet$} -- (2,3)node{$\bullet$} -- (3,3)node{$\bullet$} ;
\draw (5, 3)node{$\bullet$} -- (5, 4)node{$\bullet$} -- (6,4)node{$\bullet$} -- (7,4)node{$\bullet$} ;
\node at (4, 3.5) {$\iddots$};
\draw[color = white, ->] (8, 2) -- (9, 2);
\end{tikzpicture}
\end{center}
Each tableau above is connected to the previous, so this shows that $T$ and $S^i T$ are connected.
\end{proof}

This establishes that $W^{\vec{e}}(X)$ is connected.  Since it is also reduced (Theorem \ref{thm:reduced}) and ``geometrically connected and geometrically reduced" is an open condition in flat proper families, we have therefore proven: 

\begin{thm}\label{thm:conn}
If $f \colon C \to \pp^1$ is a general degree~$k$ cover, and $g > u(\vec{e})$, then
$W^{\vec{e}}(C)$ is connected.
\end{thm}

\section{Normality and Irreducibility}\label{sec:norm}

Let $f \colon C \to \pp^1$ be a general degree~$k$ cover.
In this section, we show that $W^{\vec{e}}(C)$ is smooth away from
more unbalanced splitting loci of codimension $2$ or more.
Since we have already established that $W^{\vec{e}}(C)$ is Cohen--Macaulay
(Theorem~\ref{CM}),
this implies that $W^{\vec{e}}(C)$ is normal by Serre's criterion ($R_1 + S_2$).  Combining this with connectedness (Theorem \ref{thm:conn}), this establishes that $W^{\vec{e}}(C)$ is irreducible when $g > u(\vec{e})$.

To do this, we will prove the following general result regarding splitting stratifications. Suppose $\E$ is a family of vector bundles on $\pi\colon B \times \pp^1 \rightarrow B$. For any splitting type $\vec{e}$, the scheme structure on the closed splitting locus $\Sigma_{\vec{e}} \subset B$ is defined as an intersection of determinantal loci. More precisely, the locus $\Sigma_{\vec{e}}$ is the intersection over all $m$ of the Fitting support for $\rk R^1 \pi_* \E(m) \geq h^1(\pp^1, \O(\vec{e})(m))$. We use $\Sigma_{\vec{e}}^\circ \subset \Sigma_{\vec{e}}$ to denote the open where the splitting type is exactly $\vec{e}$.

\begin{prop}
Suppose $\E$ is a family of vector bundles on $\pi\colon B \times \pp^1 \rightarrow B$ such that all open splitting loci $\Sigma_{\vec{e}}^\circ$ are smooth of the expected codimension $u(\vec{e})$. Then $\Sigma_{\vec{e}}$ is smooth away from all splitting loci of codimension $\geq 2$ inside $\Sigma_{\vec{e}}$.
\end{prop}

\begin{proof}
Since the statement is local on $B$, we may assume that $B$ is affine.

 Suppose $\Sigma_{\vec{e}'} \subset \Sigma_{\vec{e}}$ has codimension $1$ for some $\vec{e}' \leq \vec{e}$. Let $Z \subset \Sigma_{\vec{e}}$ be the union of all \textit{other} splitting loci properly contained in $\Sigma_{\vec{e}}$. We will show that $\Sigma_{\vec{e}'}^\circ \subset \Sigma_{\vec{e}} \smallsetminus Z$ is a Cartier divisor. It will follow that $\Sigma_{\vec{e}}$ is smooth along $\Sigma_{\vec{e}'}^\circ$, as we now explain. If $\Sigma_{\vec{e}}$ were singular at some $b \in \Sigma_{\vec{e}'}^\circ$, then $\dim T_b\Sigma_{\vec{e}} > \dim \Sigma_{\vec{e}}= \dim \Sigma_{\vec{e}'}^\circ + 1$. Assuming $\Sigma_{\vec{e}'}^\circ \subset \Sigma_{\vec{e}} \smallsetminus Z$ is Cartier, we would find $\dim T_b \Sigma_{\vec{e}'}^\circ \geq \dim T_b \Sigma_{\vec{e}} - 1  > \dim \Sigma_{\vec{e}'}^\circ$, forcing $\Sigma_{\vec{e}'}^\circ$ to be singular at $b$, which contradicts the assumption that all open splitting loci are smooth. Repeating the argument for each divisorial $\Sigma_{\vec{e}'} \subset \Sigma_{\vec{e}}$ shows that $\Sigma_{\vec{e}}$ is smooth away from all codimension $2$ subsplitting loci.
 
 \begin{center}
 \begin{tikzpicture}[scale = 1.5]
 \node at (1, 5.3) {$\Sigma_{\vec{e}}$};
 \draw(2.84,4.5) ellipse (2cm and 1cm);
\draw [blue] (.88,4.7) .. controls (2, 4.1+.4+.5) .. (2.5+.25+.125, 4.5+.125);
\draw [blue] (2.5+.25+.125, 4.5+.125) .. controls (2.5+.25+.5+.125 +.7, 4.5+.125-.65+.1) .. (4.825, 4.1+.3);
\draw [blue] (4.1, 4.4) node {$\Sigma_{\vec{e}'}$};
\draw [violet] (2.5, 3.51) .. controls (2.2, 4.2) .. (2.5+.25+.125, 4.5+.125);
\draw [violet] (2.5+.25+.125, 4.5+.125)   .. controls (3.2, 4.8) .. (3+.15, 5.49);
\draw [violet] (2.5, 4.1) node {$Z$};
 \end{tikzpicture}
 \end{center}

First, we show that  to have $\vec{e}' < \vec{e}$ with $u(\vec{e}') = u(\vec{e}) + 1$, the splitting types must have a special shape. Let $r$ be the smallest index such that $e_r' < e_r$ and let $s$ be the largest index such that $e_s' > e_s$. (The fact that $\vec{e}' < \vec{e}$ means $r \leq s$.) Then 
\begin{equation} \label{othere}
\vec{e}' \leq \vec{e}_{rs} \colonequals (e_1, \ldots, e_{r-1}, e_r - 1, e_{r+1}, \ldots, e_{s-1}, e_s + 1, e_{s+1}, \ldots, e_k) <\vec{e},
\end{equation}
from which we see
\begin{align*}
1 &= u(\vec{e}') - u(\vec{e}) \\
&\geq u(\vec{e}_{rs}) - u(\vec{e}) \\
&= \sum_{\ell \neq r, s}  \Big([h^1(\O(e_\ell - e_r + 1)) + h^1(\O(e_\ell - e_s - 1)) + h^1(\O(e_r - 1 - e_\ell)) + h^1(\O(e_s + 1 - e_\ell))] \\[-12pt]
&\qquad\qquad - [h^1(\O(e_\ell - e_r)) + h^1(\O(e_\ell - e_s)) + h^1(\O(e_r - e_\ell)) + h^1(\O(e_s - e_\ell))]\Big) \\
&\quad + [h^1(\O(e_r - e_s - 2)) + h^1(\O(e_s - e_r + 2))] - [h^1(\O(e_r - e_s)) + h^1(\O(e_s - e_r))] \\
&= \# \{\ell : e_r - 1 \leq \ell \leq e_s - 1\} + \# \{\ell : e_r + 1 \leq \ell \leq e_s + 1\} + \begin{cases}
1 & \text{if $e_r = e_s$;} \\
2 & \text{otherwise.}
\end{cases}
\end{align*}
It follows that the non-negative quantities
$\# \{\ell : e_r - 1 \leq \ell \leq e_s - 1\} = \# \{\ell : e_r + 1 \leq \ell \leq e_s + 1\} = 0$,
and $e_r = e_s$, and $\vec{e}' = \vec{e}_{rs}$.
In other words,
after twisting (down by $e_r =e_s$), we may assume $\O(\vec{e}) = N \oplus \O^{\oplus m} \oplus P$, where all parts of $N$ have degree at most $-2$, all parts of $P$ have degree at least $2$, and $\O(\vec{e}') = N \oplus \O(-1) \oplus \O^{\oplus (m-2)} \oplus \O(1) \oplus P$.

Away from $Z$, we will show that the vector bundle $\E$ on $\pi\colon \pp^1 \times (\Sigma_{\vec{e}} \smallsetminus Z) \rightarrow (\Sigma_{\vec{e}} \smallsetminus Z)$ splits as $\E = \N \oplus \T \oplus \P$ where for any $b \in \Sigma_{\vec{e}} \smallsetminus Z$,
\begin{gather}
\N|_b \simeq N, \quad \P|_b \simeq P, \quad \text{and} \quad \T|_b \simeq \begin{cases}
\O^{\oplus m} & \text{if $b \notin \Sigma_{\vec{e}'}$;} \\
\O(-1) \oplus \O^{\oplus (m-2)} \oplus \O(1) & \text{if $b \in \Sigma_{\vec{e}'}$.}
\end{cases} \label{ntp}
\end{gather}

To construct $\N$ and $\P$, let $\mathcal{Q}(-2)$ be the cokernel of $\pi^*\pi_* \E(-2) \rightarrow \E(-2)$,
which is locally free. Define $\P$ by the exact sequence
\[0 \rightarrow \mathcal{P}(-2) \rightarrow \E(-2) \rightarrow \mathcal{Q}(-2) \rightarrow 0.\]
This sequence splits if the induced map $H^0(\mathcal{H}om(\mathcal{Q}(-2), \E(-2))) \rightarrow H^0(\mathcal{H}om(\mathcal{Q}(-2), \Q(-2)))$ is surjective, which in turn follows if we show $H^1(\mathcal{H}om(\mathcal{Q}(-2), \mathcal{P}(-2))) = 0$. Now, $\P(-2)$ is globally generated on each fiber, and $\Q(-2)$ has negative summands on each fiber, so $\mathcal{H}om(\mathcal{Q}(-2), \P(-2))$ has positive summands on every fiber. Thus, by the theorem on cohomology and base change, $R^1\pi_* \mathcal{H}om(\mathcal{Q}(-2), \P(-2)) = 0$ and so $H^1(\mathcal{H}om(\mathcal{Q}(-2), \mathcal{P}(-2))) = 0$ because $B$ is affine.

Next, define $\N$ so that $\N(1)$ is the cokernel of $\pi^*\pi_* \Q(1) \rightarrow \Q(1)$, and define $\T$ by the sequence
\[0 \rightarrow \T \rightarrow \Q \rightarrow \N \rightarrow 0.\]
The same argument as before shows that this sequence splits too. Thus, $\E = \Q \oplus \P = \N \oplus \T \oplus \P$.
By construction, this splitting satisfies \eqref{ntp}.

On $\Sigma_{\vec{e}} \smallsetminus Z$, the fibers of $R^1 \pi_* \T(-1)$ have rank at most $1$. We also have that $\pi_*\T$ and $\pi_*\T(1)$ are locally free on $\Sigma_{\vec{e}} \smallsetminus Z$.
The equations that cut out $\Sigma_{\vec{e}'} \subset B$ are the same as the equations that cut out $\Sigma_{\vec{e}} \subset B$ except at one twist: namely, when we ask for the rank of $R^1\pi_*\E(-1)$. To cut out $\Sigma_{\vec{e}}$, we ask that $\rk R^1\pi_*\E(-1) \geq h^1(\pp^1, \O(\vec{e})(-1)) \colonequals n$, whereas to cut out $\Sigma_{\vec{e}'}$, we ask that $\rk R^1\pi_*\E(-1) \geq h^1(\O(\vec{e}')(-1)) = n+1$.

 To study these equations, we must build a resolution of $R^1\pi_*\E(-1)$ by vector bundles.
 On $\Sigma_{\vec{e}} \smallsetminus Z$, the theorem on cohomology and base change shows $R^1\pi_*\N(-1)$ is a vector bundle,
 and $R^1\pi_*\P(-1) = 0$. 
We also have a resolution by vector bundles
\[U_1 \colonequals \pi_* \T \otimes H^0(\pp^1, \O(1)) \xrightarrow{\psi} U_2 \colonequals \pi_*\T(1) \rightarrow R^1\pi_*\T(-1) \rightarrow 0, \]
where $U_1$ and $U_2$ have the same rank $2m$. \hannah{add reference}
Therefore,
\[0 \oplus U_1 \xrightarrow{0 \oplus \psi} R^1 \pi_*\N(-1)  \oplus U_2 \rightarrow R^1 \pi_*\N(-1) \oplus R^1 \pi_* \T(-1) =  R^1\pi_* \E(-1)\]
is a resolution of $R^1 \pi_*\E(-1)$ by vector bundles. Now, $\Sigma_{\vec{e}'}^\circ$ is cut out inside $\Sigma_{\vec{e}} \smallsetminus Z$ by the condition that $\coker(0 \oplus \psi) \geq n+1$, or equivalently that
$\coker \psi \geq 1$.
This locus is cut by the vanishing of its determinant, which we can view as a section of the line bundle
\[ {\bigwedge}^{2m}(\pi_* \T \otimes H^0(\pp^1, \O(1)) )^\vee \otimes {\bigwedge}^{2m}(\pi_*\T(1)).\]
Hence, $\Sigma_{\vec{e}}^\circ \subset \Sigma_{\vec{e}} \smallsetminus Z$ is Cartier, as desired.
\end{proof}

In Theorem~1.2 of \cite{refinedBN}, it was established that for $f \colon C \to \pp^1$ a general degree~$k$ cover,
the open Brill-Noether splitting loci are smooth of the expected dimension. The above proposition thus implies $W^{\vec{e}}(C)$ is smooth away from all splitting loci of codimension $2$ or more.
Together with Theorem~\ref{CM}, we have thus shown:

\begin{thm} \label{normal}
Let $f\colon C \rightarrow \pp^1$ be a general genus $g$, degree $k$ cover. Then $W^{\vec{e}}(C)$ is smooth away from all codimension $2$ splitting loci. Thus $W^{\vec{e}}(C)$ is normal.
\end{thm}

\noindent
Combined with \ref{thm:conn}, we have therefore proven:

\begin{thm}
If $f \colon C \to \pp^1$ is a general degree~$k$ cover, and $g > u(\vec{e})$, then
$W^{\vec{e}}(C)$ is irreducible.
\end{thm}

\section{Monodromy}\label{sec:mon}

For the remainder of the paper, we suppose that the characteristic of
our ground field does not divide $k$.
Our final task is to show that the universal $\W^{\vec{e}}$ has a unique irreducible
component dominating a component of the unparameterized Hurwitz stack $\H_{k,g}$
when $g \geq u(\vec{e})$.
When $g > u(\vec{e})$ --- or when $k = 2$ in which case $N(\vec{e}) = 1$ for all $\vec{e}$ ---
we have that $W^{\vec{e}}(C)$ is irreducible for $C$ general.
So for the remainder of this section, we suppose $g = u(\vec{e})$ and $k > 2$.

When $g = u(\vec{e})$,
we have shown that $W^{\vec{e}}(X)$ is a reduced finite set of line bundles. Using this, we obtain:

\begin{lem} \label{lem:etale} Let $\mathfrak{f}^* \colon \X^* \to \P^*$ be a deformation of $f\colon X \to P$ to a smooth cover,
with smooth total space (c.f.\ Section~\ref{our_degen}).
If $\mathcal{C} \to \P' \to B$ is a family of smooth covers containing $\mathfrak{f}^*$, over a reduced base $B$,
then $W^{\vec{e}}(\mathcal{C}/B) \to B$ is \'etale near $\mathfrak{f}^*$.
\end{lem}
\begin{proof}
Because $u(\vec{e}) = g$,  every component of $W^{\vec{e}}(\C/B)$ has dimension at least $\dim B$. 
Moreover, $W^{\vec{e}}(\C/B) \to B$ is proper, and the fiber over $\mathfrak{f}^*$ is a finite set of reduced points.
Since $B$ is reduced, we conclude that the map is \'etale near  $\mathfrak{f}^*$. 
\end{proof}

Since $f \colon X \to P$ is separable, $\mathfrak{f}^*$ is also separable. In particular, its cotangent complex is punctual,
so $\H_{k,g}$ is smooth at $\mathfrak{f}^*$.
We can therefore apply this lemma to the universal family over a component $B$ of $\H_{k,g}$.

In greater generality, suppose that $B$ is any irreducible base, and $\pi \colon W \to B$ is \'etale near $b \in B$.
Then, any irreducible component of $W$ dominating $B$ meets $\pi^{-1}(b)$,
and every point of $\pi^{-1}(b)$ is contained in a unique irreducible component of $W$, which
dominates $B$.
To show $W$ has a unique irreducible component dominating $B$, it thus suffices to show
that any two points of $\pi^{-1}(b)$ are contained in the same irreducible component of $W$.
In particular, suppose that $B'$ is irreducible and the image of $B' \to B$ meets $b$.
If $W \times_B B' \to B$ has a unique irreducible component dominating $B'$, then $W$ has a unique irreducible component dominating $B$.

Therefore, if there is \emph{some} family $\mathcal{C}/B'$ over a reduced irreducible base $B'$, containing $\mathfrak{f}^*$, so that
$W^{\vec{e}}(\mathcal{C}/B') \to B'$ has a unique component dominating $B'$, then the universal $\W^{\vec{e}}$ has a unique component dominating our component of $\H_{k,g}$.
The argument will proceed in the following steps:

\begin{enumerate}
\item In Section~\ref{sec:loc},
we define the stack $\H_{k,g,2}$ of degree $k$ genus $g$ covers with total ramification at $2$ points,
and partial and total compactifications thereof:
\[\H_{k,g,2} \subseteq \bar{\H}^{\text{sm-ch}}_{k,g,2} \subseteq \bar{\H}^{\text{ch}}_{k,g,2} \subseteq \bar{\H}_{k,g,2}.\]
The universal curve $\C^{\text{sm-ch}}$ over $\bar{\H}^{\text{sm-ch}}_{k,g,2}$
will be a family of chain curves with smooth total space. We may therefore construct
the universal $\vec{e}$-positive locus $W^{\vec{e}}(\C^{\text{sm-ch}})$
as in Section~\ref{sec:lim-line}.
We will also observe that:
\smallskip
\begin{enumerate}
\item $X$ lies in $\bar{\H}_{k,g,2}^{\text{sm-ch}}$, and $\H_{k,g,2}$ contains a deformation of $X$ as in Lemma~\ref{lem:etale},
\item $\bar{\H}^{\text{ch}}_{k,g,2}$ is smooth, so $X$ lies in a \emph{unique} component $\bar{\H}_{k,g,2}^{\text{sm-ch}, \circ}$
of $\bar{\H}_{k,g,2}^{\text{sm-ch}}$,
\item $W^{\vec{e}}(\C^{\text{sm-ch}}) \to \bar{\H}_{k,g,2}^{\text{sm-ch}}$ is \'etale near $X$.
\end{enumerate}
\end{enumerate}

We may therefore reduce to studying $W^{\vec{e}}(\C^{\text{sm-ch}})$.
(The reason for this first reduction is to sidestep questions related to the existence of a nice compactification of the Hurwitz space
in positive characteristic, given that we will need these compactifications of $\H_{k,g,2}$ anyways
later in our argument.)
In light of (1)(c), our problem is now to show that every two points
of $W^{\vec{e}}(X)$ lie in the same irreducible component of $W^{\vec{e}}(\C^{\text{sm-ch}})$.

Recall that $W^{\vec{e}}(X)$ is the reduced
finite set of line bundles $L_T$ indexed by the efficient fillings $T$
of $\Gamma(\vec{e})$ (c.f.\ Definition~\ref{WT}).
Therefore, it suffices to see that for any two tableaux $T$ and $T'$ of shape $\Gamma(\vec{e})$,
the irreducible components of $W^{\vec{e}}(\C^{\text{sm-ch}})$
containing $T$ and $T'$ coincide.
Because any two tableaux can be connected via a sequence of braid moves (c.f.\ Section~\ref{sec:conn}),
it suffices to show that $L_T$ and $L_{F^i T}$
(respectively $L_T$ and $L_{S^i T}$), when defined, lie in the same irreducible component.

\begin{enumerate}
\addtocounter{enumi}{1}
\item In Section~\ref{mon-flips} (respectively Section~\ref{mon-shuffle}),
we restrict $W^{\vec{e}}(\C^{\text{sm-ch}})$ to certain families in $\bar{\H}_{k,g,2}^{\text{sm-ch}, \circ}$ whose closures contain $X$;
these families arise by smoothing nodes of $X$ and are themselves parameterized by $\H_{k,2,2}^\circ$
(respectively $\H_{k,3,2}^\circ$).

\smallskip
\noindent
The restriction of $W^{\vec{e}}(\C^{\text{sm-ch}})$ to these families is not irreducible.
Nonetheless, for each $T$ such that $F^i T$ (respectively $S^i T$) is defined,
we describe substacks $Y_F(i, T)$ (respectively $Y_S(i, T)$) of the restriction of $W^{\vec{e}}(\C^{\text{sm-ch}})$
to these families; the fiber over $X$ of the
closure of $Y_F(i, T)$ (respectively $Y_S(i, T)$) in $W^{\vec{e}}(\C^{\text{sm-ch}})$
consists of $L_T$ and $L_{F^i T}$
(respectively $L_T$ and $L_{S^i T}$).

\medskip

\item Finally, in Section~\ref{mon-disc}, we prove that $Y_F(i, T)$ and $Y_S(i, T)$
are irreducible, thereby establishing that $L_T$ and $L_{F^i T}$ (respectively $L_T$ and $L_{S^i T}$)
lie in the same irreducible component of $W^{\vec{e}}(\C^{\text{sm-ch}})$, as desired.

\smallskip
\noindent
We do this by observing that $Y_F(i, T)$ (respectively $Y_S(i, T)$) extends naturally
over the entirety of $\bar{\H}_{k,2,2}$ (respectively $\bar{\H}_{k,3,2}$).
Moreover, they remain generically \'etale over a certain boundary stratum $R_2$ (respectively $R_3$),
where we can write down explicit equations and check irreducibility.
\end{enumerate}

\subsection{\boldmath Covers with $2$ points of total ramification\label{sec:loc}}

\begin{defin} \label{def-h2kg}
Let $\M_{g, 2}$ denote the moduli stack of curves of genus $g$ with $2$
marked points, and $\bar{\M}_{g, 2}$ denote its Deligne--Mumford compactification
by stable curves.
Write $\H_{k, g, 2} \subseteq \M_{g, 2}$ for the substack of
$(C, p, q) \in \M_{g, 2}$ with $\O_C(kp) \simeq \O_C(kq)$.
(In other words, for a scheme $B$, the $B$-points of
$\H_{k, g, 2}$ parameterize relative smooth curves $C \to B$ equipped with a pair of sections
$\{p, q\}$,
such that $B$ can be covered by opens $U$ with $\O_{C|_U}(kp) \simeq \O_{C|_U}(kq)$.)

Write $\bar{\H}_{k, g, 2} \subseteq \bar{\M}_{g, 2}$ for the closure of $\H_{k, g, 2}$.
Let $\bar{\H}_{k, g, 2}^{\text{sm-ch}} \subseteq \bar{\H}_{k, g, 2}$ (respectively $\bar{\H}_{k, g, 2}^{\text{ch}} \subseteq \bar{\H}_{k, g, 2}$)
denote the open substack parameterizing chains of smooth (respectively irreducible) curves where the marked points
are on opposite ends. Let $\C^{\text{sm-ch}}$ denote the universal curve over $\bar{\H}_{k, g, 2}^{\text{sm-ch}}$.
\end{defin}

Any $(C, p, q) \in \M_{g, 2}$ lies in $\H_{k, g, 2}$ if and only if $C$
admits a map $C \to \pp^1$
of degree $k$, totally ramified at $p$ and $q$.
Such a map is unique up to $\operatorname{Aut} \pp^1$,
so there is a natural map $\H_{k, g, 2} \to \H_{k, g}$.
The boundary can be understood explicitly in a similar fashion.
Suppose $C = C^1 \cup_{p^1} \cup_{p^2} \cdots \cup_{p^{n-1}} C^n$ with $p = p^0 \in C^1$ and $q = p^n \in C^n$ is a chain of irreducible curves where the
marked points are on opposite ends. By the theory of admissible covers (c.f.\ Section~5 of~\cite{liu}), $(C, p, q) \in \bar{\H}_{k,g,2}^{\text{ch}}$ if and only if $C$ admits a map of degree $k$
to a chain of $n$ copies of $\pp^1$ totally ramified over the $p^i$. Such a map exists if and only if $p^{i} - p^{i-1}$ is $k$-torsion in $\Pic^0(C^i)$ for $i = 1, \ldots, n$, in which case it is unique.

Note that by construction, $(X, p^0, p^g)$ lies in $\bar{\H}_{k, g, 2}^{\text{sm-ch}}$,
and the deformation constructed in Section~\ref{our_degen} lies in $\H_{k, g, 2}$.

\begin{lem} \label{lem:smooth}
$\bar{\H}_{k,g,2}^{\mathrm{ch}}$ is smooth, as is the total space $\C^{\mathrm{sm-ch}}$ of the universal curve
over $\bar{\H}_{k,g,2}^{\mathrm{sm-ch}}$.
\end{lem}
\begin{proof}
Let $C = C^1 \cup_{p^1} C^2 \cup_{p^2} \cdots \cup_{p^{n - 1}} C^n$ be a point of $\bar{\H}_{k,g,2}^{\mathrm{ch}}$,
with marked points $p = p^0 \in C^1$ and $q = p^n \in C^n$.
Let $f' \colon C \to P$ be the unique map of degree $k$
to a chain of $n$ copies of $\pp^1$ totally ramified over the $p^i$.
Because the characteristic does not divide $k$ by assumption, $f'$ is separable.

By formal patching (c.f.\ Lemma~5.6 of~\cite{liu}),
a deformation of $f'$ is uniquely determined by deformations in formal neighborhoods of every branch point $b$ of
$f'$, and every node $q^i$ of $P$.
Near a branch point $b$, the deformation space is smooth since the relative cotangent complex is punctual.
Near a node $q^i$, write $x$ and $y$ for local coordinates on $C^i$ and $C^{i + 1}$ at $p^i$.
Since the map is totally ramified and $k$ is not a multiple of the characteristic,
$a = x^k$ and $b = y^k$ give local coordinates on the two copies of $\pp^1$ meeting at $q^i$.
A local versal deformation space is then smooth of dimension~$1$:
a versal deformation with coordinate $t$ is
$\Spec K[[x, y, t]] / (xy - t) \to \Spec K[[a, b, t]] / (ab - t^k)$.
Thus $\bar{\H}_{k,g,2}^{\mathrm{ch}}$ is smooth.

Finally, along the fiber $C$, the map $\C^{\mathrm{sm-ch}} \to \bar{\H}_{k,g,2}^{\mathrm{sm-ch}}$
is smooth away from the $p^i$.
Therefore the only possible singularities of the total space $\C^{\mathrm{sm-ch}}$
along $C$ occur at the $p^i$.
In a formal neighborhood of $p^i$, the total space $\C^{\mathrm{sm-ch}}$
is a pullback under a smooth map of the total space $\Spec K[[x, y, t]] / (xy - t)$
of the universal source over the versal deformation space appearing above, which is smooth.
\end{proof}

In particular, $X$ lies in a unique irreducible component $\bar{\H}_{k,g,2}^{\mathrm{sm-ch}, \circ}$ of $\bar{\H}_{k,g,2}^{\mathrm{sm-ch}}$. The work in Section~\ref{sec:lim-line} defines a universal $\vec{e}$-positive locus $W^{\vec{e}}(\C^{\text{sm-ch}})$, which is proper over $\bar{\H}_{k, g, 2}^{\text{sm-ch}}$, and whose fiber over $X$ is $W^{\vec{e}}(X)$,
which we have shown is a reduced finite set.
By regeneration (Theorem~\ref{reg-st}), every point of $W^{\vec{e}}(X)$ lies in a component of $W^{\vec{e}}(\C^{\text{sm-ch}})$
dominating $\bar{\H}_{k,g,2}^{\mathrm{sm-ch}, \circ}$.
Since $\bar{\H}_{k,g,2}^{\mathrm{sm-ch}, \circ}$ is reduced by Lemma~\ref{lem:smooth},
we conclude as in the proof of Lemma~\ref{lem:etale} that $W^{\vec{e}}(\C^{\text{sm-ch}}) \to \bar{\H}_{k,g,2}^{\mathrm{sm-ch}, \circ}$
is \'etale near $X$.

\subsection{Flips in monodromy\label{mon-flips}}
Suppose that the flip $F^i T$ is defined. 
Our deformation of $X$ will be obtained by smoothing the node $p^i$.
Such curves are parametrized by a component $\H^\circ_{k, 2, 2}$ of 
$\H_{k, 2, 2}$. The map $\iota \colon \H^\circ_{k, 2, 2} \hookrightarrow \bar{\H}_{k, g, 2}^{\text{sm-ch}}$
sends $(C, p, q)$ to
the chain curve obtained by attaching $E^1 \cup_{p^1} \cdots \cup_{p^{i-2}} E^{i-1}$ to $C$ so that $p^{i-1}$ is identified with $p$,
and attaching $E^{i+2} \cup \cdots \cup E^g$ so that $p^{i+1}$ is identified with $q$.
\vspace{0.5\baselineskip}
\begin{center}
\begin{tikzpicture}
\draw (-2, 0) .. controls (-1, -1) and (0, -1) .. (1, 0);
\draw (0, 0) .. controls (1, -1) and (2, -1) .. (3, 0);
\draw (2, 0) .. controls (3.3, -1) and (7-1.3, -1) .. (7, 0);
\filldraw (7.8+4, -0.5) circle[radius=0.02];
\filldraw (7.5+4, -0.5) circle[radius=0.02];
\filldraw (7.2+4, -0.5) circle[radius=0.02];
\filldraw (7.8-10, -0.5) circle[radius=0.02];
\filldraw (7.5-10, -0.5) circle[radius=0.02];
\filldraw (7.2-10, -0.5) circle[radius=0.02];
\draw (6, 0) .. controls (7, -1) and (8, -1) .. (9, 0);
\draw (8, 0) .. controls (9, -1) and (10, -1) .. (11, 0);
\filldraw (2.5-2, -0.42) circle[radius=0.03];
\filldraw (4.5-2+.085, -0.42+.07) circle[radius=0.03];
\filldraw (4.5+2-.085, -0.42+.07) circle[radius=0.03];
\node at (4.5-2+.085, -.7) {$p^{i-1}$};
\node [rotate = 90] at (4.5-2+.085, -.7-.3) {$=$};
\node at (4.5-2+.085, -.7-.6) {$p$};
\node at (4.5+2-.085-.1, -.7) {$p^{i+1}$};
\node [rotate = 90] at (4.5+2-.085-.1, -.7-.3) {$=$};
\node at (4.5+2-.085-.1, -.7-.6) {$q$};
\filldraw (10.5-2, -0.42) circle[radius=0.03];
\draw (2.5-2, -0.8) node{$p^{i-2}$};
\draw (4.5, -1) node{$C$};
\draw (10.5-2, -0.8) node{$p^{i+2}$};
\draw (1.5, -0.95) node{$E^{i-1}$};
\draw (3.5-4, -0.95) node{$E^{i-2}$};
\draw (9.5, -0.95) node{$E^{i +3}$};
\draw (11.5-4, -0.95) node{$E^{i+2}$};
\end{tikzpicture}
\end{center}
Our original chain curve $X$ is in the closure of $\iota(\H_{k, 2, 2}^\circ)$.  

\begin{lem} \label{g2}
Let $L$ be a limit line bundle on $\iota(C, p, q)$ with $L^{E^t} = L_T^t$ for $t \notin \{i, i + 1\}$,
such that there exist points $x, y \in C$ with
\[L^C \simeq \O((T[i]+i-1)p + (d - T[i] - i)q + x) \simeq \O((T[i+1] +i - 1)p + (d - T[i+1] - i)q + y)\]
(which forces $\O_C(x - y) \simeq \O_C((T[i+1] - T[i])(p - q))$).
Then $L$ is limit $\vec{e}$-positive.
\end{lem}
\begin{proof}
Because $L_T$ is $\vec{e}$-positive, and $L^{E^t} = L_T^t$ for $t \notin \{i, i + 1\}$, it suffices to show that for all $a, b \in \zz$
(with the notation of Definition~\ref{limit-linebundle}):
\begin{equation} \label{want2}
h^0(L^C(-ap - bq)) \geq f(a, b) \colonequals \min_{d^i + d^{i + 1} = d} h^0\big((L_T)^{E^i \cup E^{i + 1}}_{(d^i, d^{i + 1})}(-ap - bq)\big).
\end{equation}
If $a + b \geq d$, then straight-forward casework (using our assumption
that $T[i + 1] \not \equiv T[i] - 1$ mod $k$) implies
\[f(a, b) \leq h^0\big((L_T)^{E^i \cup E^{i + 1}}_{(d - b, b)}(-ap - bq)\big) = 0,\]
and so \eqref{want2} holds. Otherwise, if $a + b \leq d - 1$, then
straight-forward casework (using our assumption
that $T[i + 1] \not \equiv T[i] + 1$ mod $k$)
implies
\[f(a, b) \leq h^0\big((L_T)^{E^i \cup E^{i + 1}}_{(d - b - 1, b + 1)}(-ap - bq)\big) = \begin{cases}
1 & \text{if $(a, b) = (T[i] + i - 1, d - T[i] - i)$;} \\
1 & \text{if $(a, b) = (T[i + 1] + i - 1, d - T[i + 1] - i)$;} \\
d - a - b - 1 & \text{otherwise.}
\end{cases}\]
This immediately implies \eqref{want2} in the ``otherwise'' case by
Riemann--Roch. In the first two cases, this implies \eqref{want2}
by our assumption that $L^C(-ap - bq)$ is effective for these
values of $a$ and $b$.
\end{proof}

On a smooth curve of genus $2$, the map $C \times C \to \Pic^0 C$
defined by $(x, y) \mapsto \O_C(x - y)$ is finite of degree $2$
away from the diagonal (which is contracted to the identity in $\Pic^0 C$).
Lemma~\ref{g2} thus produces two limit $\vec{e}$-positive line bundles
on the general curve in $\iota(\H^\circ_{k, 2, 2})$, i.e.\ a substack $Y_F(i, T)$ of the restriction of
$W^{\vec{e}}(\C^{\text{sm-ch}})$  to $\iota(\H^\circ_{k, 2, 2})$.
Limiting to $X$, we obtain two $\vec{e}$-positive line bundles
(which must be distinct because $W^{\vec{e}}(\C^{\text{sm-ch}}) \to \bar{\H}^{\text{sm-ch}}_{k,g,2}$ is \'etale near $X$).
By construction these correspond to tableaux $T_1$ and $T_2$ satisfying $T_j[t] = T[t]$ for $t \neq \{i, i + 1\}$.
But there are only two such tableaux: $T$ itself and $F^i T$.
Therefore the fiber of the closure of $Y_F(i, T)$ over $X$ consists of $L_T$ and $L_{F^i T}$ as desired.

Note that, although the map $Y_F(i, T) \to W^{\vec{e}}(\C^{\text{sm-ch}})$
depends on $i$ and $T$, the stack $Y_F(i, T)$ and the map $Y_F(i, T) \to \H_{k,2,2}^\circ$ depend only on
$n \colonequals T[i + 1] - T[i]$, up to sign and modulo $k$.
We therefore write $Y_F(n) = Y_F(i, T) \to \H_{k,2,2}^\circ$.

\subsection{Shuffles in monodromy\label{mon-shuffle}}
Suppose that the shuffle $S^i T$ is defined, i.e., $T[i] = T[i+2]$ and $T[i+1] = T[i] \pm 1$.  Without loss of generality, suppose that
$T[i + 1] = T[i] + 1$.
Our deformation of $X$ will be obtained by simultaneously smoothing the nodes $p^i$ and $p^{i + 1}$.
Such curves are parametrized by a component $\H^{\circ}_{k, 3, 2}$ of 
$\H_{k, 3, 2}$. The map $\iota \colon \H^\circ_{k, 3, 2} \hookrightarrow \bar{\H}_{k, g, 2}^{\text{sm-ch}}$
sends $(C, p, q)$ to
the chain curve obtained by attaching $E^1 \cup_{p^1} \cdots \cup_{p^{i-2}} E^{i-1}$ to $C$ so that $p^{i-1}$ is identified with $p$, and attaching $E^{i+3} \cup \cdots \cup E^g$ so that $p^{i+2}$ is identified with $q$. 
\begin{center}
\begin{tikzpicture}
\draw (-2, 0) .. controls (-1, -1) and (0, -1) .. (1, 0);
\draw (0, 0) .. controls (1, -1) and (2, -1) .. (3, 0);
\draw (2, 0) .. controls (3, -1) and (8, -1) .. (9, 0);
\filldraw (7.8+4, -0.5) circle[radius=0.02];
\filldraw (7.5+4, -0.5) circle[radius=0.02];
\filldraw (7.2+4, -0.5) circle[radius=0.02];
\filldraw (7.8-10, -0.5) circle[radius=0.02];
\filldraw (7.5-10, -0.5) circle[radius=0.02];
\filldraw (7.2-10, -0.5) circle[radius=0.02];
\draw (8, 0) .. controls (9, -1) and (10, -1) .. (11, 0);
\filldraw (2.5-2, -0.42) circle[radius=0.03];
\filldraw (4.5-2+.085, -0.42+.07) circle[radius=0.03];
\filldraw (4.5+2-.085+2, -0.42+.07) circle[radius=0.03];
\node at (4.5-2+.085, -.7) {$p^{i-1}$};
\node [rotate = 90] at (4.5-2+.085, -.7-.3) {$=$};
\node at (4.5-2+.085, -.7-.6) {$p$};
\node at (4.5+2-.085-.1+2, -.7) {$p^{i+2}$};
\node [rotate = 90] at (4.5+2-.085-.1+2, -.7-.3) {$=$};
\node at (4.5+2-.085-.1+2, -.7-.6) {$q$};
\draw (2.5-2, -0.8) node{$p^{i-2}$};
\draw (5.5, -1) node{$C$};
\draw (1.5, -0.95) node{$E^{i-1}$};
\draw (3.5-4, -0.95) node{$E^{i-2}$};
\draw (9.5, -0.95) node{$E^{i +3}$};
\end{tikzpicture}
\end{center}
Our original chain curve $X$ is in the closure of $\iota(\H^\circ_{k, 3, 2})$.

\begin{lem} \label{g3}
Suppose that $C$ is a non-hyperelliptic curve of genus $3$. Let
$L$ be a limit line bundle on $\iota(C, p, q)$ with $L^{E^t} = L_T^t$ for $t \notin \{i, i + 1, i + 2\}$,
such that 
\[L^C \simeq \O_C(z + (T[i]+i)p + (d - T[i] - i-1)q)\] 
for $z \in \{x, y\}$, where $x, y \in C$ is the pair of points such that $p +q + x + y \sim K_C$
(i.e.\ the two points colinear with $p$ and $q$ in the canonical model of $C$ as a plane quartic).
 \begin{center}
 \begin{tikzpicture}[rotate = 45]
 \draw[domain=.4:3.9,smooth,variable=\x] plot ({\x},{(\x-.7)*(\x-1.6)*(\x-2.6)*(\x-3.8)});
 \draw[dashed] (-.3, 0) -- (4.7, 0);
 \node at (.7, 0) {$\bullet$};
 \node [color = violet] at (1.6, 0) {$\bullet$};
 \node [color = violet] at (2.6, 0) {$\bullet$};
 \node at (3.8, 0) {$\bullet$};
 \node at (.5, 3.2) {$C$};
 \node[color=violet] at (1.5, .3) {$x$};
  \node[color=violet] at (2.7, .3) {$y$};
  \node at (0.45, .2) {$p$};
    \node at (4.05, .25) {$q$};
 \end{tikzpicture}
 \end{center}
Then $L$ is limit $\vec{e}$-positive.
\end{lem}
\begin{proof}
Because $L_T$ is $\vec{e}$-positive, and $L^{E^t} = L_T^t$ for $t \notin \{i, i + 1, i + 2\}$,
it suffices to show that for all $a, b \in \zz$:
\begin{equation} \label{want3}
h^0(L^C(-ap - bq)) \geq f(a, b) \colonequals \min_{d^i + d^{i + 1} + d^{i + 2} = d} h^0\big((L_T)^{E^i \cup E^{i + 1} \cup E^{i + 2}}_{(d^i, d^{i + 1}, d^{i + 2})}(-ap - bq)\big).
\end{equation}
If $a + b \geq d - 1$, then straight-forward casework implies
\[f(a, b) \leq h^0\big((L_T)^{E^i \cup E^{i + 1} \cup E^{i + 2}}_{(a, d - a - b, b)}(-ap - bq)\big) = \begin{cases}
1 & \text{if $(a, b) = (T[i] + i, d - T[i] - i - 1)$;} \\
0 & \text{otherwise}.
\end{cases}\]
This immediately implies \eqref{want3} in the ``otherwise'' case,
and shows that
in the first case, \eqref{want3} follows from the condition
\begin{equation} \label{c1}
h^0(L^C(-(T[i] + i)p - (d - T[i] - i - 1)q)) \geq 1.
\end{equation}
Otherwise, if $a + b \leq d - 2$, then straight-forward casework implies
\[f(a, b) \leq h^0\big((L_T)^{E^i \cup E^{i + 1} \cup E^{i + 2}}_{(a + 1, d - a - b - 2, b + 1)}(-ap - bq)\big) = \begin{cases}
2 & \text{if $(a, b) = (T[i] + i - 1, d - T[i] - i - 2)$;} \\
1 & \text{if $(a, b) = (T[i] + i - 1, d - T[i] - i - 1)$;} \\
1 & \text{if $(a, b) = (T[i] + i, d - T[i] - i - 2)$;} \\
d - a - b - 2 & \text{otherwise.}
\end{cases}\]
This immediately implies \eqref{want3} in the ``otherwise'' case by Riemann--Roch. Since vanishing at at $p$
or $q$ imposes at most one condition on sections of any line bundle, in the first
three cases, \eqref{want3} follows from the single condition
\begin{equation} \label{c2}
h^0(L^C(-(T[i] + i - 1)p - (d - T[i] - i - 2)q)) \geq 2.
\end{equation}
We conclude by observing that $L^C$ satisfies
\eqref{c1} and \eqref{c2} by its definition.
\end{proof}

Lemma~\ref{g3} thus produces two limit $\vec{e}$-positive line bundles
on the general curve in $\iota(\H^\circ_{k,3, 2})$, i.e.\ a substack $Y_S(i, T)$ of the restriction of
$W^{\vec{e}}(\C^{\text{sm-ch}})$  to $\iota(\H^\circ_{k, 3, 2})$.
As in the previous ``flip'' case, the
fiber of the closure of $Y_S(i, T)$ over $X$ consists of $L_T$ and $L_{S^i T}$ as desired.
Moreover, $Y_S(i, T) \to \H^\circ_{k,3, 2}$ is independent of $i$ and $T$.
We therefore write $Y_S = Y_S(i, T) \to \H^\circ_{k,3, 2}$.

\subsection{\boldmath Irreducibility of $Y_{F}(n)$ and $Y_S$\label{mon-disc}}
Our final task is to show that the following two double covers have a unique irreducible component dominating
the desired component of the base:
\begin{description}
\item[\boldmath $Y_F(n)$] The double cover of $\H_{k, 2, 2}$
parameterizing points $x$ and $y$ with $\O(x - y) \simeq \O(n(p - q))$
(where $n$ is an integer not equal to $0, \pm 1$ mod $k$);
\item[\boldmath $Y_S$] The double cover of the complement of the hyperelliptic locus in $\H_{k, 3, 2}$ parameterizing points
$x$ and $y$ with $\O(x + y) \simeq \omega(-p - q)$ 
(where $k > 2$).
\end{description}
Notice that the definition of $Y_F(n)$ (respectively $Y_S$) extends naturally to the entire closure $\bar{\H}_{k,2, 2}$ (respectively $\bar{\H}_{k,3, 2}$), although it is not a priori finite flat of degree $2$.

\begin{prop} \label{prop:finflat} \hfill
\begin{enumerate}
\item $Y_F(n)$ is finite flat of degree $2$
over the open substack of $\bar{\H}^{\mathrm{ch}}_{k, 2, 2}$ defined by $(C, p, q)$ satisfying the following conditions:
\begin{enumerate}
\item $C$ is irreducible,
\item $p - q$ is exactly $k$-torsion on $C$,
\item $p - q$ is exactly $k$-torsion on the partial normalization of $C$ at any node
(this condition is vacuous if $C$ is smooth and implies the previous one if $C$ is singular).
\end{enumerate}
\item $Y_S$ is finite flat of degree $2$
over the open substack of $\bar{\H}^{\mathrm{ch}}_{k, 3, 2}$ defined by $(C, p, q)$ satisfying the following conditions:
\begin{enumerate}
\item $C$ is irreducible,
\item $p$ and $q$ are not conjugate under the hyperelliptic involution if $C$ is hyperelliptic
(this condition is vacuous if $C$ is not hyperelliptic), 
\item $p$ and $q$ are not conjugate under the hyperelliptic involution on the partial normalization of $C$ at any node
(this condition is vacuous if $C$ is smooth and implies the previous one if $C$ is singular).
\end{enumerate}
\end{enumerate}
\end{prop}
\begin{proof}
Write $U_2 \subseteq \bar{\H}_{k, 2, 2}^{\mathrm{ch}}$ (respectively $U_3 \subseteq \bar{\H}_{k, 3, 2}^{\mathrm{ch}}$)
for the open substacks defined by the above conditions. 

\smallskip

\paragraph{\boldmath \textbf{For $Y_F(n)$:}} Consider any $(C, p, q) \in U_2$.
Since $C$ is irreducible,
the condition $\O(x - y) \simeq \O(n(p - q))$ is equivalent to
$\O(x + \bar{y}) \simeq \omega(n(p - q))$,
where $\bar{y}$ denotes the conjugate of $y$ under the hyperelliptic involution.
The line bundle $\omega(n(p - q))$ is of degree $2$, and not isomorphic to $\omega$
because $p - q$ is exactly $k$ torsion and $n \not\equiv 0$ mod $k$.
Therefore, since $C$ is irreducible, $\omega(n(p - q))$ has a unique section (up to scaling),
vanishing on a Cartier divisor $D \subset C$ of degree $2$.
If $D$ were supported at a node of $C$, consider the partial normalization $C^\nu$ of $C$
at this node, and write $s$ and $t$ for the points on $C^\nu$ above this node.
Then $\omega_C(n(p - q)) \simeq \O(s + t)$ as line bundles on $C^\nu$.
Since $\omega_C \simeq \omega_{C^\nu}(s + t) \simeq \O(s + t)$,
we would have $\O_{C^\nu}(n(p - q)) \simeq \O_{C^\nu}$.
But this is impossible since $p - q$ is exactly $k$-torsion on $C^\nu$ by assumption.
Thus $D \subset C_{\text{sm}}$.

It thus remains to see that these divisors $D$ fit together to form a Cartier divisor $\mathcal{D}$
on the universal curve $\pi \colon \mathcal{C} \to U_2$ of relative degree $2$
(which will then be supported in the smooth locus and identified with $Y_F(n)$).
Write $\mathfrak{p}, \mathfrak{q} \colon U_2 \to \mathcal{C}$ for the universal sections.
Observe that $\omega_{\mathcal{C} / U_2} (n(\mathfrak{p} - \mathfrak{q}))$ is a line bundle on $\mathcal{C}$,
with a unique section up to scaling on every geometric fiber.
Moreover, by Lemma~\ref{lem:smooth}, the base $U_2$ is smooth, and in particular reduced.
Cohomology and base change thus implies that $\pi_* \omega_{\mathcal{C} / U_2} (n(\mathfrak{p} - \mathfrak{q}))$
is a line bundle on $U_2$.
Working locally on $U_2$, we may trivialize it by picking a section, which gives a section of
$\omega_{\mathcal{C} / U_2} (n(\mathfrak{p} - \mathfrak{q}))$, vanishing along a Cartier divisor $\mathcal{D}$.

\smallskip

\paragraph{\boldmath \textbf{For $Y_S$:}}
Consider any $(C, p, q) \in U_3$.
Since $C$ is irreducible, and $p$ and $q$ are not conjugate under the hyperelliptic involution
if $C$ is hyperelliptic, we have $h^0(C, \O(p + q)) = 1$.
By Serre duality, $\omega(-p - q)$ has a unique section (up to scaling),
vanishing on a Cartier divisor $D \subset C$ of degree $2$.
If $D$ were supported at a node of $C$, consider the partial normalization $C^\nu$ of $C$
at this node, and write $s$ and $t$ for the points on $C^\nu$ above this node.
Then $\omega_C(-p-q) \simeq \O(s + t)$ as line bundles on $C^\nu$.
Since $\omega_C \simeq \omega_{C^\nu}(s + t)$,
we would have $\O_{C^\nu}(p+q) \simeq \omega_{C^\nu}$.
But this is impossible since $p$ and $q$ are not conjugate under the hyperelliptic involution
on $C^\nu$ by assumption.
Thus $D \subset C_{\text{sm}}$.

As in the previous case, these divisors $D$ fit together to form a Cartier divisor $\mathcal{D}$
on the universal curve $\pi \colon \mathcal{C} \to U_3$ of relative degree $2$.
\end{proof}

We now show that $Y_F(n)$ (respectively $Y_S$) has a unique irreducible component dominating
$\bar{\H}^{\text{ch},\circ}_{k, 2, 2}$ (respectively $\bar{\H}^{\text{ch},\circ}_{k, 3, 2}$).
To do this, we will restrict these double covers to certain schemes $R_h \to \bar{\H}^{\text{ch},\circ}_{k, h, 2}$
(with $h = 2$ respectively $h = 3$),
where we can write down the equations of $Y_F(n)$ and $Y_S$ explicitly and see that they are irreducible.
The scheme $R_2$ is an open in $\{r_1, r_2\} \in \operatorname{Sym}^2 \pp^1$,
respectively $R_3$ is an open in $\{r_1, r_2, r_3\} \in \operatorname{Sym}^3 \pp^1$.
Let $\zeta$ denote a primitive $k$th root of unity
(which exists by our assumption that the characteristic does not divide $k$).
Our schemes $R_h$ will parameterize stable curves of geometric genus $0$ of the following forms:

\begin{center}
\begin{tikzpicture}[scale=1.3]
\filldraw (0, 1) circle[radius=0.03];
\filldraw (4, 1) circle[radius=0.03];
\draw (0, 1.15) node{$p = 0$};
\draw (4, 1.15) node{$q = \infty$};
\draw (0.7, 0.67) node{$r_1$};
\draw (1.4, 0.67) node{$r_1 \zeta$};
\draw (2.7, 0.67) node{$r_2$};
\draw (3.4, 0.67) node{$r_2 \zeta$};
\draw (0, 1) .. controls (1, 1) and (2, 0) .. (1, 0);
\draw (2, 1) .. controls (1, 1) and (0, 0) .. (1, 0);
\draw (2, 1) .. controls (3, 1) and (4, 0) .. (3, 0);
\draw (4, 1) .. controls (3, 1) and (2, 0) .. (3, 0);
\end{tikzpicture}
\hspace{0.25in}
\begin{tikzpicture}[scale=1.3]
\filldraw (0, 1) circle[radius=0.03];
\filldraw (6, 1) circle[radius=0.03];
\draw (0.7, 0.67) node{$r_1$};
\draw (1.4, 0.67) node{$r_1 \zeta$};
\draw (2.7, 0.67) node{$r_2$};
\draw (3.4, 0.67) node{$r_2 \zeta$};
\draw (4.7, 0.67) node{$r_3$};
\draw (5.4, 0.67) node{$r_3\zeta$};
\draw (0, 1.15) node{$p = 0$};
\draw (6, 1.15) node{$q = \infty$};
\draw (0, 1) .. controls (1, 1) and (2, 0) .. (1, 0);
\draw (2, 1) .. controls (1, 1) and (0, 0) .. (1, 0);
\draw (2, 1) .. controls (3, 1) and (4, 0) .. (3, 0);
\draw (4, 1) .. controls (3, 1) and (2, 0) .. (3, 0);
\draw (4, 1) .. controls (5, 1) and (6, 0) .. (5, 0);
\draw (6, 1) .. controls (5, 1) and (4, 0) .. (5, 0);
\end{tikzpicture}
\end{center}

On such curves (and their normalizations at any one node), $p - q = 0 - \infty$ has order exactly $k$;
the function $t^k$ gives the linear equivalence between $kp$ and $kq$.
Moreover, any involution of $\pp^1$ exchanging $0$ and $\infty$
has the form $t \mapsto c/t$. Such an involution exchanges $r_i$ and $r_i \zeta$ only if $c = r_i^2 \zeta$.
Therefore, if the $r_i^2 \zeta$ are distinct, the points $p$ and $q$ are not conjugate
under the hyperelliptic involution on the normalization of $R_3$ at any node.
The general curves over $R_h$ therefore satisfy the conditions of Proposition~\ref{prop:finflat}.
A priori, however, $R_h$ may not map to the desired component
$\bar{\H}^{\text{ch}, \circ}_{k, h, 2}$.

\begin{prop} \label{prop:rightcomponent}
The image of $R_h$ lies in the component $\bar{\H}^{\mathrm{ch}, \circ}_{k, h, 2}$.
\end{prop}

\begin{rem}
The authors conjecture $\bar{\H}_{k, h, 2}$ is irreducible (when the characteristic does not divide $k$),
which would immediately imply Proposition~\ref{prop:rightcomponent}.
Indeed, in characteristic zero, one can establish this using transcendental techniques.
Moreover, techniques developed by Fulton in \cite{fulton} show that
irreducibility in characteristic zero implies irreducibility when the characteristic is greater than $k$.
However, it seems difficult to extend this argument to small characteristics.
Instead, we give here a direct algebraic proof of Proposition~\ref{prop:rightcomponent},
which requires only our less strict hypothesis that the characteristic
does not divide $k$.
\end{rem}

\begin{proof}
By definition, $\bar{\H}^{\text{ch}, \circ}_{k, h, 2}$ is the component of $\bar{\H}^{\text{ch}}_{k, h, 2}$
containing the locus of chains of elliptic curves.
(Note that the locus of chains of elliptic curves is itself irreducible, as it is isomorphic to
a $g$-fold product 
$X_1(k) \times \cdots \times X_1(k)$ of the classical modular curve $X_1(k)$, which is irreducible
in characteristic not dividing $k$.)
Consider the following points of $\bar{\H}^{\text{ch}}_{k, h, 2}$:

\smallskip

\begin{minipage}{.42\textwidth}
\begin{center}
\begin{tikzpicture}[scale=1.3]
\filldraw (0, 1) circle[radius=0.04];
\filldraw (4, 1) circle[radius=0.04];
\draw (0.7, 0.67) node{$1$};
\draw (1.3, 0.67) node{$\zeta$};
\draw (2.7, 0.67) node{$1$};
\draw (3.3, 0.67) node{$\zeta$};
\draw (2.3, 1.3) node{$0$};
\draw (1.7, 1.3) node{$\infty$};
\draw (0, 1.15) node{$p = 0$};
\draw (4, 1.15) node{$q = \infty$};
\draw (0, 1) .. controls (1, 1) and (2, 0) .. (1, 0);
\draw (2.5, 1.5) .. controls (1, 1) and (0, 0) .. (1, 0);
\draw (1.5, 1.5) .. controls (3, 1) and (4, 0) .. (3, 0);
\draw (4, 1) .. controls (3, 1) and (2, 0) .. (3, 0);

\filldraw[color=violet] (.98, 0.705) circle[radius=0.04];
\filldraw[color=violet] (3.025, 0.705) circle[radius=0.04];

\filldraw[color=blue] (2, 1.307) circle[radius=0.04];

\node at  (4.7, 0.67) {or}; 
\end{tikzpicture}
\end{center}
\end{minipage}
\begin{minipage}{.5\textwidth}
\begin{center}
\begin{tikzpicture}[scale=1.3]
\filldraw (0, 1) circle[radius=0.04];
\filldraw (6, 1) circle[radius=0.04];
\draw (0.7, 0.67) node{$1$};
\draw (1.3, 0.67) node{$\zeta$};
\draw (2.7, 0.67) node{$1$};
\draw (3.3, 0.67) node{$\zeta$};
\draw (4.7, 0.67) node{$1$};
\draw (5.3, 0.67) node{$\zeta$};
\draw (4.3, 1.3) node{$0$};
\draw (2.3, 1.3) node{$0$};
\draw (3.7, 1.3) node{$\infty$};
\draw (1.7, 1.3) node{$\infty$};
\draw (0, 1.15) node{$p = 0$};
\draw (6, 1.15) node{$q = \infty$};
\draw (0, 1) .. controls (1, 1) and (2, 0) .. (1, 0);
\draw (2.5, 1.5) .. controls (1, 1) and (0, 0) .. (1, 0);
\draw (1.5, 1.5) .. controls (3, 1) and (4, 0) .. (3, 0);
\draw (4.5, 1.5) .. controls (3, 1) and (2, 0) .. (3, 0);
\draw (3.5, 1.5) .. controls (5, 1) and (6, 0) .. (5, 0);
\draw (6, 1) .. controls (5, 1) and (4, 0) .. (5, 0);

\filldraw[color=violet] (.98, 0.705) circle[radius=0.04];
\filldraw[color=violet] (3, 0.723) circle[radius=0.04];
\filldraw[color=violet] (5.025, 0.705) circle[radius=0.04];

\filldraw[color=blue] (2, 1.307) circle[radius=0.04];
\filldraw[color=blue] (4, 1.307) circle[radius=0.04];

\end{tikzpicture}
\end{center}
\end{minipage}

\smallskip

By smoothing the top (blue) nodes, these curves are visibly in the closure of the image of $R_h$.
Similarly, by smoothing the bottom (violet) nodes, these curves are visibly in
the closure of the locus of chains of elliptic curves.
Finally, by Lemma~\ref{lem:smooth},
these curves lie in a unique component.
\end{proof}

We finally compute explicitly the restriction of the covers $Y_F(n)$ and $Y_S$ to $R_2$ and $R_3$
respectively. In particular, we will see that these covers are generically \'etale of degree~$2$
and have a unique irreducible component dominating $R_2$ and $R_3$ respectively.
Therefore $Y_F(n)$ and $Y_S$ have a unique irreducible component dominating $\bar{\H}_{k,2,2}^{\text{ch}, \circ}$
and $\bar{\H}_{k,3,2}^{\text{ch}, \circ}$ respectively
(c.f.\ discussion after the proof of Lemma~\ref{lem:etale}).

\smallskip

For $Y_F(n)$, we have
$\O(x - y) \simeq \O(n(0 - \infty))$, so there is a function vanishing
along $n \cdot 0 + y$, with a pole along $n \cdot \infty + x$.
The only such function on the normalization is
$t^n(t - y) / (t - x)$; this function must therefore descend to the nodal curve, i.e.:
\[\frac{r_1^n(r_1 - y)}{r_1 - x} = \frac{(r_1\zeta)^n(r_1\zeta - y)}{r_1\zeta - x} \quad \text{and} \quad \frac{r_2^n(r_2 - y)}{r_2 - x} = \frac{(r_2\zeta)^n(r_2\zeta - y)}{r_2\zeta - x}.\]
The first of these equations is linear in $y$; we may thus solve for $y$
and substitute into the second equation.
Clearing denominators, we obtain
a quadratic equation for $x$, whose coefficients are symmetric in $r_1$ and $r_2$.
Written in terms of the elementary symmetric functions $e_1 = r_1 + r_2$ and $e_2 = r_1r_2$ on $\operatorname{Sym}^2 \pp^1$,
this equation is:
\begin{equation} \label{e-2}
(\zeta^{n + 1} - 1) \cdot x^2 + (\zeta - \zeta^{n + 1}) e_1 \cdot x + (\zeta^{n + 1} - \zeta^2)e_2 = 0.
\end{equation}

This is linear in $e_1$ and $e_2$, so can only be reducible if it has a root $x \in \pp^1$
which is constant (i.e.\ independent of $e_1$ and $e_2$).
But upon setting $e_2 = \infty$, this quadratic has a double root at $x = \infty$
(note that $n \not\equiv 1$ mod $k$, so $\zeta^{n + 1} - \zeta^2 \neq 0$).
Similarly, upon setting $e_1 = e_2 = 0$,
this quadratic has a double root at $x = 0$
(note that $n \not\equiv -1$ mod $k$, so $\zeta^{n + 1} - 1 \neq 0$).
Thus no such constant root exists, and \eqref{e-2} is irreducible as desired.

For $Y_S$, we have $\O(x + y) \simeq \omega(-p - q)$, so there is a section of the dualizing sheaf
vanishing at $x$, $y$, $0$, and $\infty$. When pulled back to the normalization, this gives
a meromorphic $1$-form with poles at the points lying above the nodes ($r_1, r_1\zeta, r_2, r_2\zeta, r_3, r_3\zeta$)
that vanishes at $x$, $y$, $0$, and $\infty$. The only such $1$-form is
\[\alpha = \frac{t(t - x)(t - y) \cdot dt}{(t - r_1)(t - r_1\zeta)(t - r_2)(t - r_2\zeta)(t - r_3)(t - r_3\zeta)}.\]
This $1$-form must therefore descend to a section of the dualizing sheaf on the nodal curve, i.e.:
\[\Res_{t = r_1} \alpha + \Res_{t = r_1\zeta} \alpha = \Res_{t = r_2} \alpha + \Res_{t = r_2\zeta} \alpha = \Res_{t = r_3} \alpha + \Res_{t = r_3\zeta} \alpha = 0.\]
Since the sum of all residues
($\Res_{t = r_1} \alpha + \Res_{t = r_1\zeta} \alpha + \Res_{t = r_2} \alpha + \Res_{t = r_2\zeta} \alpha + \Res_{t = r_3} \alpha + \Res_{t = r_3\zeta} \alpha$) automatically vanishes, this is really just two conditions:
\begin{multline*}
\frac{r_1(r_1 - x)(r_1 - y)}{(r_1 - r_1\zeta)(r_1 - r_2)(r_1 - r_2\zeta)(r_1 - r_3)(r_1 - r_3\zeta)} + \frac{(r_1 \zeta)(r_1\zeta - x)(r_1\zeta - y)}{(r_1\zeta - r_1)(r_1\zeta - r_2)(r_1\zeta - r_2\zeta)(r_1\zeta - r_3)(r_1\zeta - r_3\zeta)} \\
= \Res_{t = r_1} \alpha + \Res_{t = r_1\zeta} \alpha = 0.
\end{multline*}
\begin{multline*}
\frac{r_2(r_2 - x)(r_2 - y)}{(r_2 - r_2\zeta)(r_2 - r_1)(r_2 - r_1\zeta)(r_2 - r_3)(r_2 - r_3\zeta)} + \frac{(r_2 \zeta)(r_2\zeta - x)(r_2\zeta - y)}{(r_2\zeta - r_2)(r_2\zeta - r_1)(r_2\zeta - r_1\zeta)(r_2\zeta - r_3)(r_2\zeta - r_3\zeta)} \\
= \Res_{t = r_2} \alpha + \Res_{t = r_2\zeta} \alpha = 0.
\end{multline*}
The first of these equations is linear in $y$; we may thus solve for $y$
and substitute into the second equation.
Clearing denominators,
we obtain a quadratic equation for $x$, whose coefficients are symmetric in $r_1$, $r_2$, and $r_3$.
Written in terms of the elementary symmetric functions
$e_1 = r_1 + r_2 + r_3$ and $e_2 = r_1r_2 + r_2 r_3 + r_3 r_1$
and $e_3 = r_1 r_2 r_3$ on $\operatorname{Sym}^3 \pp^1$, this equation is:
\[(\zeta + 1) e_2 \cdot x^2 - [\zeta e_1 e_2 + (\zeta^2 + \zeta + 1) e_3] \cdot x + (\zeta^2 + \zeta) e_1 e_3 = 0.\]

To see this is irreducible, it suffices to check irreducibility after specializing $e_1 = 1$,
which yields the equation
\begin{equation} \label{e1-1-3}
(\zeta + 1) e_2 \cdot x^2 - [\zeta e_2 + (\zeta^2 + \zeta + 1) e_3] \cdot x + (\zeta^2 + \zeta) e_3 = 0.
\end{equation}
This is linear in $e_2$ and $e_3$, so can only be reducible if it has a root $x \in \pp^1$
which is constant.
But upon setting $e_2/e_3 = 0$, the roots are $x = \infty$ and $x = (\zeta^2 + \zeta) / (\zeta^2 + \zeta + 1) \neq 0$
(note that $\zeta^2 + \zeta = \zeta(\zeta + 1) \neq 0$ because $\zeta$ is a primitive $k$th root of unity with $k > 2$).
Similarly, upon setting $e_2/e_3 = \infty$, the roots are $x = 0$ and $x = \zeta / (\zeta + 1) \neq \infty$
(again $\zeta + 1 \neq 0$).
It thus remains to observe that $(\zeta^2 + \zeta) / (\zeta^2 + \zeta + 1) \neq \zeta / (\zeta + 1)$
because $\zeta \neq 0$.

\bibliographystyle{amsplain.bst}
\bibliography{spl}

\end{document}